\documentclass[a4paper, 10pt, final]{amsart}
\usepackage{nag}
\usepackage{setspace}
\usepackage{multicol}
\usepackage{multirow}
\usepackage[indentafter]{titlesec}
\usepackage{upgreek}
\usepackage[pdftex]{graphicx}
\usepackage{amsmath, amsfonts, amssymb, amsthm, float, bm}
\usepackage{scalerel}
\usepackage{mathtools}
\usepackage{mathrsfs}
\usepackage{boldline}
\usepackage{enumitem}
\usepackage[toc,page]{appendix}
\usepackage[a4paper, bottom=2cm]{geometry}
\usepackage{array}
\usepackage{xtab}
\usepackage{lmodern}
\usepackage{tcolorbox}
\usepackage{xcolor}
\usepackage{url}
\usepackage{tikz-cd}
\usepackage{bm}
\usepackage{tabularx}
\usepackage{ltablex} 
\usepackage{etoolbox}
\usepackage{tipa}
\usepackage[backend=biber, style=alphabetic]{biblatex}
\usepackage[colorlinks=true,linkcolor=black,citecolor=blue!80!black]{hyperref}
\usepackage[capitalize]{cleveref}
\usepackage{relsize}
\usepackage{blindtext}
\usepackage{rotating}
\usepackage{fancyvrb}
\usepackage[justification=centering]{caption}
\usepackage{fullpage}

\renewbibmacro{in:}{}
\DeclareFieldFormat[article]{title}{#1}

\titleformat{name=\section}{}{\thetitle.}{0.8em}{\centering\scshape}
\titleformat{name=\subsection}[runin]{}{\thetitle.}{0.5em}{\bfseries}[.]
\titleformat{name=\subsubsection}[runin]{}{\thetitle.}{0.5em}{\itshape}[.]
\titleformat{name=\paragraph,numberless}[runin]{}{}{0em}{}[.]
\titlespacing{\paragraph}{0em}{0em}{0.5em}
\titleformat{name=\subparagraph,numberless}[runin]{}{}{0em}{}[.]
\titlespacing{\subparagraph}{0em}{0em}{0.5em}

\theoremstyle{plain}
\newtheorem{theorem}{Theorem}[section]
\newtheorem{proposition}[theorem]{Proposition}
\newtheorem{lemma}[theorem]{Lemma}
\newtheorem{corollary}[theorem]{Corollary}

\newtheorem*{theorem*}{Theorem}

\newtheorem*{main}{Main Theorem}
\newtheorem*{thmgg}{Theorem GG}

\theoremstyle{definition}
\newtheorem{definition}[theorem]{Definition}
\newtheorem{notation}[theorem]{Notation}
\newtheorem{hypothesis}[theorem]{Hypothesis}

\theoremstyle{remark}
\newtheorem*{remark}{Remark}

\AtBeginEnvironment{theorem}{\setlist[enumerate,1]{label=(\roman*),font=\upshape}}
\AtBeginEnvironment{theorem*}{\setlist[enumerate,1]{label=(\roman*),font=\upshape}}
\AtBeginEnvironment{main}{\setlist[enumerate,1]{label=(\roman*),font=\upshape}}
\AtBeginEnvironment{example}{\setlist[enumerate,1]{label=(\roman*),font=\upshape}}
\AtBeginEnvironment{definition}{\setlist[enumerate,1]{label=(\roman*),font=\upshape}}
\AtBeginEnvironment{lemma}{\setlist[enumerate,1]{label=(\roman*),font=\upshape}}
\AtBeginEnvironment{proposition}{\setlist[enumerate,1]{label=(\roman*),font=\upshape}}
\AtBeginEnvironment{corollary}{\setlist[enumerate,1]{label=(\roman*),font=\upshape}}
\AtBeginEnvironment{theorem}{\setlist[enumerate,2]{label=(\alph*),font=\upshape}}
\AtBeginEnvironment{lemma}{\setlist[enumerate,2]{label=(\alph*),font=\upshape}}
\AtBeginEnvironment{proposition}{\setlist[enumerate,2]{label=(\alph*),font=\upshape}}

\renewcommand{\Gamma}{\varGamma}
\renewcommand{\epsilon}{\varepsilon}
\renewcommand{\bar}{\overline}
\renewcommand{\hat}{\widehat}
\renewcommand{\leq}{\leqslant}
\renewcommand{\geq}{\geqslant}

\newcommand{\normaleq}{\trianglelefteq}

\newcommand{\divides}{\bigm|}
\newcommand{\fs}{\mathcal{F}}

\newcommand{\N}{\mathbb{N}}

\newcommand{\C}{\mathbb{C}}
\newcommand{\SL}{\mathrm{SL}} 
 
\newcommand{\GF}{\mathrm{GF}} 
\newcommand{\syl}{\mathrm{Syl}}
\newcommand{\GL}{\mathrm{GL}}

\newcommand{\PSL}{\mathrm{PSL}}
\newcommand{\PSp}{\mathrm{PSp}}
\newcommand{\PSU}{\mathrm{PSU}}

\newcommand{\Sz}{\mathrm{Sz}}
\newcommand{\Ree}{\mathrm{Ree}}
\newcommand{\Sym}{\mathrm{Sym}} 
\newcommand{\Alt}{\mathrm{Alt}} 
\newcommand{\Dih}{\mathrm{Dih}}

\newcommand{\Aut}{\mathrm{Aut}}
\newcommand{\Out}{\mathrm{Out}}
\newcommand{\Inn}{\mathrm{Inn}}

\newcommand{\Comp}{\mathrm{Comp}}

\newcommand{\Hom}{\mathrm{Hom}}

\newcommand{\hyp}{\mathfrak{hyp}}

\def \wt {\widetilde}

\begin{document}

\title{Fusion Systems and Simple Groups With Class Two Sylow $p$-subgroups}
\author{Martin van Beek}
\thanks{The author is supported by the Heilbronn Institute for Mathematical Research. The author wishes to thank Prof. Justin Lynd for helpful conversations at the beginning of this project.}
\address{Alan Turing Building, University of Manchester, UK, M13 9PL}
\email{\href{mailto:martin.vanbeek@manchester.ac.uk}{martin.vanbeek@manchester.ac.uk}}

\keywords{Fusion systems; Finite Simple Groups}
\subjclass[2020]{20D20; 20D05; 20E42}

\begin{abstract}
We determine all reduced saturated fusion systems supported on a finite $p$-group of nilpotency class two. As a consequence, we obtain a new proof of Gilman \& Gorenstein's classification of finite simple groups with class two Sylow $2$-subgroups.
\end{abstract}

\maketitle

\section{Introduction}
In recent years, much consideration has been given to fusion systems with restricted or small structure. This generally takes one of two paths. The structure of the fusion system itself may be restricted (e.g. by limiting its essential subgroups or the structure of its local subsystems); or the structure of the underlying $p$-group of the fusion system may be limited. In this article, we will take the latter path.

Let $\fs$ be a saturated fusion system on a finite $p$-group $S$. By a classical result of Burnside, if $S$ is abelian then $S$ is \emph{normal} in $\fs$. That is, $\fs=N_{\fs}(S)$ and $\fs$ is isomorphic to the $p$-fusion category of a finite $p$-solvable group. With regards to the nilpotency of the underlying $p$-group, one next natural step is to analyze the saturated fusion systems supported on $p$-groups of nilpotency class two, which is the main remit of this article. We wish to indicate that this is not the only next natural step. For instance, fusion systems supported on $p$-groups of maximal nilpotency class have been investigated in \cite{ParkerMax}.

For some motivation for the fusion systems under analysis, one may consider the non-abelian $p$-groups of order $p^3$. Necessarily, these groups have nilpotency class two. If $\fs$ is a saturated fusion system on a non-abelian $p$-group of order $p^3$ then $\fs$ is known. The interesting outcomes here are in the work of Ruiz and Viruel who handled the case when $p$ is odd and the underlying $p$-group has exponent $p$ \cite{RV1+2}. That result provides a large number of examples of realizable systems associated to certain finite simple groups, and provides three examples of exotic fusion systems at the prime $7$. In fact, as our \hyperlink{MainThm}{Main Theorem} indicates, these are, in some sense, the only exotic fusion systems supported on a $p$-group of nilpotency class at most two.

\begin{main}\hypertarget{MainThm}{}
Suppose that $\fs$ is a saturated fusion system on a finite $p$-group $S$, where $S$ has nilpotency class two. Assume that $O_p(\fs)=\{1\}$. Then $O^{p'}(\fs)=\fs_1\times \dots \times\fs_n$ for some $n\in \N$ where each $\fs_i$ is isomorphic to either:
\begin{enumerate}
    \item a saturated fusion system on a $p$-group isomorphic to a Sylow $p$-subgroup of $\PSL_3(p^{a_i})$ for some $a_i\in\N$;
    \item $\fs_T(G)$ where $G\cong \PSp_4(2^{a_i})$ and $p=2$  for some $a_i\in\N$; or
    \item $\fs_T(G)$ where $G\cong (\Alt(6)\times\dots\times \Alt(6)):A$ where $A\le \Omega_1(Z(S))$ and $p=2$.
\end{enumerate}
\end{main}

The groups arising in (iii) are isomorphic to normal subgroups of $\Sym(6)\times \dots \times \Sym(6)$ of index a power of $2$. 

To illustrate the difference between finite groups and fusion systems, one can compare this paper with Gilman and Gorenstein's treatment of the finite simple groups which have a Sylow $2$-subgroup of nilpotency class two \cite{ggtwo}. We being by pointing out some parallels with their work. First, both works operate under a ``minimal counterexample" hypothesis and so have access to detailed information about substructures in the relevant groups/fusion systems. Next, both works partition the analysis into a \emph{component type} case and a \emph{characteristic $p$} case, albeit we use a broader definition of characteristic $p$ than the direct translation of Gilman and Gorenstein's choice for groups. In spite of this, the characteristic $p$ analysis for fusion systems is similar in spirit to the finite group analysis. 

In contrast to the work of Gilman and Gorenstein, we answer the question for every prime $p$, rather than solely focusing on the case for $p=2$, and consider the larger class of \emph{reduced} fusion systems in place of simple fusion systems. Moreover, our analysis in the component type case is quite different. Gilman and Gorenstein require delicate results concerning balance and signalizer functor theory, whereas our treatment primarily makes use of Aschbacher's balance argument for fusion systems \cite[Theorem 7]{aschfit}. Aschbacher's result is far easier than the analogous statement for finite groups. From the perspective of analyzing  finite groups locally, this is one of the key reasons for using fusion systems in the first place and is the mantra behind Aschbacher's programme to reprove large parts of the classification of the finite simple groups. Aschbacher's programme has made significant progress towards this goal, and while this work will not play a part in that programme (indeed, the fusion systems we considering end up being of characteristic $2$-type), we hope that the analysis within highlights the benefits of working with fusion systems to make statements about finite simple groups. See \cite{AschComp} for an overview of Aschbacher's programme. With this in mind, in the final section of this paper, we use our fusion systems result to reprove Gorenstein and Gilman's result concerning finite simple groups with Sylow $2$-subgroups of nilpotency class two.

\begin{thmgg}\hypertarget{thmgg}{}
Suppose that $G$ is a finite simple group such that a Sylow $2$-subgroup $S$ of $G$ has nilpotency class at most two. Then one of the following holds:
\begin{enumerate}
    \item $S$ is abelian and $G$ is isomorphic to $\mathrm{J}_1$, $\PSL_2(q)$ where $q\equiv 3,5 \mod 8$, $\Ree(3^{2n+1})$, or $\PSL_2(2^n)$ for $n>1$;
    \item $S\cong \Dih(8)$ and $G$ is isomorphic to $\PSL_3(2)$, $\Alt(7)$, or $\PSL_2(q)$ where $q\equiv 7,9 \mod 16$; or
    \item $G\cong \Sz(2^n)$, $\PSU_3(2^n)$, $\PSL_3(2^n)$, or $\PSp_4(2^n)$ for $n>1$.
\end{enumerate}
\end{thmgg}

Throughout, a group of nilpotency class two will be a non-abelian group. Our notation and terminology for groups is reasonably standard and generally follows that used in \cite{asch2} and \cite{gor}. For fusion systems, we follow the notation used in \cite{ako} and \cite{aschfit}.

\section{Group Theory Preliminaries}

We begin with some elementary results regarding finite $p$-groups. 

\begin{lemma}\label{NoWreath}
Suppose that $S$ is a finite $p$-group of nilpotency class two, $T\normaleq S$ and $T=T_1\times\dots \times T_n$ where the $T_i$ are pairwise isomorphic and of nilpotency class two. Assume that $s\in S\setminus T$ permutes the set $\{T_i\}$. Then $s$ normalizes each $T_i$.
\end{lemma}
\begin{proof}
Assume that $S$ is as described in the lemma, and choose $s\in S\setminus T$ such that $s$ permutes the set $\{T_i\}$. Aiming for a contradiction, suppose that there is $T_k$ with $T_k^s=T_j$ for $k\ne j$. Let $y\in T_k$ and choose $x\in T_k\setminus Z(T_k)$. Then $y^s\in T_j$ and so $[x, y^s]=1$. On the other hand, $y^{-1}y^s=[y, s]\in Z(S)$ since $S$ has nilpotency class two. Hence, $[x, y^{-1}]=1$ and since $y$ was chosen arbitrarily, we deduce that $x\in Z(T_k)$, a contradiction. Hence, no such $k$ exists and $s$ normalizes each $T_i$. 
\end{proof}

In the analysis of fusion systems on $2$-groups of class two, those with components which are isomorphic to the $2$-fusion category of $\Alt(6)$ cause considerably more difficulty than any other type of component. We provide the following lemma which helps in treating these cases.

\begin{lemma}\label{DihCent}
Suppose that $S$ is a $2$-group of nilpotency class $2$, $P\normaleq S$ and $P\cong \Dih(8)$. Then $S=PC_S(P)$.
\end{lemma}
\begin{proof}
We observe that $Z(S)\cap P=Z(P)$ and so, as $S$ has nilpotency class two, we deduce that $[S, P]=Z(P)$. Indeed, $[\Aut_S(P), \Inn(P)]=\{1\}$. Since $|\Inn(P)|=4$ and $\Aut(P)$ is non-abelian of order $8$, we deduce that $\Aut_S(P)=\Inn(P)$ and $S=PC_S(P)$.
\end{proof}

Finally, comparing with the \hyperlink{MainThm}{Main Theorem}, we see that the simple $2$-fusion systems supported on a $2$-group of nilpotency class two are either isomorphic to the $2$-fusion category of $\PSL_3(2^a)$, or to the $2$-fusion category of $\PSp_4(2^{a})$ where $a\in \N$. The underlying $2$-group in each case has exactly two elementary abelian subgroups of maximal order. We present a well known lemma which identifies such structure in $2$-groups. For this, and for a $p$-group $S$, we denote by $\mathcal{A}(S)$ the elementary abelian subgroups of $S$ of maximal rank.

\begin{lemma}\label{elementary2}
Let $S$ be a $2$-group and suppose that there is $A,B\in\mathcal{A}(S)$ with $AB=S$. If $C_B(a)=A\cap B$ for every $a\in A\setminus (A\cap B)$, then $\mathcal{A}(S)=\{A, B\}$ and every involution of $S$ lies in $A\cup B$.
\end{lemma}
\begin{proof}
Aiming for a contradiction, let $C\in\mathcal{A}(S)$ with $A\ne C\ne B$ and choose $c$ an involution in $S$. Then $c=ab$ for some $a\in A$ and $b\in B$ and $1=c^2=(ab)^2=[a, b]$. Hence, either $a\in A\cap B$ so that $c\in B$, or $a\not\in B$ and $b\in C_B(a)=A\cap B$ so that $c=ab\in A$. Thus, $c\in A\cup B$ and we now have that $C=(C\cap A)(C\cap B)$. 

Let $c\in C\cap A$ with $c\not\in B$. Then $c\in A\setminus (A\cap B)$ so that $C_B(c)=A\cap B$. But $C\cap B\le C_B(c)$ so that $C\cap B\le A$ and $C\le A$. Since $A,C\in\mathcal{A}(S)$, we deduce that $C=A$, a contradiction. Hence, no such $C$ exists and $\mathcal{A}(S)=\{A, B\}$.
\end{proof}

\section{Fusion Systems}

In this section, we collect various important concepts regarding fusion systems, all the time assuming a familiarity with fusion systems in line with the first part of \cite{ako}. We recall some of the most pivotal results which shall be used throughout this work.

We first recall the notion of an \emph{essential subgroup} of a saturated fusion system $\fs$, and denote the set of essential subgroups of $\fs$ by $\mathcal{E}(\fs)$. As with many other works, the following theorem underpins large parts of our methodology in classifying fusion systems.

\begin{theorem}[Alperin -- Goldschmidt Fusion Theorem]
Let $\fs$ be a saturated fusion system on a $p$-group $S$. Then \[\fs=\langle \Aut_{\fs}(Q) \mid Q\,\, \text{is essential or}\,\, Q=S \rangle.\]
\end{theorem}
\begin{proof}
See \cite[Theorem I.3.5]{ako}.
\end{proof}

\begin{lemma}\label{Chain}
Let $\fs$ be a saturated fusion system on a $p$-group $S$ and $E\in\fs^{frc}$. Set $\{1\}=E_0 \normaleq E_1  \normaleq E_2 \normaleq \dots \normaleq E_m=E$ such that, for all $0 \le i \le m$, $E_i\alpha = E_i$ for each $\alpha \in \Aut_{\fs}(E)$. If $Q\le S$ is such that $[Q, E_i]\le E_{i-1}$ for all $1 \le i \le m$ then $Q\le E$.
\end{lemma}
\begin{proof}
Let $\fs$, $S$, $E$ and $Q$ be as described in the lemma. Then $N_{QE}(E)$ embeds in $\Aut(E)$ as $\Aut_{QE}(E)$ and we can arrange that $N_{QE}(E)$, and hence $\Aut_{QE}(E)$, centralizes (an $\Aut_{\fs}(E)$-invariant refinement of) $E_0 \normaleq E_1  \normaleq E_2 \normaleq \dots \normaleq E_m$. Setting $A:=\langle \Aut_{QE}(E)^{\Aut_{\fs}(E)}\rangle$, we have by \cite[Corollary I.5.3.3]{gor} that $A$ is a $p$-group. Since $A\normaleq \Aut_{\fs}(E)$, we see that $A\le O_p(\Aut_{\fs}(E))=\Inn(E)$, as $E$ is $\fs$-radical. Hence, $\Aut_{QE}(E)\le \Inn(E)$ and since $E$ is $\fs$-centric, we infer that $N_{QE}(E)\le E$ and $N_{QE}(E)=E$. The only possibility is that $E=QE$ and $Q\le E$, as required.
\end{proof}

Again, looking to the typical examples of fusion systems supported on $p$-groups of class two, as in the \hyperlink{MainThm}{Main Theorem}, we notice that the local subsystems often involve the group $\SL_2(p^n)$ acting on its \emph{natural module}. In order to prove the \hyperlink{MainThm}{Main Theorem}, we will have to recognize this action.

\begin{definition}
\sloppy{A \emph{natural $\SL_2(p^n)$-module} is any irreducible $2$-dimensional $\GF(p^n)\SL_2(p^n)$-module regarded as a $2n$-dimension module for $\GF(p)\SL_2(p^n)$ by restriction.}
\end{definition}

\begin{theorem}\label{SEFF}
Suppose that $E$ is an essential subgroup of a saturated fusion system $\fs$ on a $p$-group $S$, and assume that there are $\Aut_{\fs}(E)$-invariant subgroups $U\le V\le E$ such that $E=C_S(V/U)$ and $V/U$ is an FF-module for $G:=\Out_{\fs}(E)$. Then, writing $L:=O^{p'}(G)$ and $W:=V/U$, we have that $L/C_L(W)\cong \SL_2(p^n)$, $C_L(W)$ is a $p'$-group and $W/C_W(O^p(L))$ is a natural $\SL_2(p^n)$-module.
\end{theorem}
\begin{proof}
This follows from \cite[Theorem 5.6]{henkesl2} (c.f. \cite[Theorem 1]{henkesl2}).
\end{proof}

The following definitions are standard in the study of fusion systems.

\begin{definition}
Let $\fs$ be a saturated fusion system on a $p$-group $S$ and $Q\le S$. 
\begin{enumerate}
    \item Say that $Q$ is strongly closed in $\fs$ if, for all $P\le Q$ and $\alpha\in\Hom_{\fs}(P, S)$ we have that $P\alpha\le Q$.
    \item Say that $Q$ is normal in $\fs$, written $Q\normaleq \fs$, if for all $P,R\le S$ and $\alpha\in \Hom_{\fs}(P, R)$ there is $\hat{\alpha}\in\Hom_{\fs}(PQ, RQ)$ such that $\hat{\alpha}|_P=\alpha$ and $Q\hat{\alpha}=Q$.
\end{enumerate}
We note that if $Q$ is normal in $\fs$, then $Q$ is strongly closed in $\fs$; and if $Q$ is strongly closed in $\fs$ then $Q\normaleq S$.

There is a unique largest normal subgroup of $\fs$, denoted by $O_p(\fs)$. We will sometimes refer to this group as the \emph{$p$-core} of $\fs$.
\end{definition}

\begin{proposition}\label[proposition]{normalinF}
Let $\fs$ be a saturated fusion system over a $p$-group $S$. Then the following are equivalent for a subgroup $Q\le S$:
\begin{enumerate}
    \item $Q\normaleq \fs$;
    \item $Q$ is strongly closed in $\fs$ and contained in every centric radical subgroup of $\fs$; and
    \item $Q$ is contained in each essential subgroup, $Q$ is $\Aut_{\fs}(E)$-invariant for any essential subgroup $E$ of $\fs$ and $Q$ is $\Aut_{\fs}(S)$-invariant.
\end{enumerate}
Moreover, if $Q$ is an abelian subgroup of $S$, then $Q\normaleq \fs$ if and only if $Q$ is strongly closed in $\fs$.
\end{proposition}
\begin{proof}
See \cite[Proposition I.4.5]{ako} and \cite[Corollary I.4.7]{ako}.
\end{proof}

\begin{definition}
Let $\fs$ be a saturated fusion system on $S$. Say that $\fs$ is \emph{constrained} if $C_S(O_p(\fs))\le O_p(\fs)$.
\end{definition}

We will make liberal use of the following theorem, often without reference.

\begin{theorem}[Model Theorem]
Let $\fs$ be a constrained fusion system on a $p$-group $S$ and set $Q:=O_p(\fs)$. Then the following hold:
\begin{enumerate} 
\item There is a model for $\fs$. That is, there is a finite group $G$ with $S\in\syl_p(G)$, $F^*(G)=O_p(G)$ and $\fs=\fs_S(G)$.
\item If $G_1$ and $G_2$ are two models for $\fs$, then there is an isomorphism $\phi: G_1\to G_2$ such that $\phi|_S = \mathrm{Id}_S$.
\item For any finite group $G$ with $S\in\syl_p(G)$, $F^*(G)=Q$ and $\Aut_G(Q)=\Aut_{\fs}(Q)$, there is $\beta\in\Aut(S)$ such that $\beta|_Q = \mathrm{Id}_Q$ and $\fs_S(G) =\fs^\beta$. Thus, there is a model for $\fs$ which is isomorphic to $G$.
\end{enumerate}
\end{theorem}
\begin{proof}
See \cite[Theorem I.4.9]{ako}.
\end{proof}

As with the proofs of many theorems in finite group theory, the proof of the \hyperlink{MainThm}{Main Theorem} comes through analysis of a minimal counterexample. For this inductive approach, we require a reasonable handle on subsystems and quotient systems of fusion systems. In \cite[Section II.5]{ako}, quotients are defined with respect to strongly closed subgroups of a fusion system. We will only ever require quotients with respect to normal subgroups, and so we specialize the definition.

\begin{definition}
Let $\fs$ be a saturated fusion system on a $p$-group $S$ and $Q\normaleq \fs$. For each $P,R\le S$ where $Q\le P\cap R$, and $\alpha\in\Hom_{\fs}(P, R)$ define $\alpha^+:P/Q\to R/Q$ by $(xQ)\alpha^+:=(x\alpha)Q$. 

Define $\fs/Q$ to be the category whose objects are subgroups of $S/Q$ such that for $Q\le P\cap R$ where $P,R\le S$, we have $\Hom_{\fs/Q}(P/Q, R/Q):=\{\alpha^+ \mid \alpha\in \Hom_{\fs}(P, R)\}$.
\end{definition}

\begin{proposition}
Let $\fs$ be a saturated fusion system on a $p$-group $S$ and $Q\normaleq \fs$. Then $\fs/Q$ is a saturated fusion system supported on $S/Q$.
\end{proposition}
\begin{proof}
See \cite[Lemma II.5.2(a), Lemma II.5.5]{ako}.
\end{proof}

Again in parallel with finite group theory, the notion of normality and simplicity in fusion systems is of key importance. More than this, we will heavily focus on the structure of certain \emph{subnormal} systems of fusion systems. 

\begin{definition}
Let $\fs$ be a saturated fusion system on $S$ and $\mathcal{E}$ a fusion subsystem of $\fs$ supported on $T\le S$. Then $\mathcal{E}$ is normal in $\fs$, written $\mathcal{E}\normaleq \fs$, if $T$ is strongly closed in $\fs$ and the following conditions are satisfied:
\begin{enumerate}
    \item $\mathcal{E}$ is saturated;
    \item $\mathcal{E}^\alpha= \mathcal{E}$ for each $\alpha\in\Aut_{\fs}(T)$;
    \item for each $P\le T$ and each $\phi\in\Hom_{\fs}(P, T)$ there are $\alpha\in \Aut_{\fs}(T)$ and $\phi_0\in \Hom_{\mathcal{E}}(P, T)$ such that $\phi=\phi_0\circ\alpha$; and
    \item each $\alpha\in \Aut_{\mathcal{E}}(T)$ extends to some $\alpha\in \Aut_{\fs}(TC_S(T))$ such that $[\alpha, C_S(T)]\le Z(T)$.
\end{enumerate}

Condition (ii) will be referred to as the invariance condition, while (iii) will be referred to as the Frattini condition. If a subsystem $\mathcal{E}\subseteq \fs$ satisfies (ii) and (iii) then it is said to be $\fs$-invariant. 

Define subnormality in fusion systems to be the transitive extension of normality in fusion systems, and write $\mathcal{E}\normaleq \normaleq \fs$ if $\mathcal{E}$ is a subnormal subsystem of $\fs$.
\end{definition}

We recall that for $Q\le S$, we have that $Q\normaleq \fs$ if and only if $\fs_Q(Q)\normaleq \fs$.

\begin{definition}
Let $\fs$ be a saturated fusion system on a $p$-group $S$ and let $\mathcal{H}$ be a saturated subsystem of $\fs$ supported on $T$. Say that
\begin{enumerate}
\item $\mathcal{H}$ has \emph{index prime to $p$} in $\fs$ if $S=T$ and $O^{p'}(\Aut_{\fs}(P))\le \Aut_{\mathcal{H}}(P)$ for all $P\le S$; and
\item $\mathcal{H}$ has \emph{$p$-power index} in $\fs$ if \[\hyp(\fs):=\langle g^{-1}(g\alpha) \mid g\in P\le S, \alpha\in O^p(\Aut_{\fs}(P))\rangle \le T\] and $O^p(\Aut_{\fs}(P))\le \Aut_{\mathcal{H}}(P)$ for all $P\le S$.
\end{enumerate}
\end{definition}

By \cite[Theorem I.7.7]{ako}, there is a unique smallest fusion subsystem of index prime to $p$ in $\fs$ which we denote by $O^{p'}(\fs)$. This fusion system is normal in $\fs$.

Similarly, as in \cite[Theorem I.7.4]{ako}, we denote the unique smallest fusion system of $p$-power index in $\fs$ by $O^p(\fs)$. We remark that $O^p(\fs)$ is normal in $\fs$ and supported on $\hyp(\fs)$.

\begin{lemma}\label{SNorm}
Suppose that $\fs$ is a saturated fusion system on a $p$-group $S$. If $\mathcal{E}\normaleq \fs$ is supported on $S$ then $\mathcal{E}$ has index prime to $p$ in $\fs$ and $O^{p'}(\fs)=O^{p'}(\mathcal{E})\normaleq \mathcal{E}$. 
\end{lemma}
\begin{proof}
Since $\mathcal{E}$ is supported on $S$, for all $P\le S$ we have that $\Aut_S(P)\le \Aut_{\mathcal{E}}(P)$. Since $\mathcal{E}$ is normal in $\fs$, it is $\fs$-invariant and so by \cite[Proposition I.6.4]{ako}, we have that $\Aut_{\mathcal{E}}(P)\normaleq \Aut_{\fs}(P)$. In particular, $O^{p'}(\Aut_{\mathcal{E}}(P))=O^{p'}(\Aut_{\fs}(P))\le \Aut_{\mathcal{E}}(P)$ and $\mathcal{E}$ has index prime to $p$ in $\fs$. 

Now, $O^{p'}(\Aut_{\fs}(P))=O^{p'}(\Aut_{\mathcal{E}}(P))\le \Aut_{O^{p'}(\mathcal{E})}(P)$ and we deduce that $O^{p'}(\mathcal{E})$ has index prime to $p$ in $\fs$. Applying \cite[Theorem I.7.7]{ako}, we conclude that $O^{p'}(\fs)\subseteq O^{p'}(\mathcal{E})$. On the other hand, we have that $O^{p'}(\Aut_{\mathcal{E}}(P))=O^{p'}(\Aut_{\fs}(P))\le \Aut_{O^{p'}(\fs)}(P)$ and $O^{p'}(\fs)$ has index prime to $p$ in $\mathcal{E}$. A further application of \cite[Theorem I.7.7]{ako} implies that $O^{p'}(\mathcal{E})\subseteq O^{p'}(\fs)$, and the lemma holds.
\end{proof}

\begin{proposition}\label{HenkeNormal}
For all $T\le S$ with $\hyp(\fs)\le T$, there is a unique normal subsystem $O^p(\fs)T$ of $\fs$ supported on $T$.
\end{proposition}
\begin{proof}
This follows from \cite[Theorem 5]{aschfit}, but see also \cite[Theorem 2]{HenkeNormal}.
\end{proof}

\begin{definition}
Let $\fs$ be a saturated fusion system. Then $\fs$ is \emph{simple} if $\fs$ has no proper non-trivial normal subsystems; and $\fs$ is \emph{reduced} if $O^{p'}(\fs)=O^p(\fs)=\fs$ and $O_p(\fs)=\{1\}$. Clearly, any saturated fusion system which is simple is also reduced. 
\end{definition}

Contrary to the situation for finite groups, one has to be careful in defining the intersection of two normal fusion subsystems if the result is to remain normal. We will not provide an exact description of the appropriate intersections here, but the existence of a suitable subsystem provided by the theorem below is enough for our purposes.

\begin{theorem}\label{wedge}
Let $\mathcal{E}_i$ be a normal subsystem of $\fs$ on a subgroup $T_i$ of $S$, for $i\in\{1,2\}$. Then there exists a normal subsystem $\mathcal{E}_1\wedge \mathcal{E}_2$ of $\fs$ on $T_1\cap T_2$ normal in $\mathcal{E}_1$ and $\mathcal{E}_2$. Moreover $\mathcal{E}_1 \wedge \mathcal{E}_2$ is the largest normal subsystem of $\fs$ on $T_1\cap T_2$ normal in $\mathcal{E}_1$ and $\mathcal{E}_2$.
\end{theorem}
\begin{proof}
This is \cite[Theorem 1]{aschfit}.
\end{proof}

\begin{definition}
Let $\fs$ be a saturated fusion system on a finite $p$-group $S$. Define \[\Aut(\fs)=\{\alpha\in \Aut(S) \mid \fs^\alpha=\fs\}\] and \[\Out(\fs)=\Aut(\fs)/\Aut_{\fs}(S).\] 

We will not define \emph{tameness} in this article, but refer to \cite{AOV} for details. For the purposes of this paper, we will understand $\fs$ to be \emph{tamely realized} by some finite group $G$ if, upon identifying $S\in \syl_p(G)$ we have that $\fs=\fs_S(G)$ and $\Out(\fs)\cong\Out(G)$. That is, automorphisms of $\fs$ may be understood by automorphisms of $G$. Often $G$ will be a finite simple group and we will require the results provided in \cite{GLS3} concerning automorphisms (especially fixed points of automorphisms).
\end{definition}

\begin{theorem}\label{krullschmidt}
Suppose that $\fs$ is a saturated fusion system on a $p$-group $S$. Then there exists indecomposable fusion subsystems $\mathcal{E}_1,\dots,\mathcal{E}_k\le \fs$ such that
\[\fs=\mathcal{E}_1\times \dots \times \mathcal{E}_k.\]
If $\fs=\mathcal{E}_1^*\times \dots \times \mathcal{E}_m^*$ is another such factorization then $k=m$ and there is an automorphism $\alpha\in\Aut(\fs)$ and a permutation $\sigma\in\Sym(k)$ such that $\mathcal{E}_i\alpha=\mathcal{E}_{i\sigma}^*$ for each $i$.

Moreover, if either $Z(\fs)=\{1\}$ or $\hyp(\fs)=S$, then $\fs$ factors as a product of indecomposable fusion subsystems in a unique way. Thus $\fs$ is the direct product of all of its indecomposable direct factors.
\end{theorem}
\begin{proof}
This comes from \cite[Theorem A]{oliverkrull} and the specialized statement is \cite[Corollary 5.3]{oliverkrull}.
\end{proof}

The more specialized statement is what we intend to apply at various points in this paper. In particular, we apply this technique to a set of components of certain local subsystems later in our analysis.

\begin{definition}
Let $\fs$ be a saturated fusion system on a $p$-group $S$. Then $\mathcal{C}$ is a \emph{component} of $\fs$ if $\mathcal{C}=O^p(\mathcal{C})$ is a subnormal saturated subsystem of $\fs$ such that $\mathcal{C}/Z(\mathcal{C})$ is simple. Write $\Comp(\fs)$ for the set of components of $\fs$ and $E(\fs)$ the normal subsystem of $\fs$ generated by $\Comp(\fs)$.
\end{definition}

In understanding quasisimple fusion systems, it is necessary to understand ``central extensions" of simple fusion systems. Fortunately, in our situation, every such simple subsystem we encounter will suitably contain the $p$-fusion category of some finite group. Thus, we are often able to apply the following result:

\begin{proposition}\label{SchurMultiplier}
Let $G$ be a perfect finite group such that $H^2(G,\C^\times)$ is a $p'$-group, let $P$ be a Sylow $p$-subgroup of $G$ and let $\fs$ be a saturated fusion system on $P$ such that $\fs_P(G) \le \fs$. If there is a central extension of finite $p$-groups
\[\{1\} \to Z \to \hat{P} \to P \to \{1\}\]
such that there exists a fusion system $\hat{\fs}$ on $\hat{P}$ with $Z\le Z(\hat{\fs})$ and $\hat{\fs}/Z\cong \fs$ then 
\[\hat{P}\cong Z \times P\] and \[\hat{\fs}\cong \fs_Z(Z)\times \fs.\]
\end{proposition}
\begin{proof}
This is \cite[Corollary 4.4]{FusSchur}.
\end{proof}

Finally, in parallel with groups, we can define the \emph{generalized Fitting subsystem} of a saturated fusion system.

\begin{definition}
Let $\fs$ be a saturated fusion system on a $p$-group $S$, and assume that $E(\fs)$ is supported on $T$. We define $F^*(\fs):=\langle E(\fs), \fs_{O_p(\fs)}(O_p(\fs))\rangle_{TO_p(\fs)}$ for the generalized Fitting subsystem of $\fs$ and write $F^*(\fs)=O_p(\fs)E(\fs)$. 
\end{definition}

We have by \cite[{}9.9]{aschfit} that $F^*(\fs)\normaleq \fs$ and that $F^*(\fs)$ is central product of $E(\fs)$ and $O_p(\fs)$ (so that $E(\fs)\subseteq C_{\fs}(O_p(\fs))$ and $O_p(\fs)\le C_S(E(\fs))$).

\begin{proposition}\label{conlayer}
Let $\fs$ be a saturated fusion system. Then $\fs$ is constrained if and only if $E(\fs)=\{1\}$. 
\end{proposition}
\begin{proof}
See \cite[{}14.2]{aschfit}.
\end{proof}

\begin{proposition}\label{p'genfit}
$F^*(\fs)=O^{p'}(F^*(\fs))$.
\end{proposition}
\begin{proof}
Let $T$ be Sylow in $E(\fs)$. We have that $E(\fs)\wedge O^{p'}(F^*(\fs))$ is a normal subsystem of $E(\fs)$ supported on $T$ by \cref{wedge}. If $E(\fs)\wedge O^{p'}(F^*(\fs))<E(\fs)$ then $O^{p'}(E(\fs))<E(\fs)$ by \cref{SNorm}. But now, for each $\mathcal{C}\in\Comp(\fs)$, we have that $O^{p'}(E(\fs))\wedge \mathcal{C}$ is a normal subsystem of $\mathcal{C}$ supported on a Sylow of $\mathcal{C}$. Since $\mathcal{C}$ is quasisimple, $\mathcal{C}=O^{p'}(\mathcal{C})$ and applying \cref{SNorm}, we conclude that $\mathcal{C}=O^{p'}(E(\fs))\wedge \mathcal{C} \subset O^{p'}(E(\fs))$ from which we deduce a contradiction. Hence, $E(\fs)\subseteq O^{p'}(F^*(\fs))$ and since $F^*(\fs)=\fs_{O_p(\fs)}(O_p(\fs))E(\fs)$ and $\fs_{O_p(\fs)}(O_p(\fs))\subseteq O^{p'}(\fs)$, we have that $F^*(\fs)\subseteq O^{p'}(F^*(\fs))$, as desired.
\end{proof}

After all of this one may wonder why it is better to prove results in the context of fusion systems as opposed to finite groups. A key reason is that certain \emph{balance} arguments are far easier in fusion systems than in finite groups. We exemplify this with the following theorem due to Aschbacher.

\begin{theorem}\label{balance}
Let $\fs$ be a saturated fusion system and $U\in\fs^f$. Then $E(N_{\fs}(U))\le E(\fs)$.
\end{theorem}
\begin{proof}
This is \cite[Theorem 7]{aschfit}.
\end{proof}

This result allows for an easy proof of certain \emph{dichotomy} theorems in fusion systems, in analogy to the Gorenstein--Walter dichotomy theorem in finite groups.

\begin{definition}
Let $\fs$ be a saturated fusion system on a $p$-group $S$. Say that
\begin{enumerate}
    \item $\fs$ is of parabolic characteristic $p$ if $N_{\fs}(Q)$ is constrained for all $Q\normaleq S$; and
    \item $\fs$ is of parabolic component type if there is $x\in \Omega_1(Z(S))$ such that $\Comp(C_{\fs}(x))\ne \emptyset$.
\end{enumerate}
\end{definition}

\begin{theorem}
Suppose that $\fs$ is a saturated fusion system on a $p$-group $S$. Then either $\fs$ is of parabolic component type or $\fs$ is of parabolic characteristic $p$.
\end{theorem}
\begin{proof}
The proof of this result is essentially the same as the proof of \cite[Theorem II.14.3]{ako}. Assume that $\fs$ is not of parabolic characteristic $p$ and let $\{1\}\ne U\normaleq S$ with $N_{\fs}(U)$ not constrained. Let $1\ne x\in O_p(N_{\fs}(U))\cap \Omega_1(Z(S))$ so that $E(N_{\fs}(U))\le C_{\fs}(x)$ by \cite[Theorem 6 (3)]{aschfit}. In particular, $E(N_{\fs}(U))\subseteq N_{C_{\fs}(x)}(U)$ and it follows that $E(N_{\fs}(U))\normaleq N_{C_{\fs}(x)}(U)$. Hence, $E(N_{\fs}(U))\subseteq E(N_{C_{\fs}(x)}(U))$. 

Then by \cref{balance}, we have that $E(N_{C_{\fs}(x)}(U))\subseteq E(C_{\fs}(x))$ and as $E(N_{\fs}(U))$ is non-trivial by \cref{conlayer}, we see that $E(C_{\fs}(x))$ is non-trivial. Hence, $C_{\fs}(x)$ has a component and $\fs$ is of parabolic component type.
\end{proof}

\begin{remark}
In the statement of the \hyperlink{MainThm}{Main Theorem}, $\fs$ is of parabolic characteristic $p$ if and only if $n=1$ and case (iii) does not occur.
\end{remark}

\section{Known Quasisimple Fusion Systems on Class Two Groups}

In this section, we characterize the ``known" quasisimple fusion systems supported on $p$-groups of nilpotency class two, and illuminate some of their structure and embeddings in larger fusion systems.

\begin{proposition}\label{KnownSL}
Suppose that $\fs$ is a saturated fusion system supported on a $p$-group $S$ such that $O_p(\fs)=\{1\}$. If $S$ is isomorphic to a Sylow $p$-subgroup of $\PSL_3(p^n)$ then $O^{p'}(\fs)$ is simple and one of the following holds:
\begin{enumerate}
    \item $O^{p'}(\fs)$ is isomorphic to the $p$-fusion category of $\PSL_3(p^n)$; or
    \item $S\cong p^{1+2}_+$, $p$ is odd and $\fs$ is classified in \cite{RV1+2}.
\end{enumerate}
\end{proposition}
\begin{proof}
If $n=1$ then we apply \cite[Example I.2.8]{ako} when $p=2$ and \cite{RV1+2} when $p$ is odd. If $n>1$ and $p$ is odd, then the determination of $O^{p'}(\fs)$ follows from \cite[Theorem 4.5.1]{Clelland} and its simplicity from \cite[Lemma II.13.3]{ako}. We provide a short proof when $p=2$ and $n>1$, although this is certainly well known.

Let $E\in\mathcal{E}(\fs)$. Then $Z(S)\le E$ and so, as $S$ has class two, we have that $E\normaleq S$. Let $A,B$ be the unique maximal elementary abelian subgroups of $S$, as described in \cref{elementary2}. If $\Omega_1(E)\cap A>Z(S)$ then $Z(E)\le A$. Indeed, it follows that $A$ centralizes the chain $\{1\}\normaleq Z(E)\normaleq E$ and so $A\le E$ by \cref{Chain}. If $A<E$ then $\Omega_1(E)\cap B>Z(S)$ and by a similar argument, we conclude that $B\le E$ so that $S=E$, a contradiction. Hence, if $\Omega_1(E)\cap A>Z(S)$ then $E=A$. Similarly, if $\Omega_1(E)\cap B>Z(S)$ then $E=B$. The only other possibility is that $\Omega_1(E)=Z(S)$. But then $S$ centralizes the chain $\{1\}\normaleq Z(S)\normaleq E$, another contradiction. Hence, $\mathcal{E}(\fs)\subseteq \{A, B\}$.

We observe that if $|\mathcal{E}(\fs)|\leq 1$ then $O_2(\fs)\ne\{1\}$ by \cref{normalinF}. Thus, as $O_2(\fs)=\{1\}$, we have that $\mathcal{E}(\fs)=\{A, B\}$. Then $\fs$ is determined in \cite[Corollary A]{vbbook} (although in this context, we only really require certain results from \cite{Greenbook}).
\end{proof}

Implicit in the above proposition is the observation that the $2$-fusion category of $\Alt(6)$ is isomorphic to the $2$-fusion category of $\PSL_3(2)$.

\begin{proposition}\label{KnownSp}
Suppose that $\fs$ is a saturated fusion system supported on a $2$-group $S$ such that $O_2(\fs)=\{1\}$. If $S$ is isomorphic to a Sylow $2$-subgroup of $\PSp_4(2^n)$ then one of the following holds:
\begin{enumerate}
    \item $n=1$, $\fs$ is isomorphic to the $2$-fusion system of $\PSp_4(2)\cong \Sym(6)$, and $O^2(\fs)$ is simple and isomorphic to the $2$-fusion category of $\Alt(6)$; or
    \item $n>1$, and $O^{2'}(\fs)$ is simple and isomorphic to the $2$-fusion category of $\PSp_4(2^n)$.
\end{enumerate}
\end{proposition}
\begin{proof}
This proof is similar to \cref{KnownSL}. Applying \cref{elementary2}, and using that $O_2(\fs)=\{1\}$ we conclude that $\mathcal{E}(\fs)=\mathcal{A}(S)=\{A, B\}$ and so $\fs=\fs_S(G)$ where $F^*(G)$ is known by \cite[Corollary A]{vbbook} (again this only really requires \cite{Greenbook}). Since the isomorphism type of $S$ is given, we deduce that $O^{2'}(\fs)$ is isomorphic to the $2$-fusion category of $\PSp_4(2^n)$. If $n>1$ then $O^{2'}(\fs)$ is simple by \cite[Lemma II.13.3]{ako}. If $n=1$ then \cite[Theorem A]{FSLie} reveals that $\fs=O^{2'}(\fs)$. Moreover, since $\PSp_4(2)\cong \Sym(6)$, we have that $O^2(\fs)$ is isomorphic to the $2$-fusion category of $\Alt(6)$, which is simple and isomorphic to the $2$-fusion category of $\PSL_3(2)$ by \cite[Example I.2.8]{ako}. 
\end{proof}

To simplify notation, we define $\mathcal{S}$ to be the set of $p$-groups which are isomorphic to a Sylow $p$-subgroup of $\PSL_3(p^n)$ or a Sylow $2$-subgroup of $\PSp_4(2^n)$ for some $n\in\N$, and $\mathcal{K}$ to be fusion systems in \cref{KnownSL} and \cref{KnownSp}.

We will use the following fact in later sections of this article.

\begin{lemma}\label{PSp4lem}
Let $G\cong \PSp_4(2^b)$ for some $b\geq 2$, and let $S\in\syl_2(G)$. Choose $A\in \mathcal{A}(S)$. For $R\le S$ with $|RA/A|\geq 4$, we have $|[R, A]|=2^{2b}$.
\end{lemma}
\begin{proof}
Choose $A$ and $R$ as in the lemma. We observe that $[R, A]\le S'$ and $S'$ has order $2^{2b}$. Aiming for a contradiction, assume that $|[R, A]|< 2^{2b}$. Write $H:=O^{2'}(N_G(A)/A)$ so that $H\cong \SL_2(2^b)$ and $A/C_A(H)$ is a natural $\SL_2(2^b)$-module for $H$. Let $1\ne y\in T\in\syl_2(H)\setminus \{S/A\}$. Then $|[y, A]|=2^b$ and $H=\langle RA/A, y\rangle$. But now, $[y, A][R, A]$ is normalized by $H$ and has order strictly less than $2^{3b}$. But one can check that in $\PSp_4(2^b)$, any normal $2$-subgroup of$N_G(A)$ which has order strictly less than $2^{3b}$ is centralized by $H$ (as a $\GF(2)H$-module, $A$ is an $\Omega_3(2^b)$-module and so is a non-split extension of modules). It follows that $[y, A]\le C_A(H)$ from which we deduce that $[H, A]=\{1\}$, a contradiction. Hence, $|[R, A]|=2^{2b}$ and the lemma holds. 
\end{proof}

\begin{proposition}\label{L34}
Suppose that $\fs$ is a quasisimple fusion system such that $S$ has class two and $S/Z(\fs)\in \mathcal{S}$. Then $\fs$ is simple or $p=2$ and $\fs$ is isomorphic to the $2$-fusion category of $2.\PSL_3(4)$ or $2^2.\PSL_3(4)$. In the latter case, $Z(\fs)$ is elementary abelian.
\end{proposition}
\begin{proof}
Suppose that $p=2$. Assume first that $S/Z(\fs)$ is isomorphic to a Sylow $2$-subgroup of $\PSp_4(2)\cong \Sym(6)$. Then $S/Z(\fs)\cong \Dih(8)\times 2$ and $\fs/Z(\fs)$ is not simple by \cref{KnownSp}. 

Hence, we continuing assuming that $\fs/Z(\fs)$ is a simple fusion system supported on a Sylow $2$-subgroup of $\PSL_3(2^a)$ for some $a\in \N$ or on a Sylow $2$-subgroup of $\PSp_4(2^b)$ for $b>1$. We have that $\fs/Z(\fs)$ is isomorphic to the $2$-fusion system of $\PSL_3(2^a)$ or $\PSp_4(2^b)$ in each case by \cref{KnownSL} and \cref{KnownSp}. Applying \cref{SchurMultiplier} and comparing with \cite[Section 6.1]{GLS3} we have that $\fs$ is simple unless $\fs/Z(\fs)$ is isomorphic to the $2$-fusion category of $\PSL_3(2)$ or $\PSL_3(4)$. 

Assume that $\fs/Z(\fs)$ is isomorphic to the $2$-fusion category of $\PSL_3(2)$ so that $S/Z(\fs)\cong \Dih(8)$. Assume that $|Z(\fs)|\ne \{1\}$ and let $Z<Z(\fs)$ with $|Z(\fs)/Z|=2$. Since $S/Z$ has nilpotency class two and $O^2(\fs)/Z=O^2(\fs/Z)$, by a transfer argument \cite[Corollary I.8.5]{ako} we have that $Z(S/Z)=[S/Z, S/Z]$ has order $4$. Perusing the Small Groups Library in MAGMA \cite{MAGMA}, we conclude that no group of order $2^4$ satisfies the properties required of $S/Z$ and we must have that $Z(\fs)=\{1\}$. Hence, $\fs$ is simple in this case.

Assume that $\fs/Z(\fs)$ is isomorphic to the $2$-fusion category of $\PSL_3(4)$. Then we have that $\mathcal{E}(\fs/Z(\fs))=\{A/Z(\fs), B/Z(\fs)\}=\mathcal{A}(S/Z(\fs))$. It follows quickly that $\mathcal{E}(\fs)=\{A, B\}$. Let $G$ be a model for $N_{\fs}(A)$ so that $C_G(A)\le A=O_2(G)$ and $G/A\cong \GL_2(4)$. Moreover, $Z(G)=Z(\fs)$. We have that $G=[G, G]Z(\fs)$ and as $O^2(\fs)=\fs$, we conclude that $G=[G, G]$. Hence, $G$ is a central quotient of the universal central extension of $G/Z(\fs)$, which we denote $\hat{G}$. Since the $2$-fusion category of $G/O_{2'}(G)$ is isomorphic to the $2$-fusion category of $G$, we need only consider the $2$-part of the central extension. 

Let $\hat{S}\in\syl_2(\hat{G})$ arranged such that $S$ is a quotient of $\hat{S}$, and let $Z\le Z(\fs)$ be such that $\hat{S}/Z=S$. Set $\hat{A}=O_2(\hat{G})$ so that $A=\hat{A}/Z$. We compute in MAGMA \cite{MAGMA} that for $S/Z$ to have nilpotency class two, we require $|Z(\fs)/Z|\leq 2^3$. Moreover, if $|Z(\fs)/Z|=2^3$ then $A/Z=\Omega_1(S/Z)$. Hence, if $|Z(\fs)|=2^3$ then $\Omega_1(B)=A\cap B=Z(S)$. But then $S$ centralizes the $\Aut_{\fs}(B)$-chain $\{1\}\normaleq Z(S)\normaleq B$, a contradiction by \cref{Chain}. By a similar argument, we observe that $S=\Omega_1(S)$ and we now search for $Z$ such that $|Z(\fs)/Z|\leq 4$ such that $S/Z=\Omega_1(S/Z)$. There is a unique group $Z$ with $|Z(\fs)/Z|=4$ and satisfying the required properties (moreover, any group $Y$ with $|Z(\fs)/Y|\leq 2$ satisfying the required properties properly contains $Z$). We finally verify that $S/Z$ is isomorphic to a Sylow $2$-subgroup of $2^2.\PSL_3(4)$. We leave the verification that $\fs$ is isomorphic to the $2$-fusion category of a central quotient of $2^2.\PSL_3(4)$ to the reader.

Suppose now that $p$ is odd and $S/Z(\fs)$ is isomorphic to a Sylow $p$-subgroup of $\PSL_3(p^n)$ for some $n\in\N$. If $n>1$, then $\fs/Z(\fs)$ is isomorphic to the $p$-fusion category of $\PSL_3(p^n)$ by \cref{KnownSL}. A consideration of Schur multipliers as provided in \cite[Section 6.1]{GLS3} combined with \cref{SchurMultiplier} yields that $\fs$ is simple. Hence, $S/Z(\fs)\cong p^{1+2}_+$ and $\fs/Z(\fs)$ is listed in \cite{RV1+2}. We again apply \cite[Section 6.1]{GLS3} and \cref{SchurMultiplier} unless $\fs/Z(\fs)$ is exotic to show that $\fs$ is simple.

We note that only two of the exotic systems given in \cite[Table 2]{RV1+2} are simple. One has two classes of essential subgroups: one of size $6$ and one of size $2$. We form a saturated subsystem of this system by pruning the class of size $2$ by \cite[Lemma 6.4]{Comp1}. This pruned system still has trivial $7$-core, and has a unique conjugacy class of essential subgroups of size $6$. Comparing again with \cite[Table 2]{RV1+2}, we ascertain that this subsystem is isomorphic to the $7$-fusion category of $\mathrm{Fi}_{24}$. Hence, applying \cite[Section 6.1]{GLS3} and \cref{SchurMultiplier} we have that $\fs$ is simple. The other simple exotic fusion system listed in \cite[Table 2]{RV1+2} has two classes of essentials: both of size $4$. We apply the same methodology as before to obtain a saturated subsystem with a unique class of essentials, which we identify with the $7$-fusion category of $\mathrm{O'N}$. Applying \cite[Section 6.1]{GLS3} and \cref{SchurMultiplier}, we deduce that $\fs$ simple.
\end{proof}

\begin{lemma}\label{Sym6}
Suppose that $\fs$ is a saturated fusion system on a $2$-group $S$ such that $S$ has nilpotency class two and $O_2(\fs)=\{1\}$. If $O^2(O^{2'}(\fs))$ is supported on a $2$-group isomorphic to $\Dih(8)$ then $\fs\cong \Alt(6)$ or $\Sym(6)$.
\end{lemma}
\begin{proof}
Since $O_2(\fs)=\{1\}$ and $O^2(O^{2'}(\fs))$ is a characteristic subsystem of $\fs$, we have $O_2(O^2(O^{2'}(\fs)))=\{1\}$. Hence, by \cite[Example I.2.8]{ako}, $O^2(O^{2'}(\fs))$ is isomorphic to the $2$-fusion category of $\Alt(6)$. Moreover, $\Alt(6)$ \emph{tamely realizes} the fusion system $O^2(O^{2'}(\fs))$ by \cite[Theorem I.C]{FSLie} and we deduce that that $S$ is isomorphic to a $2$-subgroup of $\Aut(\Alt(6))$. Since $S$ has nilpotency class two, one can verify that either $S$ is isomorphic to $\Dih(8)$ or to a Sylow $2$-subgroup of $\Sym(6)\cong \PSp_4(2)$, which itself is isomorphic to $\Dih(8)\times 2$. If $S\cong \Dih(8)$ then $\fs$ is isomorphic to the $2$-fusion category of $\Alt(6)$ by \cite[Example I.2.8]{ako}. If $|S|=2^4$, then $S$ is isomorphic to a Sylow $2$-subgroup of $\Sym(6)$ and $\fs$ is isomorphic to the $2$-fusion category of $\Sym(6)$ by \cref{KnownSp}.
\end{proof}

\begin{remark}
With regards to automorphisms, the $2$-fusion category of $\Alt(6)$ is far from the typical situation for simple fusion systems. See the comment after \cite[Definition I.1.3]{FSLie}.
\end{remark}

Before the next lemma, we explain what we mean when we say a component of $\fs$ is \emph{normalized} by $S$. Let $E(\fs)$ be supported on $T\le S$, and assume that $\mathcal{C}\in\Comp(\fs)$ with $\mathcal{C}$ supported on $P$. We have that $E(\fs)\normaleq \fs$ and so for each $\alpha\in \Aut_{\fs}(T)$, $E(\fs)^\alpha=E(\fs)$. In particular, since $T\normaleq S$, we have that $S$ acts on $T$ and for $\alpha\in \Aut_S(T)$ we have that $E(\fs)^\alpha=E(\fs)$. We say that $\mathcal{C}$ is normalized by $S$ if $P\normaleq S$ and for $\alpha\in \Aut_S(P)$, $\mathcal{C}^\alpha=\mathcal{C}$. 

\begin{lemma}\label{normcomp}
Suppose that $\fs$ is a saturated fusion system on a $p$-group $S$. Then each component of $\fs$ is normalized by $S$.
\end{lemma}
\begin{proof}
Let $\mathcal{K}\in\Comp(\fs)$ supported on $R$. Set $\bar{\fs}=\fs/Z(E(\fs))$. Then $\bar{\mathcal{K}}\in \Comp(\bar{\fs})$. Indeed, $E(\bar{\fs})=\bar{E(\fs)}$ is supported on $T=T_1\times \dots T_n$ where $T\normaleq \bar{S}$, $T_i$ is of nilpotency class $2$ and $\mathcal{C}_i\in\Comp(\bar{\fs})$ is the unique component supported on $T_i$ by \cite[{}9.7(1)]{aschfit}. We may arrange that $\bar{R}=T_1$ and $\bar{\mathcal{K}}=\mathcal{C}_1$. It follows that $\bar{S}$ acts on $\Comp(\bar{\fs})$ and so permutes the set $\{T_i\}_i$. But then, by \cref{NoWreath}, $\bar{S}$ normalizes each $T_i$ and so by \cite[{}9.7(2)]{aschfit}, normalizes each component of $\bar{\fs}$. Hence, $S$ normalizes $\langle \mathcal{K}, Z(E(\fs))\rangle_S$ and another application of \cite[{}9.7(2)]{aschfit} yields that $S$ normalizes $\mathcal{K}$. 
\end{proof}

\begin{proposition}\label{CompAutos}
Suppose that $\fs$ is a saturated fusion system on a $p$-group $S$ such that $S$ has class two. Assume there is $\mathcal{C}\in\Comp(\fs)$ supported on $T$ with $\mathcal{C}/Z(\mathcal{C})\in\mathcal{K}$. Then $S=TC_S(T)$ and unless $\mathcal{C}$ is isomorphic to the $2$-fusion system of $\Alt(6)$, we have that $S=C_S(\mathcal{C})T$.
\end{proposition}
\begin{proof}
Assume first that $\mathcal{C}/Z(\mathcal{C})$ is isomorphic to the $2$-fusion system of $\Alt(6)$. Then \cref{L34} reveals that $Z(\mathcal{C})=\{1\}$. Hence $T$ is isomorphic to $\Dih(8)$. By \cref{normcomp}, $S$ normalizes $T$ so that the hypothesis of \cref{DihCent} are satisfied and $S=TC_S(T)$. For the remainder of the proof we may assume that $\mathcal{C}/Z(\mathcal{C})$ is not isomorphic to the $2$-fusion category of $\Alt(6)$.

Assume that $p=2$ and $Z(\fs)=\{1\}$ so that $\mathcal{C}$ is isomorphic to the $2$-fusion category of $G$, where $G$ is $\PSL_3(2^a)$ or $\PSp_4(2^a)$ for some $a>1$. Applying \cite[Theorem I.A]{FSLie}, we have that $G$ tamely realizes $\mathcal{C}$. Now, the $2$-part of $\Out(G)$ comprises of graph, field and graph-field automorphisms, and so we may assume that for $s\in S\setminus C_S(\mathcal{C})T$, modulo the action of an inner automorphism, $s$ acts on $\mathcal{C}$ as a graph, field or graph-field automorphism acts on $G$. Suppose that $s$ acts as a non-trivial graph or graph-field automorphism. Then $s$ permutes the two maximal elementary abelian subgroups of $T$, impossible since $[s, T]\le Z(S)\cap T\le Z(T)=A\cap B$ where $\mathcal{A}(T)=\{A, B\}$. Suppose that $s=if$ where $i\in T$ and $f$ acts as a field automorphism. Since $S$ has class two, we have that $Z(T)=T'\le S'\le Z(S)$ and so $Z(T)$ is centralized by $s$. Since $i$ centralizes $Z(T)$ we have that $f$ also centralizes $Z(T)$. Therefore, $Z(T)\le Z(C_T(f))$ and applying \cite[Proposition 4.9.1(a)]{GLS3}, we have a contradiction. Hence, $S=C_S(\mathcal{C})T$.

Suppose that $Z(\mathcal{C})\ne\{1\}$ so that by \cref{L34}, $\mathcal{C}/Z(\mathcal{C})$ is isomorphic to the $2$-fusion system of $\PSL_3(4)$. We present the argument for the $2$-fusion category of $2^2.\PSL_3(4)$. The other case follows working in suitable quotients. Let $E$ be an essential subgroup of $\mathcal{C}$ so that $E$ is elementary abelian of order $2^6$ and $\Out_{\mathcal{C}}(E)\cong \PSL_2(4)$. Let $G$ be a model for $N_{\mathcal{C}}(E)$. Note that upon recognizing $S$ as a Sylow $2$-subgroup of $2^2.\PSL_3(4)$, $N_{2^2.\PSL_3(4)}(E)$ is a model for $N_{\mathcal{C}}(E)$ and by the model theorem, using the uniqueness of models, we have that $G\cong N_{2^2.\PSL_3(4)}(E)$. Then $G=[G, G]$. In particular, for any $\alpha\in \Aut(N_{\mathcal{C}}(E))$ with $\alpha$ acting trivially on $N_{\mathcal{C}}(E)/Z(\mathcal{C})$, we see that $[G, \alpha]\le Z(\mathcal{C})$ the three subgroups lemma yields that $[G, \alpha]=\{1\}$. It follows that for $s\in S\setminus T$, if $sZ(\mathcal{C})\in C_{\fs}(\mathcal{C}/Z(\mathcal{C}))$ then $N_{\mathcal{C}}(E)\subseteq C_{\fs}(s)$. For $\hat{E}\in \mathcal{E}(\mathcal{C})\setminus \{E\}$, by similar argument we have that $N_{\mathcal{C}}(\hat{E})\subseteq C_{\fs}(s)$. Moreover, $\mathcal{C}=\langle N_{\mathcal{C}}(E), N_{\mathcal{C}}(\hat{E})\rangle_T$ from which we conclude that $\mathcal{C}\subseteq C_{\fs}(s)$ and $s\in C_S(\mathcal{C})$. Hence, if $TC_S(\mathcal{C})<S$ then there is $s\in S\setminus TC_S(\mathcal{C})$ with $sZ(\mathcal{C})\not\in C_S(\mathcal{C}/Z(\mathcal{C}))$. Since $\mathcal{C}/Z(\mathcal{C})$ is isomorphic to the $2$-fusion system of $\PSL_3(4)$, we obtain a contradiction as in the previous paragraph. Hence, $S=TC_S(\mathcal{C})$.

Assume now that $p$ is odd, so that $Z(\mathcal{C})=\{1\}$ by \cref{L34} and $\mathcal{C}$ is simple. Suppose that $\mathcal{C}$ is isomorphic to the $p$-fusion category of $G=\PSL_3(p^a)$ for some $a>1$. As in the $p=2$ case, we apply \cite[Theorem A]{FSLie} and observe that $s$ acts on $\fs$ as a field automorphism, modulo an inner automorphism, acts on $G$. As in the $p=2$ case, we write $s=if$ and observe that $Z(T)\le Z(C_T(f))$ and \cite[Proposition 4.9.1(a)]{GLS3} gives a contradiction. Suppose now that $T\cong p^{1+2}_+$. We apply \cite{RV1+2} for a list of simple fusion systems, and apply \cite[Theorem II. A]{FSLie} and \cite[Table II.0.3]{FSLie} to see that if $\fs$ is not exotic and not isomorphic to the $3$-fusion category of ${}^2\mathrm{F}_4(2)'$, then for $G$ which tamely realizes $\fs$, $\Out(G)$ has order coprime to $p$ from which we conclude that $S=TC_S(\mathcal{C})$. If $\fs$ is isomorphic to the $3$-fusion category of ${}^2\mathrm{F}_4(2)'$ then \cite[Theorem I.D]{FSLie} implies that $\fs$ is tamely realized by ${}^2\mathrm{F}_4(2)'$ and as $|\Out({}^2\mathrm{F}_4(2)')|$ is coprime to $3$, again we see that $S=TC_S(\mathcal{C})$. 

Thus, we are reduced to the case where $p=7$, $T\cong 7^{1+2}_+$ and $\fs$ is a simple exotic fusion system. We observe that a Sylow $7$-subgroup of $\Aut(T)$ is isomorphic to $7^{1+2}_+$. Assume that $TC_S(\mathcal{C})<S$ and set $\bar{S}:=S/C_S(\mathcal{C})$. Choose $TC_S(\mathcal{C})<P\le S$ such that $|\bar{P}|=7^4$. Observe that as $T\cong 7^{1+2}_+$, $S=TC_S(T)$ and $S$ normalizes each essential subgroup of $\mathcal{C}$. It follows that $P$ normalizes $N_{\mathcal{C}}(E)$ for each $E\in\mathcal{E}(\mathcal{C})$. We appeal to \cite{RV1+2} to see that for $G$ a model of $N_{\mathcal{C}}(E)$, we have that $G\cong 7^2:\SL_2(7):2$ or $7^2:\GL_2(7)$. We calculate in MAGMA \cite{MAGMA} that $Z(G)=\{1\}$ and $\Aut(G)\cong 7^2: \GL_2(7)$ and so we deduce that $P=TC_P(G)$. As $\mathcal{C}$ is simple, we have that $\mathcal{C}=\langle N_{\mathcal{C}}(E) \mid E\in\mathcal{E}(\mathcal{C})\rangle$ and we deduce that that $P=TC_P(\mathcal{C})$, a contradiction. Hence, $S=TC_S(\mathcal{C})$.
\end{proof}

\begin{corollary}\label{SFitting}
Suppose that $\fs$ is a saturated fusion system on a $p$-group $S$ such that $S$ has class two. Assume that $E(\fs)\ne\{1\}$, let $\mathcal{X}\subset \Comp(\fs)$ and $\mathcal{H}$ be the central product of the components in $\mathcal{X}$. Let $T_{\mathcal{X}}$ be Sylow in $\mathcal{H}$. Then $S=T_{\mathcal{X}}C_S(T_{\mathcal{X}})$. 

If $S=C_S(\mathcal{C})T$ for all components $\mathcal{C}$ in $\mathcal{X}$, then $S=C_S(\mathcal{H})T_{\mathcal{X}}$. In particular, this holds if $\mathcal{X}$ has no components isomorphic to the $2$-fusion system of $\Alt(6)$.
\end{corollary}
\begin{proof}
Throughout, we set $T\le S$ such that $E(\fs)$ is supported on $T$. By \cref{normcomp}, $S$ normalizes each component of $\fs$ and by \cref{CompAutos} we have that $S=\bigcap_{\mathcal{C}\in\Comp(\fs)} C_S(\mathcal{C})T_{\mathcal{C}}$, where $\mathcal{C}\in\Comp(\fs)$ is supported on $T_{\mathcal{C}}$.

We claim that for a collection $\mathcal{X}$ of components of $\fs$, if $S=TC_S(\mathcal{C})$ for each $\mathcal{C}\in\mathcal{X}$ with $\mathcal{C}$ supported on $T$, then we have that \[S=\left(\bigcap_{\mathcal{C}\in\mathcal{X}} C_S(\mathcal{C})\right)\left(\prod_{\mathcal{C}\in\mathcal{X}} T_{\mathcal{C}}\right).\] Observing that $C_S(\mathcal{C})\le C_S(T)$, it is then easy to recreate the proof to show that $S=T_{\mathcal{X}}C_S(T_{\mathcal{X}})$.

Let $\mathcal{X}=\{\mathcal{C}_1,\dots, \mathcal{C}_n\}$. The case where $n=1$ follows holds by \cref{CompAutos} and so by induction on $n$, we have that \[S=\left(\bigcap\limits_{\mathcal{C}\in\mathcal{X}\setminus \{\mathcal{C}_n\}} C_S(\mathcal{C})\right)\left(\prod_{\mathcal{C}\in\mathcal{X}\setminus \{\mathcal{C}_n\}} T_{\mathcal{C}}\right)\cap T_{\mathcal{C}_n}C_S(\mathcal{C}_n).\] Now, $T_{\mathcal{C}_n}\le C_S(\mathcal{C}_i)$ for $i\in \{1,\dots, n-1\}$ by \cite[{}9.2, 9.6]{aschfit} and so we have that \[T_{\mathcal{C}_n}\le \bigcap\limits_{\mathcal{C}\in\mathcal{X}\setminus \{\mathcal{C}_n\}} C_S(\mathcal{C}).\] By the Dedekind modular law, \[S=T_{\mathcal{C}_n}\left(\left(\bigcap_{\mathcal{C}\in\mathcal{X}\setminus \{\mathcal{C}_n\}} C_S(\mathcal{C})\right)\left(\prod_{\mathcal{C}\in\mathcal{X}\setminus \{\mathcal{C}_n\}} T_{\mathcal{C}}\right)\cap C_S(\mathcal{C}_n)\right).\] Again by \cite[{}9.2, 9.6]{aschfit}, for all $i\in\{1,\dots, n-1\}$, $T_{\mathcal{C}_i}\le C_S(\mathcal{C}_n)$ so that \[\prod_{\mathcal{C}\in\mathcal{X}\setminus \{\mathcal{C}_n\}} T_{\mathcal{C}}\le C_S(\mathcal{C}_n).\] Another application of the Dedekind modular gives
\begin{align*}
    S&=T_{\mathcal{C}_n}\left(\prod_{\mathcal{C}\in\mathcal{X}\setminus \{\mathcal{C}_n\}} T_{\mathcal{C}}\right)\left(\left(\bigcap_{\mathcal{C}\in\mathcal{X}\setminus \{\mathcal{C}_n\}} C_S(\mathcal{C})\right)\cap C_S(\mathcal{C}_n)\right)\\
    &=\left(\bigcap_{\mathcal{C}\in\mathcal{X}} C_S(\mathcal{C})\right)\left(\prod_{\mathcal{C}\in\mathcal{X}} T_{\mathcal{C}}\right)
\end{align*}
and the claim holds.
\end{proof}

\section{A Minimal Counterexample to the Main Theorem}

With the aim of proving the \hyperlink{MainThm}{Main Theorem}, throughout the remainder of the fusion system portion of this work we operate under the following hypothesis.

\begin{hypothesis}\label{mainhyp}
$\fs$ is a saturated fusion system on a $p$-group $S$, where $S$ has nilpotency class two, such that $O_p(\fs)=\{1\}$ and $\fs$ is a minimal counterexample to the \hyperlink{MainThm}{Main Theorem} chosen first with respect to $|S|$ then with respect to the number of morphisms in $\fs$. 
\end{hypothesis}

We show that no such system exists in Sections 6 \& 7.

\begin{lemma}\label{GeneralEssentialStructure}
For $E\in\mathcal{E}(\fs)$ we have that $E\normaleq S$, $[O^{p'}(\Aut_{\fs}(E)), E]\le Z(E)$ and $E=C_S(Z(E))$. 
\end{lemma}
\begin{proof}
Let $E\in\mathcal{E}(\fs)$. Since $E$ is $\fs$-centric and $S$ has class two we have that $[S, E]\le S'\le Z(S)\le Z(E)$. In particular, $E\normaleq S$ and since $O^{p'}(\Aut_{\fs}(E))=\langle \Aut_S(E)^{\Aut_{\fs}(E)}\rangle$ normalizes $Z(E)$, we have that $[O^{p'}(\Aut_{\fs}(E)), E]\le Z(E)$. Finally, if $E<C_S(Z(E))$ then $C_S(Z(E))$ centralizes the $\Aut_{\fs}(E)$-invariant chain $\{1\}\normaleq Z(E)\normaleq E$, against \cref{Chain} since $E$ is $\fs$-essential. Hence, $E=C_S(Z(E))$.
\end{proof}

\begin{lemma}\label{EssentialGeneration}
We have that $\fs=\langle O^{p'}(\Aut_{\fs}(E)), \Aut_{\fs}(S) \mid E\in\mathcal{E}(\fs)\rangle_S$.
\end{lemma}
\begin{proof}
By the Frattini argument, $\Aut_{\fs}(E)=O^{p'}(\Aut_{\fs}(E))N_{\Aut_{\fs}(E)}(\Aut_S(E))$. Since $E$ is receptive and $E\normaleq S$, for all $\alpha\in N_{\Aut_{\fs}(E)}(\Aut_S(E))$ we have that $\alpha$ extends to some $\hat{\alpha}\in\Aut_{\fs}(S)$ with $\alpha=\hat{\alpha}|_E$. Then the result holds by the Alperin--Goldschmidt theorem.
\end{proof}

\begin{proposition}\label{MinSimple}
$\fs$ is simple.
\end{proposition}
\begin{proof}
Since $O^{p'}(\fs)=O^{p'}(O^{p'}(\fs))$ and $\fs$ is a minimal counterexample we have that $\fs=O^{p'}(\fs)$. Since $O_p(\fs)=\{1\}$, we have that $F^*(\fs)=E(\fs)$ is a saturated normal fusion subsystem of $\fs$. We observe by \cref{SFitting} that $F^*(\fs)$ is supported on $S$ unless $p=2$ and there is $\mathcal{C}\in\Comp(\fs)$ isomorphic to the $2$-fusion system of $\Alt(6)$. 

Suppose that $F^*(\fs)$ is supported on $T<S$. Let $\mathfrak{m}\subseteq \Comp(\fs)$ be the collection of components not isomorphic to the $2$-fusion system of $\Alt(6)$ and assume that $\mathfrak{m}$ is non-empty. Since $T<S$, we deduce that $\mathfrak{m}\ne \Comp(\fs)$. Then $\mathcal{M}:=\langle \mathfrak{m}\rangle\normaleq \fs$ and by \cite[Theorem 4]{aschfit}, $C_{\fs}(\mathcal{M})\normaleq M$. Applying \cref{SFitting}, we deduce that $S$ is Sylow in $\mathcal{M}C_{\fs}(\mathcal{M})$. Let $P\normaleq S$ with $\mathcal{M}$ supported on $P$. Since $O_2(\mathcal{M})\normaleq \fs$ we have that $P\cap C_S(\mathcal{M})=Z(\mathcal{M})\le O_2(\mathcal{M})=\{1\}$ and so $S=P\times C_S(\mathcal{M})$. Following \cite[Proposition I.6.7]{ako}, and using that $\fs=O^{2'}(\fs)$, we deduce that $\fs=\mathcal{M}\times C_{\fs}(\mathcal{M})$. Even more, $\fs=O^{2'}(\mathcal{M})\times O^{2'}(C_{\fs}(\mathcal{M}))$. By minimality, both $O^{2'}(\mathcal{M})$ and $O^{2'}(C_{\fs}(\mathcal{M}))$ satisfy the conclusion of the \hyperlink{MainThm}{Main Theorem} and so $\fs$ does, a contradiction. Hence, 

\textbf{(1)} if $F^*(\fs)$ is supported on $T<S$ then every component of $\fs$ is isomorphic to the $2$-fusion system of $\Alt(6)$.\hfill

Assume that $F^*(\fs)S\ne \fs$. If $O_2(F^*(\fs)S)\ne \{1\}$ then applying \cite[{}9.11]{aschfit}, we have \[\{1\}\ne O_2(F^*(\fs)S)\cap Z(S)\le C_S(F^*(\fs))\le Z(F^*(\fs))\le O_2(F^*(\fs))\le O_2(\fs),\] a contradiction. Thus, $F^*(\fs)S$ is classified by the \hyperlink{MainThm}{Main Theorem}. In particular, we have that $Z(S)$ is elementary and $S=Z(S)T$. Let $E\in\mathcal{E}(\fs)$ so that $E\normaleq S$. Indeed, $S=ET$ and for any $t\in T\setminus (T\cap E)$ we have that $|[S, t]|=2$. Hence, by \cref{SEFF}, we have that $O^{2'}(\Out_{\fs}(E))\cong \Sym(3)$. It follows that $O^{2'}(\Aut_{\fs}(E))\subset F^*(\fs)S$. We verify that $F^*(\fs)S\normaleq \fs$. 

First, $S$ is strongly closed in $S$, $F^*(\fs)S$ is saturated and $F^*(\fs)S$ is clearly invariant under $\alpha\in\Aut_{\fs}(S)$ and so conditions (i) and (ii) of normality hold. Condition (iv) is trivial to check.

Let $\theta\in \Hom_{\fs}(P, S)$ with the property that $\theta\ne \phi \circ \alpha$ for some $\alpha\in\Aut_{\fs}(S)$ and $\phi\in \Hom_{F^*(\fs)S}(P, S)$. Further, choose $\theta$ with property that, upon decomposing $\theta$ in the fashion of \cref{EssentialGeneration}, the decomposition of $\theta$ has minimal length. Write $\theta=\theta_1\circ \dots \circ \theta_l$ where $\theta_i$ is a restriction of an element of $\Aut_{\fs}(S)$ or an element of $O^{2'}(\Aut_{\fs}(E))$ for some $E\in\mathcal{E}(\fs)$. Suppose that $l>1$. By minimality, $\theta_1=\phi^*\circ \alpha^*$ for some $\alpha^*\in\Aut_{\fs}(S)$ and $\phi^*\in F^*(\fs)S$. Then, again by minimality, $\alpha^*\circ \theta_2=\hat{\phi}\circ \hat{\alpha}$ for some $\hat{\alpha}\in\Aut_{\fs}(S)$ and $\hat{\phi}\in F^*(\fs)S$. Continuing in this manner yields a contradiction. Hence, we are reduced to the case where $l=1$. We may assume that $\theta\in O^{2'}(\Aut_{\fs}(E))$ for some $E\in\mathcal{E}(\fs)$. By \cite[Proposition I.6.4]{ako}, we have that $\Aut_S(E)\le \Aut_{F^*(\fs)S}(E)\normaleq \Aut_{\fs}(E)$ from which we conclude that $O^{2'}(\Aut_{\fs}(E))\le \Aut_{F^*(\fs)S}(E)$. By a Frattini argument, $\theta=\phi\circ \alpha^*$ where $\phi\in \Aut_{F^*(\fs)S}(E)$ and $\alpha^*\in N_{\Aut_{\fs}(E)}(\Aut_S(E))$. But then $\alpha^*$ lifts to $\alpha\in N_{\Aut_{\fs}(S)}(E)$ and $\theta$ decomposes as $\phi \circ \alpha$ in the required manner. Hence, no $\theta$ exists and condition (iii) of normality for $F^*(\fs)S$ holds.

Hence, $F^*(\fs)S\normaleq \fs$ and as $\fs=O^{2'}(\fs)$, by \cref{SNorm}, we have a contradiction. Therefore,

\textbf{(2)} if $F^*(\fs)$ is supported on $T<S$ then $\fs^*(\fs)=O^2(\fs)$ and $\fs=F^*(\fs)S$.\hfill

Assume that $F^*(\fs)\subsetneq \fs=F^*(\fs)S$ so that $F^*(\fs)=O^2(\fs)$. By \cref{normcomp}, we have that $S$ normalizes each component of $\fs$ and applying \cref{CompAutos} we have that $C_S(\mathcal{C})T$ has index at most $2$ in $S$ for all $\mathcal{C}\in\Comp(\fs)$. Following the proof of \cref{SFitting}, we see the $S/TC_S(F^*(\fs))$ is elementary abelian and since $C_S(F^*(\fs))\le O_2(\fs)=\{1\}$, we have that $S/T$ is elementary abelian. Furthermore, we observe that $C_S(\mathcal{C})$ has index at most $4$ in $C_S(P)$ where $\mathcal{C}\in\Comp(\fs)$ is supported on $P$, and that $\Phi(C_S(P))\le C_S(\mathcal{C})$. As in \cref{SFitting}, we ascertain that $\Phi(C_S(T))\le C_S(F^*(\fs))\le O_2(\fs)=\{1\}$ and so $C_S(T)$ is elementary abelian. Finally, \cref{SFitting} yields that $S=TC_S(T)$ and so $C_S(T)\le \Omega_1(Z(S))$. 

We observe that the $2$-fusion category of $\Alt(6)$ is tamely realized by $\Alt(6)$ and applying \cite[Theorem C]{ako}, it follows that $\fs$ is tamely realized by a finite group $G$ such that $F^*(G)\cong (\Alt(6)\times \dots \times \Alt(6))$. Indeed, $\fs$ is realized by $F^*(G)S$. Since $S=TC_S(T)$, $S$ normalizes each component of $G$ and since the only subgroups of $\Aut(\Alt(6))$ which contain $\Inn(\Alt(6))$ and have class two Sylow $2$-subgroups are the unique subgroups isomorphic to $\Alt(6)$ or $\Sym(6)$,  it follows that $G$ is isomorphic to a subgroup of $\Sym(6)\times \dots \Sym(6)$. Therefore, $G$ satisfies conclusion (iii) of the \hyperlink{MainThm}{Main Theorem}, a contradiction since $\fs$ is a counterexample.

Hence, we have shown that 

\textbf{(3)} $F^*(\fs)$ is supported on $S$.\hfill

Applying \cref{SNorm}, using that $\fs=O^{p'}(\fs)$ we have that $\fs=F^*(\fs)=E(\fs)$. Now, we may write $\fs=\fs_1\times \dots \times\fs_n$ for some $n\in\N$, where $\fs_i$ are the components of $\fs$. Hence $O^{p'}(\fs_i)=\fs_i$, $O_p(\fs_i)=\{1\}$ and $\fs_i$ is supported on a $p$-group of class two for all $i$. In particular, each $\fs_i$ is simple. If $n>1$ then by induction, each $\fs_i$ is isomorphic to a simple fusion system on a $p$-group $S_i$ with $S_i\in\mathcal{S}$ and so $\fs$ itself satisfies the outcome of the \hyperlink{MainThm}{Main Theorem}, a contradiction. Hence, $n=1$ and $\fs$ is simple.
\end{proof}

\section{The Parabolic Component Type Analysis}\label{Component Section}

We fix the following throughout this section. 

\begin{notation}
For $x\in \Omega(Z(S))$ we write:
\begin{itemize}
    \item $\fs_x:=N_{\fs}(C_S(E(C_{\fs}(x)))$.
    \item $\mathfrak{L}(x)$ the number of components in $E(C_{\fs}(x))$.
    \item $\mathfrak{L}(\fs):=\text{max}_{1\ne x\in\Omega_1(Z(S))} \mathfrak{L}(x)$.
    \item $\mathfrak{X}(\fs):=\{1\ne x \in \Omega_1(Z(S)) \mid \mathfrak{L}(x)=\mathfrak{L}(\fs)\}$.
    \item $\mathfrak{Q}(\fs):=\{O_p(\fs_x) : x \in \mathfrak{X}(\fs)\}$.
\end{itemize}
\end{notation}

We have the following hypothesis for this section.

\begin{hypothesis}\label{CompHyp}
\cref{mainhyp} holds and $\fs$ is of parabolic component type.
\end{hypothesis}

An immediate consequence of \cref{CompHyp} is that $l:=\mathfrak{L}(\fs)\ne 0$. Indeed, for any $x\in \mathfrak{X}(\fs)$, we have that $E(C_{\fs}(x))$ is non-trivial. Throughout, for $x\in \mathfrak{X}(\fs)$ write $Q_x:=O_p(\fs_x)$.

\begin{lemma}\label{BasicQx}
For $x\in \mathfrak{X}(\fs)$, we have that $E(C_{\fs}(x))=E(\fs_x)$ and $Q_x=C_S(E(\fs_x))$. 
\end{lemma}
\begin{proof}
 By \cite[{}9.9]{aschfit} we have that that $E(\fs_x)$ centralizes $x\in Q_x$. Since $C_S(E(C_{\fs}(x)))\normaleq S$, applying \cref{balance} we have that $E(\fs_x)\le E(C_{\fs}(x))$. But now $E(C_{\fs}(x))\le \fs_x$ by definition and it follows that $E(C_{\fs}(x))\le E(\fs_x)$ and so $E(C_{\fs}(x))=E(\fs_x)$. 

Moreover, $C_S(E(\fs_x))\le O_p(\fs_x)=Q_x$ by definition. By \cite[{}9.9]{aschfit}, we have that $Q_x=O_p(\fs_x)\le C_S(E(\fs_x))$ and so $Q_x=C_S(E(\fs_x))$. 
\end{proof}

We explain our choice to analyze the systems $\fs_x$ in place of $C_{\fs}(x)$ with a motivating ``example." Consider the group $\PSp_4(2^n)$ for $n>1$. Then, for some choice of $S\in\syl_2(\PSp_4(2^n))$ and $x\in Z(S)$, we have that $C_{\PSp_4(2^n)}(x)\cong 2^{3n}:\SL_2(2^n)$. 

For the examples we expect in the \hyperlink{MainThm}{Main Theorem}, writing $\fs=\fs_1\times \fs_2\times \dots \times \fs_n$ where $\fs_1$ supported on $T\le S$ and is isomorphic to the $2$-fusion system of $\PSp_4(2^n)$, we would have for $x\in T\cap \Omega_1(Z(S))$ that $O_2(C_{\fs}(x))$ is elementary abelian of order $2^{3n}$ and $F^*(C_{\fs}(x))$ is supported on a $2$-group strictly small than $S$. On the other hand, in this situation, we have that $Q_x=C_S(E(C_{\fs}(x)))=C_S(\fs_2\times \dots \times \fs_n)$ is Sylow in $\fs_1$. Indeed, if $\fs_2\times \dots \times \fs_n$ is reduced then $F^*(\fs_x)$ is supported on $S$. 

Generically, this allows us to decompose $S$ as a direct product, and then our goal is to demonstrate that each of the direct factors is strongly closed in $\fs$, for then \cite[Proposition 3.3]{AOV} promises that $\fs$ splits as a direct product, contradicting \cref{MinSimple}.

\begin{lemma}\label{FonS}
We have that $F^*(\fs_x)$ is supported on $S$, or $\fs_x$ has a component isomorphic to the $2$-fusion category of $\Alt(6)$.
\end{lemma}
\begin{proof}
This is a consequence of \cref{CompAutos} and \cref{SFitting}.
\end{proof}

We remark that if $F^*(\fs_x)$ is supported on $S$, then $Z(Q_x)\le Z(S)$ for each $x\in\mathfrak{X}(\fs)$.

\begin{proposition}\label{NoL34}
Suppose that there is $Q\normaleq S$ and $K\in\Comp(N_{\fs}(Q))$ with $K/Z(K)$ isomorphic to the $2$-fusion system of $\PSL_3(4)$. Then $Z(K)=\{1\}$.
\end{proposition}
\begin{proof}
Aiming for a contradiction, we assume throughout that there is $Q\normaleq S$ and $K\in\Comp(N_{\fs}(Q))$ such that $K\cong 2.\PSL_3(4)$ or $2^2.\PSL_3(4)$. We fix this choice of $Q$ and $K$. We write $\mathcal{N}:=N_{\fs}(Q)$ and observe by \cref{CompAutos} that $Z(K)\le Z(S)$. 

Suppose first that $Z(K)$ is central in $O^{2'}(N_{\fs}(E))$ for all $E\in\mathcal{E}(\fs)$. By a Frattini argument, we have that $\fs=\langle O^{p'}(N_{\fs}(E)), \Aut_{\fs}(S) \mid E\in\mathcal{E}(\fs)\rangle_S$. Now, $\mathcal{E}(\fs)$ is an $\Aut_{\fs}(S)$-invariant set, and so $\Aut_{\fs}(S)$ normalizes $Z:=\bigcap\limits_{E\in\mathcal{E}(\fs)} Z(O^{p'}(N_{\fs}(E)))$ so that $Z\normaleq \fs$. Since $O_2(\fs)=\{1\}$, we have that $Z=\{1\}$. But $Z(K)\le Z$, a contradiction.

Hence, there is $E\in\mathcal{E}(\fs)$ such that $Z(K)\not\normaleq O^{2'}(N_{\fs}(E))$. Let $R\le S$ be such that $K$ is supported on $R$. Since $E'\le Z(S)$, we have that $E'\le Z(O^{2'}(N_{\fs}(E)))$. Indeed, if $R\le E$, then $Z(K)\le R'\le E'$, a contradiction. Hence, $R\not\le E$. By \cref{CompAutos}, we have that $S=RC_S(R)$ so that $Z(R)\le Z(S)$ and as $|R/Z(R)|=2^4$, applying \cref{SEFF} we deduce that $E\cap R$ has index at most $4$ in $R$. But now, one can check that either $(E\cap R)/Z(K)$ is elementary abelian of order $2^4$, or $Z(E)\cap R=Z(E\cap R)=Z(R)$. In the latter case, we have that $R$ centralizes the chain $\{1\}\normaleq Z(E)\normaleq E$ and we have that $R\le E$ by \cref{Chain}, a contradiction. Hence, $(E\cap R)/Z(K)$ is elementary abelian of order $2^4$ and applying \cref{SEFF} we see that $O^{2'}(\Out_{\fs}(E))\cong \PSL_2(4)$ and $S=ER$. In particular, $E=(E\cap R)C_S(E\cap R)$. 

Let $e\in (E\cap R)\setminus Z(R)$ and let $\alpha\in \Aut_K(E\cap R)$ such that $e\alpha \in Z(R)\le Z(S)$. Then by \cite[Lemma I.2.6 (c)]{ako}, $\alpha$ extends to $\hat{\alpha}: C_S(e)\to S$. But $C_S(e)=(E\cap R)C_S(E\cap R)=E$ and since $\alpha$ does not preserve $Z(S)$, which is characteristic in $S$, we must have that $\hat{\alpha}\in\Aut_{\fs}(E)$. Moreover, $\hat{\alpha}$ does not normalize $S$ and so this generates $O^{2'}(\Aut_{\fs}(E))$ together with $\Aut_S(E)$. But then $O^{2'}(\Aut_{\fs}(E))$ normalizes $Z(K)$, a contradiction.
\end{proof}

\begin{proposition}\label{CompDecomp}
For $x\in \mathfrak{X}(\fs)$ we have that $F^*(\fs_x)=Q_x\times \mathcal{C}_1\times \dots \times \mathcal{C}_l$ and each $\mathcal{C}_i$ is a simple fusion system on a $p$-group $S_i\in\mathcal{S}$. In particular, if no $\mathcal{C}_i$ is isomorphic to the $2$-fusion category of $\Alt(6)$, then $F^*(\fs_x)=O^{p'}(\fs_x)$.
\end{proposition}
\begin{proof}
We observe that $O^{p'}(F^*(\fs_x))=F^*(\fs_x)$ by \cref{p'genfit}. For each $\mathcal{C}\in\Comp(\fs_x)$, we have that $O_p(\mathcal{C})=\{1\}$ by \cref{L34} and \cref{NoL34} so that $\mathcal{C}$ is simple. Moreover, from this we deduce that $Z(E(\fs_x))=\{1\}$ and since $Q_x=C_S(E(\fs_x))$ it follows that $F^*(\fs_x)=Q_x\times E(\fs_x)$. Since $E(\fs_x)$ is supported on a $p$-group of class two, and $O^{p'}(E(\fs_x))=E(\fs_x)$, by induction $E(\fs_x)$ is determined as in the \hyperlink{MainThm}{Main Theorem}. This proves the first statement. 

If $\fs_x$ has no components isomorphic to the $2$-fusion category of $\Alt(6)$, then $F^*(\fs_x)$ is supported on $S$ by \cref{FonS} and as $O^{p'}(F^*(\fs_x))=F^*(\fs_x)$, we have that $O^{p'}(\fs_x)=F^*(\fs_x)$ by \cref{SNorm}.
\end{proof}

We now begin investigating the relationship between elements of $\mathfrak{X}(\fs)$, and of $\mathfrak{Q}(\fs)$. 

\begin{lemma}\label{QTI1}
Let $A\normaleq S$, $z\in\Omega_1(Z(S))$ and assume that $C_{N_{\fs}(A)}(z)$ is not constrained. Then for $\mathcal{K}\in \Comp(C_{N_{\fs}(A)}(z))$, we have that $\mathcal{K}\in \Comp(N_{\fs}(A))$. In particular, for $x\in\mathfrak{X}(\fs)$ we have that either $Q_x=Q_z$ or $\Comp(C_{\fs_x}(z)) \subsetneq \Comp(\fs_x)$. 
\end{lemma}
\begin{proof}
Let $A\normaleq S$ and $z\in\Omega_1(Z(S))$ with $C_{N_{\fs}(A)}(z)$ not constrained. By \cref{balance} and \cite[{}10.11 (1)]{aschfit}, we have that $E(C_{N_{\fs}(A)}(z))\subseteq C_{E(N_{\fs}(A))}(z)\normaleq C_{N_{\fs}(A)}(z)$. Furthermore, $E(C_{N_{\fs}(A)}(z))\normaleq C_{N_{\fs}(A)}(z)$. Let $C_i\in \Comp(N_{\fs}(A))$ so that by \cite[{}10.11 (2)]{aschfit}, $C_{\mathcal{C}_i}(z)\normaleq C_{E(N_{\fs}(A))}(z)$. Hence, by \cref{wedge} we may form $X_i:=C_{\mathcal{C}_i}(z) \wedge E(C_{N_{\fs}(A)}(z))\normaleq C_{E(N_{\fs}(A))}(z)$. 

Let $\mathcal{K}\in \Comp(C_{N_{\fs}(A)}(z))$ so that $\mathcal{K}\subseteq C_{E(N_{\fs}(A))}(z)$. Then $\mathcal{K}\normaleq E(C_{N_{\fs}(A)}(z))$ and so we may form $X_i \wedge \mathcal{K}\normaleq E(C_{N_{\fs}(A)}(z))$. Let $\mathcal{K}$ be supported on $K$, let $\{\mathcal{C}_i\}$ be the set of components of $N_{\fs}(A)$, $l:=|\Comp(N_{\fs}(A))|$ and let $\mathcal{C}_i$ be supported on $R_i$. We note that $K\le R:= R_1\times \dots R_l$. If $[K, R_i]=\{1\}$ for all $i$, then $K\le Z(R)$ and $K$ is abelian, a contradiction since $\mathcal{K}$ is quasisimple. Hence, $X_i\wedge \mathcal{K}$ is non-trivial for some $i$, and normal in $\mathcal{K}$ by \cref{wedge}. Since $\mathcal{K}$ is quasisimple, either $\mathcal{K}\subseteq X_i$ or $X_i\wedge \mathcal{K}\le Z(\mathcal{K})$. The latter case does not hold by \cref{L34} and \cref{NoL34}.

Hence, we deduce that $\mathcal{K}\subseteq X_i$ and so, reindexing if necessary, we may assume that $\mathcal{K}\subseteq \mathcal{C}_1$. Moreover, $\mathcal{C}_1$ is the unique component of $N_{\fs}(A)$ containing $\mathcal{K}$. We observe that $z\in \Omega_1(Z(S))$ so that $C_{N_{\fs}(A)}(z)$ is supported on $S$. Hence, by \cref{CompAutos} we deduce that $K\normaleq S$ and $S=KC_S(K)$.

Assume that $R_1$ is isomorphic to a Sylow $p$-subgroup of $\PSL_3(p^n)$ for some $n\in \N$. Then for any $x\in K\setminus Z(R_1)$, we have that $[x, R_1]=Z(R_1)\le K\cap Z(R_1)\le Z(K)$. Since $K, R_1\in\mathcal{S}$, it follows that $K=R_1$ unless perhaps $\mathcal{K}$ is isomorphic to the $2$-fusion system of $\PSp_4(2^a)$ for some $a>1$. But in this latter case, by \cref{CompAutos} we have that $R_1=KC_{R_1}(\mathcal{K})$. Since $Z(R_1)\le K$, we conclude that $Z(R_1)\cap C_{R_1}(\mathcal{K})=\{1\}$ from which we deduce that $C_{R_1}(\mathcal{K})=\{1\}$ and $R_1=K$. By \cref{KnownSL}, we either have that $\mathcal{C}_1=\mathcal{K}$, or $R_1\cong p^{1+2}_+$ and $p$ is odd. In this latter case, we have that $C_S(\mathcal{C}_1)\le C_S(\mathcal{K})\le C_S(K)=C_S(R_1)$. Moreover, by \cref{CompAutos}, it follows that $C_S(\mathcal{C}_1)$ has index $p$ in $C_S(R_1)$ and as $C_S(\mathcal{K})<C_S(K)$, we conclude that $C_S(\mathcal{C}_1)=C_S(\mathcal{K})$ and $\mathcal{C}_1\subseteq C_{N_{\fs}(A)}(z)$. Indeed, $\mathcal{C}_1\subseteq E(C_{N_{\fs}(A)}(z))$ and since $\mathcal{K}\normaleq E(C_{N_{\fs}(A)}(z))$ and $\mathcal{C}_1$ is simple, we conclude that $\mathcal{C}_1=\mathcal{K}$. Hence, the first statement holds in this case.

Thus, for the first statement, it remains to verify the claim when $p=2$ and $R_1$ is isomorphic to a Sylow $2$-subgroup of $\PSp_4(2^n)$ for some $n>1$. As above, let $x\in K\setminus Z(R_1)$ so that $[x, R_1]\le K\cap Z(R_1)\le Z(K)$. By \cref{CompAutos} we have that $S=KC_S(K)$ so that $[x, K]\le [x, R_1]\le [x, S]=[x, K]$. Now, $|[x, R_1]|=2^n$. If $K$ is isomorphic to a Sylow $2$-subgroup of $\PSL_3(2^a)$ then $|[x, K]|=2^a$ so that $a=n$ and $R_1=KZ(R_1)$. But then $2^{2n}=|R_1'|=|K'|=2^a$, a contradiction.  If $K$ is isomorphic to a Sylow $2$-subgroup of $\PSp_4(2^a)$ then $|[x, K]|=2^a$ from which we deduce that $a=n$ and $K=R_1$. Then \cref{KnownSp} implies that $\mathcal{C}_1=\mathcal{K}$. Hence, the first assertion of the lemma holds.

Now, as $E(C_{\fs}(x))=E(\fs_x)$, we have that $E(C_{\fs}(x))=E(\fs_x)=E(N_{\fs}(Q_x))$. The only case to consider is where $\Comp(C_{\fs_x}(z))=\Comp(\fs_x)$. Then $z\in C_S(E(\fs_x))=Q_x$ and $E(C_{\fs}(x))=E(C_{C_{\fs}(z)}(x))\subseteq E(C_{\fs}(z))$. Hence, we have that $\Comp(C_{\fs}(x))\subseteq \Comp(C_{\fs}(z))$ and as $x\in \mathfrak{X}(\fs)$, we must have that $\Comp(C_{\fs}(x))=\Comp(C_{\fs}(z))$. Then $E(C_{\fs}(x))=E(C_{\fs}(z))$ and we conclude that $Q_x=Q_z$. 
\end{proof}

We record some consequences of \cref{QTI1} for application later.

\begin{lemma}\label{QTI3}
For each $1\ne z\in Q_x\cap \Omega_1(Z(S))$ we have that $z\in \mathfrak{X}(\fs)$ and $E(C_{\fs}(z))=E(C_{\fs}(x))$.
\end{lemma}
\begin{proof}
We note that $E(C_{\fs_x}(z))\subset E(\fs_x)$ and as $E(\fs_x)$ centralizes $z$, we deduce that $E(C_{\fs_x}(z))=E(\fs_x)$. Hence, by \cref{QTI1}, we have that $Q_x=Q_z$. Now, $E(N_{\fs}(Q_x))=E(\fs_x)$ and from this we conclude that $E(\fs_x)=E(\fs_z)$. Then $E(C_{\fs}(x))=E(C_{\fs}(z))$ and $z\in \mathfrak{X}(\fs)$, as desired.
\end{proof}

\begin{lemma}\label{QTI2}
For each $x, y\in \mathfrak{X}(\fs)$ we have that either $Q_x\cap Q_y=\{1\}$ or $Q_x=Q_y$.
\end{lemma}
\begin{proof}
Note that $Q_x=C_S(E(C_{\fs}(x)))$ and $Q_y=C_S(E(C_{\fs}(y)))$. Assume that $Q_x\cap Q_y\ne\{1\}$ and let $z\in Q_x\cap Q_y\cap \Omega_1(Z(S))$. Then by \cref{QTI3}, we have that $z\in \mathfrak{X}(\fs)$ and $E(C_{\fs}(x))=E(C_{\fs}(z))=E(C_{\fs}(y))$. By definition, we have that $Q_x=Q_y$, as desired.
\end{proof}

Choose $x\in \mathfrak{X}(\fs)$. Write $\{\mathcal{C}_i\}_{i\in\{1,\dots, l\}}$ for the systems described in \cref{CompDecomp}, and $R_i$ for the $p$-subgroup that $\mathcal{C}_i$ is supported on. Set $R:=\prod_{i\in \{1,\dots, l\}} R_i$ so that $E(\fs_x)$ is supported on $R$.

\begin{lemma}\label{Radical}
If $Q\in\fs^{frc}$ then either:
\begin{enumerate}
    \item $R_i\le Q$ and $Z(R_i)=R_i\cap Z(Q)$; or
    \item $R_i\cap Q\in\mathcal{A}(R_i)$,
\end{enumerate}
for all $i\in \{1,\dots, l\}$. Moreover, if $Q\in \mathcal{E}(\fs)$ then there is at most one $R_i$ with $R_i\not\le Q$ and $S=R_iQ$.
\end{lemma}
\begin{proof}
Suppose first that $p=2$. Then $\mathcal{A}(R_i)=\{A, B\}$ and every involution in $R_i$ is contained in $A$ or $B$. Moreover, for any $x\in A\setminus Z(R_i)$, using that $S=R_iC_S(R_i)$, we have that $C_S(x)=C_S(R_i)A$. Similarly, for any $y\in B\setminus Z(R_i)$, we have that $C_S(y)=C_S(R_i)B$. Observe also that $[R_i, Q]\le Z(R_i)\le \Omega_1(Z(S))\le \Omega_1(Z(Q))$.

Assume that there is $x\in (A\cap Q)\setminus Z(R_i)$ and $y\in (B\cap Q)\setminus Z(R_i)$. Then $\Omega_1(Z(Q))\le C_S(x)\cap C_S(y)=C_S(R_i)$. But then $R_i$ centralizes the chain $\{1\}\normaleq \Omega_1(Z(Q))\normaleq Q$ and as $Q$ is centric-radical, $R_i\le Q$ by \cref{Chain} and (i) holds. Thus, we may assume that $\Omega_1(Z(Q))\cap R_i$ is contained in one of $A$ or $B$. Without loss of generality, assume that $\Omega_1(Z(Q))\cap R_i\le A$ so that $\Omega_1(Z(Q))\le C_S(A)$. Then, again using that $Q$ is centric-radical, we see that $A\le Q$. Indeed, we must have that $A=R_i\cap Q$ and (ii) holds, for otherwise $Q\cap B>Z(R_i)$ and we may select $x$ and $y$ as before.

Suppose now that $Q_i:=R_i\cap Q<R_i$ and $p$ is odd so that $R_i\cong q^{1+2}$ where $q=p^{n_i}$ for some $n_i\in\N$. In particular, $R_i$ has exponent $p$. Since $Q$ is centric, there is $x\in Q_i\setminus Z(R_i)$. By an elementary counting argument, there is $A\in\mathcal{A}(R_i)$ with $x\in A$. For any $y\in R_i\setminus A$, we record that $C_A(y)=Z(R_i)$. It follows that either $Q_i\le A$ and, as $Q$ is centric, $Q_i=A$; or there is $y\in Q_i\setminus (Q_i\cap A)$ and $Z(Q_i)=Z(R_i)$. In the latter case, we see that $R_i$ centralizes the chain $\{1\}\normaleq Z(Q)\normaleq Q$, a contradiction by \cref{Chain} since $R_i\not\le Q$. 

Suppose that $Q\in\mathcal{E}(\fs)$ and suppose that $R_i\not\le Q$. Hence, $R_i\cap Q\in\mathcal{A}(Q_i)$ and so $|(R_i\cap Q)/Z(R_i)|=|R_iQ/Q|$. Since $S=R_iC_S(R_i)$, $R_i$ centralizes an index $|R_iQ/Q|$ subgroup of $Q$ and applying \cref{SEFF}, we see that $O^{p'}(\Out_{\fs}(E))\cong \SL_2(|R_iQ/Q|)$ and $S=R_iQ$. If there is $R_j\not\le Q$ with $i\ne j$, then $R_i\ge [R_i, Z(Q)]=[R_iQ, Z(Q)]=[S, Z(Q)]=[R_jQ, Z(Q)]=[R_j, Z(Q)]\le R_j$. Since $R_i\cap R_j=\{1\}$, we must have that $[S, Z(Q)]=\{1\}$, a contradiction since $Z(Q)$ contains all non-central $O^{p'}(\Out_{\fs}(E))$-chief factors by \cref{GeneralEssentialStructure}.

This completes the proof of the lemma.
\end{proof}

\begin{lemma}\label{EssentialsIn}
Suppose that $E\in\mathcal{E}(\fs)$. If $R\not\le E$ then $Q_x\le C_E(O^p(O^{p'}(\Aut_{\fs}(E))))$.
\end{lemma}
\begin{proof}
Let $E\in\mathcal{E}(\fs)$ such that $R\not\le E$ and so, by \cref{Radical}, there is a unique $k\in \{1,\dots, l\}$ with $R_k\not\le E$. Indeed, by \cref{CompAutos} and \cref{Radical}, for $q_k:=|R_k/E|$ we have that $R_k$ centralizes an index $q_k$ subgroup of $E$. By \cref{SEFF}, we deduce that $O^{p'}(\Out_{\fs}(E))\cong \SL_2(q_k)$ and $S=ER_k$. Then $[Q_x, Z(E)]\le [S, Z(E)]=[R_k, Z(E)]$ from which we deduce that $[Q_x, Z(E)]\le Q_x\cap R_k=\{1\}$ so that $Q_x\le C_S(Z(E))=E$ by \cref{GeneralEssentialStructure}.

We claim that $E\cap R_k\in \mathcal{E}(\mathcal{C}_k)$. By \cref{Radical}, $E\cap R_k\in\mathcal{A}(R_k)=\mathcal{E}(\mathcal{C}_k)$ if $p=2$ and so we may assume that $p$ is odd. Let $t$ be the lift to $\Aut_{\fs}(S)$ of an involution in $Z(O^{p'}(\Aut_{\fs}(E)))$. Then $Z:=Z(S)\cap R_k=[Z(E), S]$ is normalized by $t$ and so $t\in \Aut_{N_{\fs}(Z)}(S)$. Applying \cref{QTI1}, we see that $\mathcal{C}_1, \dots, \mathcal{C}_{k-1}, \mathcal{C}_{k+1}, \dots, \mathcal{C}_l\in \Comp(N_{\fs}(Z))$. 

Observe that if $t\in \Aut_{N_{\fs}(Q_x)}(S)$ then $t$ acts on $\Comp(N_{\fs}(Q_x))=\Comp(\fs_x)$, and since $t$ normalizes $Z$, $t$ normalizes $R_k$. Since $p$ is odd, we deduce that $E\cap R_k=[E, t]=[Z(E), O^{p'}(\Aut_{\fs}(E))]$. But then \[\SL_2(q_k)\cong \Aut_{O^{p'}(\Aut_{\fs}(E)}(E\cap R_k)=O^{p'}(\Aut_{N_{\fs}(Q_x)}(E\cap R_k))=O^{p'}(\Aut_{\mathcal{C}_k}(E\cap R_k))\] and the claim holds. Since $Q_x\le E$, if $Q_x$ is non-abelian then $Q_x'\le E'\le C_E(O^{p'}(\Aut_{\fs}(E)))$ and we have by \cref{QTI3} that $t\in \Aut_{N_{\fs}(Q_x)}(S)$ and the claim holds. 

Suppose there is $\mathcal{K}\in \Comp(N_{\fs}(Z))$ with $\mathcal{K}\ne \mathcal{C}_i$ for any $i$, and let $\mathcal{K}$ be supported on $T$. Then $[T, R_i]=\{1\}$ for $i\in \{1,\dots, k-1, k+1,\dots, l\}$. Moreover, $[T, R_k]\le T\cap Z\le T\cap O_p(N_{\fs}(Z))=\{1\}$ by \cref{NoL34}. Then $T\le Q_xZ(S)$ and since $T$ is non-abelian, so too is $Q_x$, and the claim holds. Suppose now that $\mathcal{C}_1, \dots, \mathcal{C}_{k-1}, \mathcal{C}_{k+1}, \dots, \mathcal{C}_l=\Comp(N_{\fs}(Z))$. Then $Q_x\times Z=Z(C_S(E(C_{\fs}(Z))))$ is normalized by $t$. Moreover, since $t\in \Aut_{\fs}(S)$, $t$ preserves the set $\mathfrak{X}(\fs)$. But $Q_x\setminus \{1\}$ is exactly the set of elements of $Q_x\times Z$ contained in $\mathfrak{X}(\fs)$ from which we conclude that $t$ normalizes $Q_x$, and that $t\in \Aut_{N_{\fs}(Q_x)}(S)$ and so the claim holds.

We observe that $O^{p'}(\Aut_{\mathcal{C}_k}(E\cap R_k))\cong \SL_2(q_k)\cong O^{p'}(\Out_{\fs}(E))$. Let $e\in (E\cap R_k)\setminus Z(R_k)$ and let $\alpha\in \Aut_{\mathcal{C}_k}(E\cap R)$ such that $e\alpha \in Z(R_k)\le Z(S)$ by \cref{CompAutos}. Then by \cite[Lemma I.2.6 (c)]{ako}, $\alpha$ extends to $\hat{\alpha}: C_S(e)\to S$. But $C_S(e)=(E\cap R_k)C_S(E\cap R_k)=E$ and since $\alpha$ does not preserve $Z(S)$, which is characteristic in $S$, we must have that $\hat{\alpha}\in\Aut_{\fs}(E)$. Moreover, $\hat{\alpha}$ does not normalize $S$ and so this generates $O^{p'}(\Aut_{\fs}(E))$ together with $\Aut_S(E)=\Aut_{R_k}(E)\Inn(E)$. Since $\alpha\in \mathcal{C}_k$, $\alpha\in C_{\fs}(Q_x)$, and as $\Aut_{R_k}(E)$ also centralizes $Q_x$, we deduce that $O^p(O^{p'}(\Aut_{\fs}(E)))$ centralizes $Q_x$ and $Q_x\le C_E(O^p(O^{p'}(\Aut_{\fs}(E))))$, as desired.
\end{proof}

\begin{lemma}\label{S/Qx-type}
Suppose that $E(\fs_x)$ is isomorphic to the $2$-fusion category of $\Alt(6)$. Then $S/Q_x$ is isomorphic to $\Dih(8)$ or $\Dih(8)\times 2$. 
\end{lemma}
\begin{proof}
Let $E(\fs_x)$ be supported on $R$. Note that there are morphisms in $E(\fs_x)$ (e.g. elements of order $3$ acting on the elementary abelian $2$-subgroups of $R$ of order $4$) which centralize $Q_x$ but do not lift to morphisms of $R$. In particular, we must have that $O_2(\fs_x/Q_x)<S/Q_x$. Write $\bar{S}:=S/Q_x$ so that $|\bar{S}|=16$. Note that $S=RC_S(R)$ so that $\bar{S}=\bar{R}\bar{C_S(R)}=\bar{R}C_{\bar{S}}(\bar{R})$.

Let $s\in\bar{C_S(R)}$ of minimal order such that $r\not\in \bar{R}$. Since $|Z(\bar{R})|=2$ we see that $s$ has order at most $4$. If $s$ has order $2$ then $\bar{S}=\bar{R}\times s\cong \Dih(8)\times 2$. If $s$ has order $4$ then $\bar{S}\cong \Dih(8)\circ C_4\cong Q_8\circ C_4$. Then $Z(\bar{S})$ is cyclic of order $4$ and $\mho_1(\bar{S})=\bar{S}'=Z(\bar{R})$. 

Assume there is $E\in\mathcal{E}(\fs_x/Q_x)$ so that $Z(\bar{S})<E$. It follows that $Z(\bar{R})=\mho^1(E)$. But then $[\bar{S}, E]\le \bar{S}'=\mho_1(E)$, a contradiction by \cref{Chain}. Hence, $\fs_x/Q_x$ has no essential subgroups and by the Alperin--Goldschmidt theorem and \cref{normalinF}, we see that $\bar{S}\normaleq \fs_x/Q_x$, a contradiction. 
\end{proof}

By \cref{QTI3}, it is clear that the choice of $x\in\mathfrak{X}(\fs)$ is not unique. We now demonstrate that the choice of $Q_x$ is not unique. 

In the following proof, when $p=2$ we appeal to a result of Bender for candidates for $\Out_{\fs}(E)$ whenever $E\in\mathcal{E}(\fs)$ and $m_2(S/E)\geq 2$. That is, we obtain a list of certain groups with strongly $2$-embedded subgroups (a result we will use in the determination of simple groups in the final section). However, since $S'\le Z(S)\le E$, using a standard reduction to almost simple groups, we could instead appeal to the classification of simple groups with abelian Sylow $2$-subgroups due to Walter \cite{walter} (or see a shorter proof due to Bender \cite{BenderAbelian}). Indeed, \cite{Bender} makes use of this result, and so we only appeal to Bender's result for convenience, since it (and Walter's theorem) will be required in other parts of this work.

\begin{proposition}\label{NotUnique}
We have that $|\mathfrak{Q}(\fs)|>1$.
\end{proposition}
\begin{proof}
Aiming for a contradiction, we assume that $\mathfrak{Q}(\fs)=\{Q_x\}$. We note first that by uniqueness, $Q_x$ is invariant under $\Aut_{\fs}(S)$, as is $Z(Q_x)$. Then since $E(C_{\fs}(x))=E(C_{\fs}(Z(Q_x)))$ by \cref{QTI3}, it follows that $R$ is also invariant under $\Aut_{\fs}(S)$. Our first aim will be to demonstrate that both $Q_x$ and $R$ are strongly closed in $\fs$. We will suppose that $E\in\mathcal{E}(\fs)$ is chosen such that either $(E\cap Q_x)$ or $(E\cap R)$ is not $\Aut_{\fs}(E)$-invariant. 

By \cref{EssentialsIn}, we observe that if $R\not\le E\in\mathcal{E}(\fs)$ then by a Frattini argument, we have $Q_x\normaleq \Aut_{\fs}(E)$. Moreover, for $z\in Q_x\cap \Omega_1(Z(S))$, we have that $O^{p'}(\Aut_{\fs}(E))\subseteq C_{\fs}(z)$. By \cref{QTI3}, we have that $E(C_{\fs}(x))=E(C_{\fs}(z))$. Since $R$ is strongly closed in $C_{\fs}(z)$, we deduce that $E\cap R$ is normalized by $O^{p'}(\Aut_{\fs}(E))$. Since $R$ and $Q_x$ are $\Aut_{\fs}(S)$-invariant, the Frattini argument shows that

\textbf{(1)} if $E\in\mathcal{E}(\fs)$ has $R\not\le E$ then $Q_x$ and $E\cap R$ are normalized by $\Aut_{\fs}(E)$.\hfill

Assume now that $R\le E\in\mathcal{E}(\fs)$. If $[E, E\cap Q_x]\ne \{1\}$ then there is $z\in E'\cap Q_x\cap \Omega_1(Z(S))$ centralized by $O^{p'}(\Aut_{\fs}(E))$. As above, an application of \cref{QTI3} implies that $O^{p'}(\Aut_{\fs}(E))$ normalizes $R$ since $E(C_{\fs}(x))=E(C_{\fs}(z))\normaleq C_{\fs}(z)$, and also that 
$O^{p'}(\Aut_{\fs}(E))$ normalizes $E\cap Q_x=C_E(E(C_{\fs}(z)))$. Hence, we assume from now that $[E, E\cap Q_x]=\{1\}$ so that $E\cap Q_x\le Z(E)$. More, by a similar argument we may assume that $C_{Z(E)}(O^{p'}(\Aut_{\fs}(E)))\cap Q_x=\{1\}$. 

Suppose first that $\mathfrak{L}(\fs)>1$ so that $\fs_x$ has at least two components $\mathcal{C}_1$ and $\mathcal{C}_2$ supported on $R_1$ and $R_2$ respectively. Let $r_i\in R_i'$. Then $r_i\le R_i'\le E'$ so that $O^{p'}(\Aut_{\fs}(E))$ centralizes $r_i$. Write $Q_{r_i}:=C_S(E(C_{\fs}(r_i)))$ so that $Q_{r_i}\normaleq S$. By a similar argument to \cref{BasicQx}, we see that $E(C_{\fs}(r_i))=E(C_{\fs}(Q_{r_i}))=E(N_{\fs}(Q_{r_i}))$. Then $\mathcal{C}_{3-i}, \mathcal{C}_j\in \Comp(C_{N_{\fs}(Q_{r_i})}(x))$ for $i\in \{1,2\}$ and $j\in \{3,\dots, l\}$. By \cref{QTI1}, we have that $\mathcal{C}_{3-i}, \mathcal{C}_j\in \Comp(C_{\fs}(r_i))$. If $\mathfrak{L}(r_i)=\mathfrak{L}(\fs)$ then by uniqueness of $Q_x$, we have that $r_i\in O_p(C_{\fs}(r_i))\le C_S(E(C_{\fs}(r_i)))=Q_x$, a contradiction since $R_i\cap Q_x=\{1\}$. Hence, $\mathcal{E}(C_{\fs}(r_i))=\mathcal{C}_{3-i}\times \mathcal{C}_3\times \dots \times \mathcal{C}_l$ for $i\in \{1,2\}$.

Now, $Q_xR_i$ has index at most $2$ in $Q_{r_i}$, with equality only occurring perhaps when $\mathcal{C}_i$ is isomorphic to the $2$-fusion category of $\Alt(6)$. Indeed, as $Q_{r_1}\cap Q_{r_2}\le C_S(\mathcal{C}_1)\cap C_S(\mathcal{C}_2)$, we deduce that $Q_{r_1}\cap Q_{r_2}=Q_x$. Since $O^{p'}(\Aut_{\fs}(E))$ centralizes $r_i$, we have that $O^{p'}(\Aut_{\fs}(E))\subseteq C_{\fs}(r_i)$. Firstly, $O^{p'}(\Aut_{\fs}(E))$ normalizes $E\cap Q_{r_i}=C_E(E(C_{\fs}(r_i)))$. Hence, $O^{p'}(\Aut_{\fs}(E))$ normalizes $E\cap Q_{r_1}\cap Q_{r_2}=E\cap Q_x$. Secondly, $E(C_{\fs}(r_i))$ is supported on $R_{3-i}\times R_3\times \dots \times R_l$ and so $O^{p'}(\Aut_{\fs}(E))$ normalizes $R=(R_{1}\times R_3\times \dots \times R_l)(R_{2}\times R_3\times \dots \times R_l)$. 

Hence, we may assume that 

\textbf{(2)} if $|\mathfrak{Q}(\fs)|=1$ and either $Q_x$ or $R$ is not strongly closed in $\fs$ then $\fs_x$ contains a unique component.\hfill

Assume that $N_{\fs}(R')$ is not constrained. Then for $\mathcal{K}\in\Comp(N_{\fs}(R'))$, $\mathcal{K}$ centralizes $r\in R'\cap \Omega_1(Z(S))$. Hence, $\mathcal{K}\in\Comp(N_{C_{\fs}(r)}(R'))$ and by \cref{balance}, $E(C_{\fs}(r))\ne \{1\}$. Then $r\in\mathfrak{X}(\fs)$. But then, by uniqueness of $Q_x$, $r\in O_p(C_{\fs}(r))\le C_S(E(C_{\fs}(r)))=Q_x$, a contradiction. Hence, $N_{\fs}(R')$ is constrained. Observe that $R'\le E'\le C_E(O^{p'}(\Aut_{\fs}(E)))$ so that $O^{p'}(\Aut_{\fs}(E))\subseteq N_{\fs}(R')$. It follows that $E\in\mathcal{E}(N_{\fs}(R'))$. 

Set $D:=O_p(N_{\fs}(R'))$ so that $C_S(D)\le D\le E$ by \cref{normalinF}. Since $[D, R]\le [S, R]=[RC_S(R), R]=R'$, we deduce by \cref{Chain} that $R\le D$. Since $E\cap Q_x\le Z(E)$ and $D\le E$, we deduce that $E\cap Q_x\le D$.

Suppose that $E(\fs_x)$ is not isomorphic to the $2$-fusion category of $\Alt(6)$. Then $E=(E\cap Q_x)\times R=D$ and so $E$ is $\Aut_{\fs}(S)$-invariant. Let $K$ be a Hall $p'$-subgroup of $\Aut_{E(\fs_x)}(R)$ and consider its lift $\hat{K}$ to $\Aut_{\fs}(S)$. Since $E$ is $\Aut_{\fs}(S)$-invariant, $\hat{K}$ restricts to $\bar{K}\in \Aut_{\fs}(E)$ and since $R\le E$, this is restriction is faithful. Now, $\bar{K}$ centralizes $E\cap Q_x$ and so centralizes $Z(E)/E'=(E\cap Q_x)R'/R'$. Moreover, $O^{p'}(\Out_{\fs}(E))$ acts faithfully on $Z(E)/E'$ and so \[[O^{p'}(\Out_{\fs}(E)), \bar{K}\Inn(E)/\Inn(E)]\le [O^{p'}(\Out_{\fs}(E)), C_{\Out_{\fs}(E)}(Z(E)/E')]=\{1\}\] so that $O^{p'}(\Out_{\fs}(E))$ commutes with $\bar{K}\Inn(E)/\Inn(E)$. But then, $O^{p'}(\Out_{\fs}(E))$ normalizes \linebreak $C_{Z(E)}(\bar{K}\Inn(E)/\Inn(E))=C_{Z(E)}(\bar{K})=E\cap Q_x$. Similarly, $O^{p'}(\Out_{\fs}(E))$ normalizes $[E, \bar{K}]=R$ and as $Q_x$ and $R$ are $\Aut_{\fs}(S)$-invariant, $\Aut_{\fs}(E)$ normalizes both $E\cap Q_x$ and $R$. Thus

\textbf{(3)} if $|\mathfrak{Q}(\fs)|=1$ and either $Q_x$ or $R$ is not strongly closed in $\fs$ then the unique component of $\fs_x$ is isomorphic to the $2$-fusion category of $\Alt(6)$.\hfill

Assume now that $E(\fs_x)$ is isomorphic to the $2$-fusion category of $\Alt(6)$. We claim first that $E=D$. If $E=(E\cap Q_x)\times R$ then clearly $E=D$. Hence, there is $t\in E\setminus (E\cap Q_x)R$. Since $S=RC_S(R)$, we can arrange that $t\in C_E(R)$. Then $[t, E\cap Q_x]\le [E, E\cap Q_x]=\{1\}$ and $[t, R]=\{1\}$ and so $[t, D]=\{1\}$. Since $D$ is centric, we see that $t\in D$ and so $(E\cap Q_x)\times R<E=D$.

Suppose that $m_2(S/E)=1$. Then by \cite[Theorem I.5.4.10]{gor}, and using that $S'\le Z(S)\le E$, we see that $S/E$ is cyclic. Then by the Hall-Higman argument \cite[Theorem 11.1.1]{gor}, using that $S$ acts quadratically on $E$, we see that $|S/E|=2$. Let $F\in\mathcal{E}(\fs)$ with $R\le F$. Then $R'\le F'$ so that $O^{2'}(\Aut_{\fs}(F))$ centralizes $R'$ and $F\in\mathcal{E}(N_{\fs}(R'))$. Then $E=D\le F$ and since $|S/E|=2$, we conclude that $E=F$. Hence, $E$ is the unique essential subgroup containing $F$. We will demonstrate that $E$ is strongly closed in $\fs$, and so by a transfer argument \cite[Theorem TL]{transfer}, using that $\fs=O^2(\fs)$, we will have a contradiction. For this, we need only show that $E\cap F$ is normalized by $\Aut_{\fs}(F)$ whenever $R\not\le F$.

Assume that $E=(E\cap Q_x)\times R$ and let $F\in\mathcal{E}(\fs)$ with $R\not\le F$. Then by \textbf{(1)} we have shown that $\Aut_{\fs}(F)$ normalizes $F\cap R$. Moreover, by \cref{EssentialsIn} we see that $O^2(O^{2'}(\Aut_{\fs}(F)))$ centralizes $Q_x$, and so centralizes $E\cap Q_x$. Since $E\cap Q_x=D\cap Q_x$ is $\Aut_{\fs}(S)$-invariant, we conclude that $\Aut_{\fs}(F)$ normalizes $E\cap Q_x$ and so, by the Dedekind modular law, $\Aut_{\fs}(F)$ normalizes $(E\cap Q_x)(F\cap R)=E\cap F$. 

Hence, $(E\cap Q_x)\times R$ has index $2$ in $E$. Then $E=C_E(R)R$. If $C_E(R)$ is not the unique abelian subgroup of maximal order in $C_S(R)$, then for $A$ another maximal order abelian subgroup of $C_S(R)$, we have that $|AE/E|\geq |Z(E)/Z(E)\cap A|$ and so by \cref{SEFF} we have that $O^{2'}(\Out_{\fs}(E))\cong \Sym(3)$ and $C_{Z(E)}(O^{2'}(\Aut_{\fs}(E)))$ has index $4$ in $Z(E)$. Since $C_{Z(E)}(O^{2'}(\Aut_{\fs}(E)))\cap Q_x=\{1\}$, we see that $|E\cap Q_x|\leq 4$ and so $|Q_x|\leq 2^3$ and $|S|\leq 2^7$. Then $\fs$ is known  by \cite{bobsmall} and we find no examples. Hence, $C_E(R)$ is the unique maximal abelian subgroup of $C_E(R)$. It follows that $E\cap F=C_E(R)(F\cap R)$ is the unique abelian subgroup of $F$ of maximal order and so is normalized by $\Aut_{\fs}(F)$. Hence, $E$ is strongly closed in $\fs$ and we reach our desired contradiction.

Suppose now that $m_2(S/E)\geq 2$. In particular, since $[S:Q_XR]\leq 2$, we have that $Q_x\not\le E$. By \cite{Bender}, writing $L_E:=O^{2'}(\Out_{\fs}(E))$, we have that $L_E/O_{2'}(L_E)\cong \PSL_2(q)$, where $q>p=2$. Indeed, $N_{L_E}(\Out_S(E))$ acts irreducibly on $\Out_S(E)$. Since $Q_x$ is $\Aut_{\fs}(S)$-invariant, it follows that $S=Q_xE$. 

Assume that $C_E(R)$ is the unique abelian subgroup of maximal order in $C_S(R)$. Then for $A,B$ the unique maximal abelian subgroups of $R$, we have that $C_E(R)A$ and $C_E(R)B$ are the unique maximal abelian subgroups of $S$ and so $E=C_E(R)R$ is a characteristic subgroup of every subgroup of $S$ which contains it ($E$ is the \emph{Thompson subgroup} of $S$). By coprime action \cite[Theorem 6.2.4]{gor}, $O_{2'}(L_E)=\langle C_{O_{p'}(L_E)}(\Out_A(E)) \mid E<C<S\,\,\text{with}\,\, |C/E|=2\rangle$. Then for each such $C$, we may arrange that $C=E(C\cap Q_x)$ so that $\{1\}\ne [C, Z(E)]\le Z(E)\cap Q_x$. Then the preimage in $\Aut_{\fs}(E)$ of $C_{O_{2'}(L_E)}(\Out_A(E))$ normalizes $[A, Z(E)]$ and so is contained in $N_{\fs}([A, Z(E)])$. Since $[A, Z(E)]\le Q_x=C_S(E(\fs_x))$ it follows from \cref{balance} and \cref{QTI3} that $E(N_{\fs}([A, Z(E)]))=E(\fs_x)$ from which we conclude that the preimage of $C_{O_{2'}(L_E)}(\Out_A(E))$ normalizes $R$ and $E\cap Q_x$. 

Since $O_{2'}(L_E)=\langle C_{O_{2'}(L_E)}(\Out_A(E)) \mid E<A<S\,\,\text{with}\,\, |A/E|=2\rangle$, the preimage in $O_{2'}(L_E)$ normalizes $E\cap Q_x$ and $R$. Now, $L_E$ centralizes $R'$ so that $O_{2'}(L_E)$ centralizes $R'$. Moreover, $O_{2'}(L_E)$ centralizes $E/(E\cap Q_x)R'$ and so we conclude that $[E, O_{2'}(L_E)]\le E\cap Q_x$. Since $L_E$ normalizes $[E, O_{2'}(L_E)]$, we either have that $O_{2'}(L_E)$ centralizes $E$, in which case $O_{2'}(L_E)=\{1\}$; or that the preimage in $\Aut_{\fs}(E)$ of $L_E$ is contained $N_{\fs}([E, O_{2'}(L_E)])$. A similar argument as above exploiting \cref{balance} implies that $L_E$ normalizes $E\cap Q_x$ and $R$ in the latter scenario, a contradiction. 

Thus, $L_E\cong \PSL_2(q)$. Write $W=\langle Z(S)^{L_E}\rangle$ so that $W/C_W(L_E)=[W/C_W(L_E), L_E]$. Now $S$ acts quadratically on $W_E$ and by a result of Niles (see \cite[Lemma 3.4]{nil}), we have that $W/C_W(L_E)$ is a direct sum of natural $\SL_2(q)$-modules. Since $C_W(L_E)\cap Q_x=\{1\}$ and $[S, W]=[Q_x, W]\le Q_x\cap W$, we see that $W=[W, L_E]\times C_W(L_E)$. Now, for $K$ a Hall $p'$-subgroup of $N_{\Aut_{\fs}(E)}(\Aut_S(E))$, we have that $[W, L_E]=[W, K]$ and since $K$ normalizes $W\cap Q_x$ and centralizes $Z(E)/(E\cap Q_x)R'$, we conclude that $[W, L_E]\le E\cap Q_x$. Then balance arguments as above implies that $O^{2'}(\Aut_{\fs}(E))$ normalizes $E\cap Q_x$ and $R$, a contradiction.

Finally, it remains to consider the case where $m_2(S/E)\geq 2$ and $C_E(R)$ is not the unique abelian subgroup of maximal order in $C_S(R)$. By a standard argument and using \cref{SEFF}, we deduce that $L_E\cong \PSL_2(q)$ and $E/C_E(L_E)$ is a natural $\SL_2(q)$-module for $L_E$. Defining $W$ as above, we deduce that $E=W/C_E(L_E)$ and then a similar analysis as above, using that $C_E(L_E)\cap Q_x=\{1\}$ and that $[W, L_E]=[W, K]$, provides a contradiction.

Hence, for all $E\in\mathcal{E}(\fs)$ both $E\cap Q_x$ and $E\cap R$ are normalized by $O^{p'}(\Aut_{\fs}(E))$, and as $Q_x$ and $R$ are $\Aut_{\fs}(S)$-invariant, we have shown that 

\textbf{(4)} if $|\mathfrak{Q}(\fs)|=1$ then both $Q_x$ and $R$ are strongly closed in $\fs$.\hfill

If $S=Q_x\times R$ then an application of \cite[Proposition 3.3]{AOV} yields that there are subsystems $\fs_1, \fs_2$ supported on $Q_x$ and $R$ respectively such that $\fs=\fs_1\times \fs_2$, a contradiction since $\fs$ is simple. 

Hence, by \cref{CompDecomp}, we must have that $p=2$ and $\fs_x$ has a component isomorphic to the $2$-fusion category of $\Alt(6)$. In this case, since $Q_x\not\normaleq \fs$ there is $E\in\mathcal{E}(\fs)$ such that $Q_x\not\le E$ but $E\cap Q_x$ is normalized by $\Aut_{\fs}(E)$.

Assume that $\fs_x$ has a unique component. Then $Q_xR$ has index $2$ in $S$ and $Q_xR$ is strongly closed in $\fs$. But then by the Thompson--Lyons transfer argument \cite[Theorem TL]{transfer}, we have that $O^2(\fs)\ne \fs$, a contradiction. Hence, 

\textbf{(5)} if $|\mathfrak{Q}(\fs)|=1$ then $\fs_x$ has at least two components.

Assume that $N_{\fs}(R)$ is not constrained and let $\mathcal{K}\in\Comp(N_{\fs}(R))$. Choose $r_1\in \Omega_1(Z(S))\cap R_1$. Then $\mathcal{K}\subseteq C_{\fs}(r_1)$ and $N_{C_{\fs}(r_1)}(R)\subseteq N_{\fs}(R)$. It follows that $\mathcal{K}\subseteq E(N_{C_{\fs}(r_1)}(R))$. Then by \cref{balance}, we see that $\mathcal{K}\subseteq E(C_{\fs}(r_1))$. Write $Q_{r_1}:=C_S(E(C_{\fs}(r_1)))$ so that $E(C_{\fs}(r_1))=E(C_{\fs}(Q_{r_1}))=E(N_{\fs}(Q_{r_1}))$. Then by \cref{QTI1}, we have that $\mathcal{K}\in\Comp(C_{\fs}(r_1))$. On the other hand, $\mathcal{C}_i\in \Comp(C_{C_{\fs}(r_1)}(x))$ for $i\in \{2,\dots, l\}$ and by \cref{QTI1}, $\mathcal{C}_i\in \Comp(C_{\fs}(r_1))$. Note that $\mathcal{K}$ centralizes $R$ from which we conclude that $|\Comp(C_{\fs}(r_1))|\geq l$. Since $x\in\mathfrak{X}(\fs)$, we have that $r\in\mathfrak{X}(\fs)$ and by uniqueness of $Q_x$, we conclude that $Q_{r_1}=Q_x$ and $E(\fs_x)=E(C_{\fs}(r_1))$, a contradiction since $\mathcal{K}$ centralizes $R$.

Hence, $N_{\fs}(R)$ is constrained. Note that if $F\in\mathcal{E}(\fs)$ with $R\not\le F$ then $O^2(O^{2'}(\Aut_{\fs}(F)))$ centralizes $Q_x$ and so $O^{2'}(\Aut_{\fs}(F))$ normalizes $O_2(N_{\fs}(R))\cap Q_x$. On the other hand, if $R\le F$ then $R$ is normalized by $O^{2'}(\Aut_{\fs}(F))$. In particular, $F\in\mathcal{E}(N_{\fs}(R))$ and since $O^{2'}(\Aut_{\fs}(F))$ normalizes $F\cap Q_x$, we have that $O^{2'}(\Aut_{\fs}(F))$ normalizes $O_2(N_{\fs}(R))\cap Q_x=O_2(N_{\fs}(R))\cap Q_x\cap F$. Since $O_2(N_{\fs}(R))\cap Q_x$ is normalized by $\Aut_{\fs}(S)$, we must have that $O_2(N_{\fs}(R))\cap Q_x\normaleq \fs$ by \cref{normalinF}. Thus, $O_2(N_{\fs}(R))\cap Q_x=\{1\}$. But $N_{\fs}(R)$ is constrained and $\Omega_1(Z(S))\cap Q_x$ centralizes $O_2(N_{\fs}(R))$, a final contradiction.
\end{proof}

Our aim will be to demonstrate that $Q_x$ is isomorphic to one of the groups in $\mathcal{S}$. For this, we first dispel the case where $Q_x$ is abelian.

\begin{lemma}\label{NonAb}
$Q_x$ is non-abelian.
\end{lemma}
\begin{proof}
Aiming for a contradiction, suppose that $Q_x$ is abelian. Set $X:=Q_x\cap \Omega_1(Z(S))$ so that $X\ne\{1\}$ and let $E\in\mathcal{E}(\fs)$ be such that $X$ is not normalized by $O^{p'}(\Aut_{\fs}(E))$. If $R\not\le E$, then by \cref{EssentialsIn} we have that \[X\le C_E(O^p(O^{p'}(\Aut_{\fs}(E))))\cap \Omega_1(Z(S))\le C_E(O^{p'}(\Aut_{\fs}(E)))\] and so $X$ is normalized by $O^{p'}(\Aut_{\fs}(E))$. Hence, 

\textbf{(1)} if $Q_x$ is abelian and $X$ is not normalized by $O^{p'}(\Aut_{\fs}(E))$ for some $E\in\mathcal{E}(\fs)$ then $R\le E$.\hfill

Since $Q_x$ is abelian and centralized by $R$, and $Z(S)\le E$, if $Q_x\not\le E$ then we must have that $S\ne Q_xR$ and so $p=2$, $\fs_x$ has some component $\mathcal{K}$ supported on $K$ which isomorphic to the $2$-fusion category of $\Alt(6)$ and $C_S(\mathcal{K})K$ has index $2$ in $S$ by \cref{CompDecomp}. If $p$ is odd then $S=Q_x\times R\le E$, a contradiction. Hence, we may assume that $p=2$ for the remainder of this proof.

Suppose that $\mathfrak{L}(\fs)=1$ so that if $Q_x\not\le E$ then $E(\fs_x)=\mathcal{C}_1$ is isomorphic to the $2$-category of $\Alt(6)$. If $Q_x\le E$, then unless $E(\fs_x)$ is isomorphic to the $2$-fusion category of $\Alt(6)$, we have that $S=Q_x\times R=E$, a contradiction. Hence, whether $Q_x$ is contained in $E$ or not, if $\mathfrak{L}(\fs)=1$ then we have that $E(\fs_x)$ is isomorphic to the $2$-fusion category of $\Alt(6)$. 

Assume that $Q_xR=C_S(Q_x)$ has index $2$ in $S$. Choose $y\in\mathfrak{X}(\fs)$ with $Q_x\ne Q_y$. Then $\fs_y$ has one component and by \cref{CompDecomp} and \cref{S/Qx-type}, we either have that $S/Q_y\cong \Dih(8)\times 2$ or $S/Q_y$ is isomorphic to a group in $\mathcal{S}$. Now, $C_S(Q_x)$ has index $2$ in $S$, $Q_xQ_y/Q_y\cong Q_x$ and $C_{S/Q_y}(Q_xQ_y/Q_y)=C_S(Q_x)Q_y/Q_y=EQ_y/Q_y$. Comparing with the groups in $\mathcal{S}$, we see that $S/Q_y\cong \Dih(8)\times 2$ or $\Dih(8)$. Hence, $|Q_x|\leq 2^3$ so that $|S|\leq 2^7$. Then $\fs$ is described in \cite{bobsmall} and we find no examples.

Hence, if $\mathfrak{L}(\fs)=1$ then $Q_x\le Z(S)$. In particular, $Q_xR\le E$ from which we conclude that $E=Q_x\times R$ has index $2$ in $S$. Since $S/Q_x\cong \Dih(8)\times 2$ and $R\cap Q_x=\{1\}$ we conclude that $S$ has an abelian subgroup of index $2$. Again, choose $y\in\mathfrak{X}(\fs)$ with $Q_x\ne Q_y$. Then $\fs_y$ has one component and since $S/Q_y$ has index subgroup of index $2$, by \cref{CompDecomp} and \cref{S/Qx-type} we deduce that $S/Q_y\cong \Dih(8)$ or $\Dih(8)\times 2$. Hence, $|Q_x|\leq 4$ and $|S|\leq 2^6$. As above, $\fs$ is described in \cite{bobsmall} and we find no examples.  

Hence, we may assume now that

\textbf{(2)} if $Q_x$ is abelian then $\mathfrak{L}(\fs)>1$.\hfill

If $Q_x\not\le E$, then we arrange our indexing so that $\mathcal{C}_1$ is isomorphic to the $2$-fusion category of $\Alt(6)$ and $S\ne R_1C_S(\mathcal{C}_1)$.

For $i\in\{1,2\}$ let $r_i\in R_i'$ so that $r_i\in \Omega_1(Z(S))$. Assume first that $r_i\in\mathfrak{X}(\fs)$ so that $E(N_{\fs}(Q_{r_i}))=E(C_{\fs}(r_i))=E(\fs_{r_i})$. Since $\mathcal{C}_{3-i}, \mathcal{C}_3,\dots, \mathcal{C}_l\in\Comp(C_{N_{\fs}(Q_{r_i})}(x))$, applying \cref{QTI1} we see that there is $\mathcal{K}_i\in\Comp(\fs_{r_i})$ such that $E(\fs_{r_i})=\mathcal{K}_i\times \mathcal{C}_{3-i}\times \mathcal{C}_3\times \dots \times \mathcal{C}_l$. Moreover, $\mathcal{K}_i$ centralizes $\mathcal{C}_{3-i}$ and so centralizes $r_{3-i}$ and by another application of \cref{QTI1}, it follows that $\mathcal{K}_i\in\Comp(\fs_{r_{3-i}})$ and $r_{3-i}\in \mathfrak{X}(\fs)$. Indeed, $\mathcal{K}_i\subseteq C_{E(\fs_{r_{3-i}})}(\mathcal{C}_3\times \dots \times \mathcal{C}_l)=\mathcal{C}_i\times \mathcal{K}_{3-i}$. Then $\mathcal{K}_i$ also centralizes $r_i$ from which we conclude that $\mathcal{K}_1=\mathcal{K}_2$. Let $\mathcal{K}_1$ be supported on $K$ so that $K\in\mathcal{S}$. Then $K\le C_S(\mathcal{C}_1\times \dots \times \mathcal{C}_l)$ so that $K\le Q_x$ and $K$ is abelian, a contradiction. 

Hence, we have that $\mathfrak{L}(r_i)=\mathfrak{L}(\fs)-1$ for $i\in \{1,\dots, l\}$. Indeed, arguing as above and using \cref{QTI1}, we have that $E(C_{\fs}(r_i))=\mathcal{C}_{1}\times \dots \times \mathcal{C}_{i-1}\times \mathcal{C}_{i+1} \dots \times \mathcal{C}_l$ and $Q_x=C_S(E(C_{\fs}(r_i)))\cap C_S(E(C_{\fs}(r_j)))$ when $i\ne j$. Now, $r_i\in R_i'\le R'\le E'$ and so $r_i\in C_{E}(O^{2'}(\Aut_{\fs}(E)))$. Hence, $O^{2'}(\Aut_{\fs}(E))\subseteq C_{\fs}(r_i)$ and $O^{2'}(\Aut_{\fs}(E))$ normalizes $C_E(E(C_{\fs}(r_i)))$. Thus $O^{2'}(\Aut_{\fs}(E))$ normalizes $C_E(E(C_{\fs}(r_i)))\cap C_E(E(C_{\fs}(r_j)))=E\cap Q_x$ for $i\ne j$. Then $Q_x$ centralizes the chain $\{1\}\normaleq E\cap Q_x\normaleq E$ and by \cref{Chain} we conclude that $Q_x\le E$ and is $O^{2'}(\Aut_{\fs}(E))$-invariant.

We claim that $Q_x\le Z(S)$.  Choose $y\in\mathfrak{X}(\fs)$ with $Q_x\ne Q_y$ and write $\{\mathcal{T}_i\}_{i\in\{1,\dots, l\}}$ for the components of $O^{p'}(\fs_y)$. We write $T_i$ for the $p$-subgroup of $S$ that $\mathcal{T}_i$ is supported on, and $T=\prod T_i$. If $[Q_x, T]=\{1\}$ then it follows that $Q_xQ_y/Q_y\le Z(S/Q_y)$ from which we conclude that $[S, Q_x]\le Q_x\cap Q_y=\{1\}$. Hence, we may assume without loss of generality that $\{1\}\ne [Q_x, T_1]\le Q_x\cap T_1$. 

Applying \cref{QTI1} and \cref{QTI3}, we deduce that $\mathcal{T}_2,\dots, \mathcal{T}_l\in \Comp(\fs_x)$. Consider $P:=C_S(\mathcal{T}_2\times \dots \times\mathcal{T}_l)$. Then $Q_y\times T_1$ has index at most $2$ in $P$. On the other hand, applying \cref{krullschmidt} we deduce that $\{\mathcal{T}_2,\dots, \mathcal{T}_l\}=\{\mathcal{C}_1,\dots, \mathcal{C}_{c-1}, \mathcal{C}_{c+1}, \dots, \mathcal{C}_l\}$ for some $c\in \{1,\dots, l\}$. Hence, $Q_x\times R_c$ has index at most $2$ in $P$. In fact, as $[Q_x, T_1]\ne \{1\}$, $Q_x$ is not central in $P$ and $Q_x\times R_c$ has index exactly $2$ in $P$. Hence, $R_c\cong \Dih(8)$. Observe that as $\mathfrak{L}(r_c)=\mathfrak{L}(\fs)-1$ for $r_c\in R_c'$, $Q_y$ centralizes $R_c$. Moreover, for $a\in P$ with $a^2=1$, $[a, R_c]=\{1\}$ and $[a, Q_x]\ne \{1\}$, we have that $P$ normalizes $Q_x\langle a\rangle$ and $Q_x$ from which we deduce that $Q_y$ centralizes $Q_x\langle a\rangle$ and so $Q_y$ centralizes $R_cQ_x\langle a\rangle=P$ and $Q_y\le Z(P)$. Now, $Q_x\times R_c$ has an abelian subgroup of index $2$ and has an order $2$ derived subgroup. Hence, the only possibility is that $T_1\cong \Dih(8)$ from which we we deduce that $Z(P)$ has index $4$ in $P$. But $Z(R_c)$ has index $4$ in $R_c$ and as $Q_x\not\le Z(P)$, we have a contradiction. Thus, $Q_x\le Z(S)$ and the claim holds. 

Therefore, 

\textbf{(3)} if $Q_x$ is abelian then $Q_x\le Z(S)$ and $X=\Omega_1(Q_x)$ is normalized by $O^{p'}(\Aut_{\fs}(E))$ for all $E\in\mathcal{E}(\fs)$.\hfill

Since $X\le \Omega_1(Z(S))$, we have that $X$ is centralized by $O^{p'}(\Aut_{\fs}(E))$ for all such $E$. But then $\{1\}\ne X\le \bigcap_{E\in \mathcal{E}(\fs)} C_E(O^{p'}(\Aut_{\fs}(E)))$ and as $\mathcal{E}(\fs)$ is an $\Aut_{\fs}(S)$-invariant set, $\bigcap_{E\in \mathcal{E}(\fs)} C_E(O^{p'}(\Aut_{\fs}(E)))$ is normalized by $\Aut_{\fs}(S)$. By the Frattini argument, $\bigcap_{E\in \mathcal{E}(\fs)} C_E(O^{p'}(\Aut_{\fs}(E)))$ is normalized by $\Aut_{\fs}(E)$ for all $E\in\mathcal{E}(\fs)$ and by \cref{normalinF}, we have that \[\{1\}\ne X\le \bigcap_{E\in \mathcal{E}(\fs)} C_E(O^{p'}(\Aut_{\fs}(E)))\le O_2(\fs)=\{1\},\] a contradiction.
\end{proof}

\begin{proposition}
There is $1\leq i\leq l$ and $y\in \mathfrak{X}(\fs)$ such that $y\in R_i$. In particular, $y\not\in Q_x$.
\end{proposition}
\begin{proof}
By \cref{NotUnique}, there is $t\in \mathfrak{X}(\fs)$ with $Q_t\ne Q_x$. Indeed, by \cref{QTI2}, we have that $Q_x\cap Q_t=\{1\}$ so that $[Q_x, Q_t]=\{1\}$. By \cref{NonAb} we have that $Q_t'\ne \{1\}$. Now, if $[Q_t, R_i]\ne\{1\}$ for some $i\in \{1,\dots, l\}$ then there is $y\in Q_t\cap R_i\cap \Omega_1(Z(S))$. Then by \cref{QTI3}, we have that $y\in \mathfrak{X}(\fs)$ and the result holds. 

On the other hand, since $C_S(\mathcal{C}_i)$ has index at most $2$ in $C_S(R_i)$ by \cref{CompAutos}, if $[Q_t, R_i]=\{1\}$ for all $i$ then we see that $Q_t'\le C_S(E(\fs_x))\cap Q_t=Q_x\cap Q_t=\{1\}$, a contradiction since $Q_t$ is non-abelian by \cref{NonAb}.
\end{proof}

Choose $y\in \mathfrak{X}(\fs)$ as described above and, if necessary, adjust labeling such that $y\in R_1$. We apply \cref{CompDecomp} to $y$ and write $\{\mathcal{T}_i\}_{i\in\{1,\dots, l}$ for subsystems of $O^{p'}(\fs_y)$ which play the analogous role to $\mathcal{C}_i$ for $\fs_x$. We write $T_i$ for the $p$-subgroup of $S$ that $\mathcal{T}_i$ is supported on.

We remark that a further application of \cref{CompAutos} implies that $[Q_x, T_i]\ne\{1\}$ for some $i\in\{1,\dots, l\}$. 

\begin{lemma}\label{KrullRemak}
Suppose that $l>1$. Then, up to permuting indices, we have that $\mathcal{C}_i=\mathcal{T}_i$ for $2\leq i\leq l$, and that $x\in T_1$.
\end{lemma}
\begin{proof}
We note that $C_{\fs}(\langle x, y\rangle)=C_{C_{\fs}(x)}(y)=C_{C_{\fs}(y)}(x)$. By \cref{balance} we have that $E(C_{\fs}(\langle x,y\rangle))\le E(C_{\fs}(x))$. In particular, $E(C_{\fs}(\langle x,y\rangle))\le C_{E(C_{\fs}(x))}(y)=C_{\mathcal{C}_1}(y)\times \mathcal{C}_2\times \dots \mathcal{C}_l$. Since $\mathcal{C}_2\times \dots\times \mathcal{C}_l\le C_{\fs}(\langle x,y\rangle)$, it follows from \cref{QTI1} that $E(C_{\fs}(\langle x,y\rangle))=\mathcal{C}_2\times \dots \mathcal{C}_l$. 

Again, by \cref{balance} we have that $E(C_{\fs}(\langle x,y\rangle))\le E(C_{\fs}(y))$. Let $T_i$ be such that $[Q_x, T_i]\ne \{1\}$ and let $z\in Q_x\cap T_i\cap \Omega_1(Z(S))$, noting that $E(C_{\fs}(x))=E(C_{\fs}(z))$ for any $z\in Q_x\cap Z(S)$ by \cref{QTI3}.  We may as well arrange that $z\in T_1$ and we ascertain that $\mathcal{T}_2\times \dots \mathcal{T}_l\le C_{\fs}(z)$. 

Repeating the proof with $z$ in place of $x$, we have that $E(C_{\fs}(\langle z, y\rangle))=E(C_{\fs}(\langle x,y\rangle))=\mathcal{C}_2\times \dots \mathcal{C}_l$. In particular, $E(C_{\fs}(\langle x,y\rangle))\le C_{E(C_{\fs}(y)}(z)=C_{\mathcal{T}_1}(z)\times \mathcal{T}_2\times \dots \mathcal{T}_l$. As before, it follows that $E(C_{\fs}(\langle x,y\rangle))=\mathcal{T}_2\times \dots \mathcal{T}_l$.

Hence, we have that $E(C_{\fs}(\langle x,y\rangle))=\mathcal{C}_2\times \dots \mathcal{C}_l=\mathcal{T}_2\times \dots \mathcal{T}_l$. Applying \cref{krullschmidt} and permuting the indices if necessary, we have that $\mathcal{C}_i=\mathcal{T}_i$ for $2\leq i\leq l$.
\end{proof}

\begin{lemma}\label{QeqT}
Either 
\begin{enumerate}
    \item $Q_x\cong T_1$; or
    \item $\mathcal{T}_1$ is isomorphic to the $2$-fusion system of $\Alt(6)$, $Q_x\cong \Dih(8)\times 2$ and $C_S(\mathcal{C}_2\times\dots \mathcal{C}_l)=Q_x\times Q_y$. 
\end{enumerate} 
Similarly, either
\begin{enumerate}
    \item $Q_y\cong R_1$; or
    \item $\mathcal{C}_1$ is isomorphic to the $2$-fusion system of $\Alt(6)$, $Q_y\cong \Dih(8)\times 2$ and $C_S(\mathcal{C}_2\times\dots \mathcal{C}_l)=Q_x\times Q_y$.
\end{enumerate}
In all cases, $Q_xQ_y$ has index at most $2$ in $C_S(\mathcal{C}_2\times\dots \mathcal{C}_l)$.
\end{lemma}
\begin{proof}
Set $P:=C_S(\mathcal{C}_2\times\dots \mathcal{C}_l)$. Then either $P=Q_x\times R_1$, or $R_1\cong \Dih(8)$ and $Q_x\times R_1$ has index $2$ in $P$. Similarly, either $P=Q_y\times T_1$, or $T_1\cong \Dih(8)$ and $Q_y\times T_1$ has index $2$ in $P$. 

\underline{\textbf{Case 1:}} $R_1\cong \Dih(8)$ and $Q_x\times R_1$ has index $2$ in $P$.\hfill

Suppose first that $R_1\cong \Dih(8)$ and $Q_x\times R_1$ has index $2$ in $P$. Then it follows from \cref{S/Qx-type} and considering the system $C_{\fs_x}((\mathcal{C}_2\times\dots \mathcal{C}_l))$ that $P/Q_x\cong \Dih(8)\times 2$ and as $Q_x\cap Q_y=\{1\}$, we see that $Q_y$ is isomorphic to a subgroup of $\Dih(8)\times 2$. Since $Q_y$ is non-abelian, we deduce that $Q_y$ is isomorphic to $\Dih(8)$ (so isomorphic to $R_1$), or is isomorphic to $\Dih(8)\times 2$ from which we deduce that $P=Q_x\times Q_y$. Assuming that $T_1\cong \Dih(8)$ and $Q_y\times T_1<P$ gives a similar conclusion for $Q_x$.

\underline{\textbf{Case 2:}} $Q_y\times T_1<P=Q_x\times R_1$ or $Q_x\times R_1<P=Q_y\times T_1$.\hfill

By symmetry, to prove the result we may assume that $P=Q_x\times R_1$. Our aim is to show that $R_1\cong Q_y$. Suppose that $Q_yT_1<P$ so that $T_1\cong \Dih(8)$, and assume first that $P=Q_x\times Q_y$ so that $Q_x\cong \Dih(8)\times 2$. Then, as $Q_x\cap Q_y=\{1\}=Q_x\cap R_1$, we have that $Q_y\cong P/Q_x\cong R_1$, as the result holds. 

Assume now that $Q_x\cong T_1\cong \Dih(8)$ so that $Q_xQ_y<P$. Now, $Q_y$ is isomorphic to a subgroup of $P/Q_x\cong R_1$. Since $Q_y$ is non-abelian, we immediately deduce that $R_1\not\cong \Dih(8)$. Hence, $R_1$ is isomorphic to a Sylow $2$-subgroup of $\PSL_3(2^a)$ or of $\PSp_4(2^a)$ for some $a\geq 2$. 

Since $T_1$ is non-abelian, we must have that $|P/C_P(T)|\geq 4$. Since $Q_xQ_y$ has index $2$ in $P$ and $[P, T_1]=Z(S)\cap T_1=Z(T_1)$ has order $2$, we must have that $[Q_xQ_y, T_1]=Z(T_1)$. Since $[Q_y, T_1]=\{1\}$, we conclude that $Z(T_1)=[Q_x, T_1]\le Z(S)\cap Q_x=Z(Q_x)$, and as $|Z(Q_x)|=2$, we have that $Z(T_1)=Z(Q_x)$. In particular, as $R_1\cap T_1\normaleq S$, $R_1\cap T_1\normaleq T_1$ and since $R_1\cap Q_x=\{1\}$, we must have that $R_1\cap T_1=\{1\}$. Hence, $R_1$ centralizes $T_1$ and since $C_S(\mathcal{T}_1)$ has index at most $2$ in $C_S(T_1)$ by \cref{CompAutos}, we conclude that $R_1'$ centralizes $\mathcal{T}_1$. Then $R_1'\le Z(Q_y)$.

Since $a\geq 2$, $Q_y\cong Q_yQ_x/Q_x$ has index $2$ in $P/Q_x$ and as $R_1Q_x/Q_x\cong R_1$, the structure of $R_1$ implies that $|Z(Q_y)|=|R_1'|$. thus, $R_1'=Z(Q_y)$. Choose $A\in\mathcal{E}(\mathcal{C}_1)$ and let $k$ be an element of order $2^a-1$ in $N_{O^{p'}(\Aut_{\mathcal{C}_1}(A))}(R_1)$. Then $k$ normalizes $Z(Q_y)=R'$ and lifts to $\bar{k}\in \Aut_{\fs_x}(S)\subseteq \Aut_{\fs}(S)$. Hence, this lift also normalizes $Q_x$. Now, $\bar{k}\in N_{\fs}(Z(Q_y))$ and as $E(N_{\fs}(Z(Q_y)))=E(\fs_y)$, we have that $\bar{k}$ normalizes $C_S(E(\fs_y))=Q_y$. Then $\bar{k}$ normalizes $Q_xQ_y\cap R_1$ which has index $2$ in $R_1$. But $k$ normalizes no such subgroups since $a>1$, and we have arrived at a contradiction.

\underline{\textbf{Case 3:}} $P=Q_y\times T_1=Q_x\times R_1$.\hfill

Note that $R_1\cong P/Q_x$ and $Q_y\cong Q_yQ_x/Q_x$ from which we deduce that $|Q_y|\leq |R_1|$. But now $P/T_1\cong Q_y$ and so, upon proving that $T_1\cap R_1=\{1\}$, we will have demonstrated that $P/T_1\cong R_1$ and $Q_y\cong R_1$. 

Assume that $T_1$ is isomorphic to a Sylow $p$-subgroup of $\PSL_3(p^a)$ for some $a\in \N$. Now, $[Q_xQ_y, T_1Q_y]/Q_y=[Q_x, T_1]Q_y/Q_y\cong [Q_x, T_1]$ and since $Q_x\cong Q_xQ_y/Q_y$ is non-abelian (and hence not central in $P/Q_y=T_1Q_y/Q_y$), we deduce that $[Q_x, T_1]=Z(T_1)$. In particular, $Z(T_1)\le Q_x$. Since $R_1\cap T_1\normaleq T_1$ and $R_1\cap Q_x=\{1\}$, we conclude that $R_1\cap T_1=\{1\}$, as desired. 

Assume now that $T_1$ is isomorphic to a Sylow $2$-subgroup of $\PSp_4(2^b)$ for some $b\geq 2$. Write $\bar{P}:=P/Q_y\cong T_1$ and let $\bar{A}\in\mathcal{A}(\bar{P})$. Note that if $\bar{Q_x}\le \bar{A}$, then $\bar{Q_x}=\{1\}$ so that $Q_x'\le Q_y$. But then $Q_x'=\{1\}$, a contradiction since $Q_x$ is non-abelian. Hence, $\bar{Q_x}\not\le \bar{A}$. Since $[Q_x, R_1]=\{1\}$, we have that $\bar{R_1}\cap \bar{A}\le C_{\bar{A}}(\bar{Q_x})=Z(\bar{P})$. In particular, we either have that $|\bar{R_1}\bar{A}/\bar{A}|\geq 4$ or $\bar{R_1}\le Z(\bar{T_1})$. In the latter case, we have that $R_1\le Z(T_1)Q_y$ so that $R_1'\le Q_y$. But $R_1'=Z(R_1)$ and as $R_1\cap T_1\normaleq R_1$, we conclude that $R_1\cap T_1=\{1\}$, as desired. Hence, $|\bar{R_1}\bar{A}/\bar{A}|\geq 4$.

Now, since $\bar{T_1}$ is isomorphic to a Sylow $2$-subgroup of $\PSp_4(2^b)$, we conclude that $|[\bar{R_1}, \bar{A}]|=2^{2b}$ by \cref{PSp4lem}. Since $[Q_y, T_1]=\{1\}$, it follows that $[R_1, T_1]$ has order $2^{2b}$ so that $[R_1, T_1]=Z(T_1)$. But then $[T_1, Q_x]\le Q_x\cap Z(T_1)=\{1\}$ and as $[Q_x, Q_y]=\{1\}$ and $P=Q_y\times T_1$, we deduce that $Q_x\le Z(P)$, a contradiction since $Q_x$ is non-abelian by \cref{NonAb}.
\end{proof}

By \cref{QTI3}, we arrange that $x\in Q_x'$ and $y\in Q_y'$. Moreover, if $p=2$ and $Q_x$ is isomorphic to a Sylow $2$-subgroup of $\PSp_4(2^n)$ for $n\geq 1$, then we arrange that $x=[t, z]$ for $t\in A\setminus Z(Q_x)$ and $z\in B\setminus Z(Q_x)$ where $\mathcal{A}(Q_x)=\{A, B\}$. We make a similar choice for $y$. 

\begin{lemma}\label{CentEq}
We have that $\fs_x=C_{\fs}(x)$ and $\fs_y=C_{\fs}(y)$.
\end{lemma}
\begin{proof}
We prove this only for $x$ but the proof is the same for $y$. We observe that $\fs_x=N_{C_{\fs}(x)}(Q_x)$ and so to prove the lemma, if suffices to show that $Q_x\normaleq C_{\fs}(x)$. Since $Q_x=C_S(E(C_{\fs}(x)))$ is clearly $\Aut_{C_{\fs}(x)}(S)$-invariant, with the aim of applying \cite[Proposition I.4.5]{ako}, we demonstrate that $Q_x\le P$ and $Q_x$ is $\Aut_{C_{\fs}(x)}(P)$-invariant for all $P\in\mathcal{E}(C_{\fs}(x))$. Observe that $\Aut_{C_{\fs}(x)}(P)$ normalizes $C_P(E(C_{\fs}(x)))$ and so we need only show that $Q_x\le P$.

Aiming for a contradiction, assume that $Q_x\not\le P$ for some $P\in\mathcal{E}(C_{\fs}(x))$. We observe that $[Q_x, P]\le [Q_x, S]\le Q_x'$. Moreover, since $P\cap Q_x\normaleq \Aut_{C_{\fs}(x)}(P)$, we observe that $Z(P\cap Q_x) \normaleq \Aut_{C_{\fs}(x)}(P)$. Consider the chain $\{1\}\normaleq Z(P\cap Q_x)\le P\cap Q_x\le P$. From this, we ascertain by \cref{Chain} that $Z(Q_x)\le P$ and that $Q_x$ acts non-trivially on $Z(P\cap Q_x)$. Indeed, $|Q_xP/P|\leq |[Z(Q_x\cap P), Q_x]\langle x\rangle/\langle x\rangle|$ by \cref{SEFF}. 

Suppose first that $|Q_x'|=2$. Then $[Z(Q_x\cap P), Q_x]=Q_x'=\langle x\rangle$, a contradiction. Hence, by \cref{QeqT}, we have that $Q_x\cong T_1$. Suppose that $Q_x$ is isomorphic to a Sylow $p$-subgroup of $\PSL_3(p^a)$ where $p^a>2$. In particular, $Q_x$ is a special group of order $q^3$, where $q=p^a$. If $|Q_x\cap P|>q^2$ then $Z(Q_x\cap P)=Z(Q_x)$ is centralized by $Q_x$, a contradiction. Hence, $|Q_x\cap P|\leq q^2$ so that $|Q_xP/P|\geq q$. But then $|[Z(Q_x\cap P), Q_x]|\leq |Q_x|'=q$ and we deduce that $|Q_xP/P|>|[Z(Q_x\cap P), Q_x]\langle x\rangle/\langle x\rangle|$, a contradiction. Hence, $p=2$ and $Q_x$ is isomorphic to a Sylow $2$-subgroup of $\PSp_4(q)$ where $q=2^a$ for some $a>1$.

Let $\mathcal{A}(Q_x)=\{A, B\}$. If $A\cap P>Z(Q_x)<B\cap P$, then $Z(Q_x\cap P)=Z(Q_x)$, a contradiction. Without loss of generality, assume that $A\cap P=Z(Q_x)$ so that, writing $x=[t,z]$, we have that $t\not\in P$. Then $|AP/P|=q$ and $\Omega_1(Z(P\cap Q_x))\le B$. Since $A$ centralizes $Z(Q_x)$ which has index $q$ in $B$, we conclude that $P\cap Q_x=Z(P\cap Q_x)=\Omega_1(Z(P\cap Q_x))=B$ and $z\in P$. Now, $A$ centralizes an index $q$ subgroup of $B$ and by \cref{SEFF} we conclude that $S=AP$ and $B/C_B(O^{2'}(\Out_{C_{\fs}(x)}(P)))$ is a natural module for $O^{2'}(\Out_{C_{\fs}(x)}(P))\cong \SL_2(q)$. Now, $Z(Q_x)=[B, t]C_B(O^{2'}(\Out_{C_{\fs}(x)}(P)))$ and $Z(Q_x)/C_B(O^{2'}(\Out_{C_{\fs}(x)}(P)))=C_{B/C_B(O^{2'}(\Out_{C_{\fs}(x)}(P)))}(t)$. But $[t, z]=x\in C_B(O^{2'}(\Out_{C_{\fs}(x)}(P)))$ from which we deduce that $z\in Z(Q_x)$, a contradiction. Hence, no such $P$ exists and $Q_x\normaleq C_{\fs}(x)$.
\end{proof}

\begin{lemma}\label{CompDecomp2}
We have that $O^{p'}(\fs_x)=SE(\fs_x)$ and $O^{p'}(\fs_y)=SE(\fs_y)$. 
\end{lemma}
\begin{proof}
We will prove this for $x$, as the proof of $y$ is the same. By \cref{CompDecomp}, we may as well assume that $p=2$ and $\fs_x$ has a component isomorphic to the $2$-fusion category of $\Alt(6)$. Write $\fs_x^0:=SE(\fs_y)$.

Suppose that for all $E\in\mathcal{E}(\fs_x)$, we have that $O^{2'}(\Aut_{\fs}(E))\le \Aut_{\fs_x^0}(E)$. Since $\fs_x^0$ is clearly $\Aut_{\fs_x}(S)$-invariant, and $\fs_x=\langle O^{2'}(\Aut_{\fs_x}(F)), \Aut_{\fs_x}(S) \mid F\in\mathcal{E}(\fs_x)\rangle$, we have that $\fs_x^0$ is $\fs_x$-invariant by \cite[Proposition I.6.4]{ako}. By that same result, we have that $\Aut_{\fs_x^0}(P)\normaleq \Aut_{\fs_x}(P)$ for all $P\le S$ and since $\Aut_S(P)\le \Aut_{\fs^0_x}(P)$, we have that $\fs_x^0$ has index prime to $p$ in $\fs_x$. Since $\fs_x^0\subseteq O^{2'}(\fs_x)$ by \cref{p'genfit}, the result holds.

Hence, there is $E\in\mathcal{E}(\fs_x)$ with $O^{2'}(\Aut_{\fs}(E))\not\le \Aut_{\fs_x^0}(E)$. Since $Q_x\normaleq \fs_x$, we have that $Q_x\le E$. If $R\not\le E$, then by \cref{Radical} and \cref{EssentialsIn}, there is a unique $i\in \{1,\dots, l\}$ with $R_i\not\le E$ and $Q_x\le C_E(O^2(O^{2'}(\Aut_{\fs_x}(E))))$. Indeed, by the proof of \cref{EssentialsIn} we have that $O^{2'}(\Aut_{\fs_x}(E))=\langle \Inn(S), \hat{\alpha}\rangle$, where $\hat{\alpha}$ is the lift to $E$ of some prescribed element in $\Aut_{\mathcal{C}_i}(E)$. Since $\mathcal{C}_i\subseteq E(\fs_x)\subseteq \fs_x^0$, we have that $O^{2'}(\Aut_{\fs}(E))\le \Aut_{\fs_x^0}(E)$, a contradiction. 

Hence, $Q_xR\le E$. Now, by \cref{QeqT}, we either have that $Q_x'=Z(Q_x)$, or $Q_x\cong \Dih(8)\times 2$. In the former case, since $Z(Q_xR)=Z(Q_x)Z(R)=Q_x'R'\le E'$ and $[S, E]\le S'\cap Q_xR\le Z(S)\cap Q_xR\le Z(Q_xR)$, we have that $[S, E]=E'$. Then \cref{Chain} delivers a contradiction. Hence, we have that $Q_x\cong \Dih(8)\times 2$. We follow the proof from before, and we arrive at the same contradiction unless $[S, E]\cap Q_x=Z(Q_x)$. Now, $Z(Q_x)\normaleq \fs_x$ and $R$ is strongly closed in $\fs_x$ from which we deduce that $O^{2'}(\Aut_{\fs_x}(E))$ normalizes $Z(Q_x)R'$ and $Z(Q_x)R'=S'\le Z(S)$. Then $S$ centralizes the chain $\{1\}\normaleq Z(Q_x)R'\normaleq E$ and once again, \cref{Chain} provides a contradiction. 
\end{proof}

\begin{proposition}
We have that $\fs=\langle O^{p'}(C_{\fs}(x)), O^{p'}(C_{\fs}(y))\rangle_S$.
\end{proposition}
\begin{proof}
We first claim that $\fs=\langle \Aut_{\fs}(S), O^{p'}(C_{\fs}(x)), O^{p'}(C_{\fs}(y))\rangle_S$. By \cref{EssentialGeneration}, to prove the claim we need to show that for each $E\in\mathcal{E}(\fs)$, we have that $O^{p'}(\Aut_{\fs}(E))$ centralizes $x$ or $y$.

By \cref{GeneralEssentialStructure}, $Z(E)$ contains all the non-central $\Aut_{\fs}(E)$-chief factors inside of $E$. Suppose that $Q_x\le E$. Then $x\le Q_x'\le E'$ and since $O^{p'}(\Aut_{\fs}(E))$ centralizes $E'$, we have that $O^{p'}(\Aut_{\fs}(E))$ centralizes $x$. Suppose now that $Q_x\not\le E$ so that by \cref{Radical}, $S=Q_xE$. If $Q_y\not\le E$ then $S=Q_yE=Q_xE$. But then $[S, Z(E)]=[Q_x, E]=[Q_y, E]\le Q_x\cap Q_y=\{1\}$, a contradiction since $Z(E)$ contains all non-central chief factors for $O^{p'}(\Aut_{\fs}(E))$. Hence, $Q_y\le E$ and we conclude that $y\le E'$ is centralized by $O^{p'}(\Aut_{\fs}(E))$. Hence, 

\textbf{(1)} $\fs=\langle \Aut_{\fs}(S), O^{p'}(C_{\fs}(x)), O^{p'}(C_{\fs}(y))\rangle_S$.\hfill

Set $\fs_0:=\langle O^{p'}(C_{\fs}(x)), O^{p'}(C_{\fs}(y))\rangle_S$. We prove that $\fs=\fs_0$ in two steps. First, we show that $\langle O^{p'}(\Aut_{\fs}(P))\mid P\in\fs^c\rangle\le \fs_0$ so that $\fs_0$ has index prime to $p$ in $\fs$. Then we show that $\fs_0$ is saturated, at which point the fact that $\fs=O^{p'}(\fs)$ yields that $\fs=\fs_0$.

We claim first that $C_{\fs}(x\alpha)=C_{\fs}(x)^{\alpha}$ for $\alpha\in \Aut_{\fs}(S)$. To this end, let $\beta\in\Hom_{C_{\fs}(x)^\alpha}(\langle x\alpha\rangle, S)$. Then $\beta=\alpha^{-1}\circ \gamma \circ \alpha$ where $x\gamma=x$. Then $(x\alpha)\beta=(x\gamma)\alpha=x\alpha$ and we deduce that $(x\alpha)^{C_{\fs}(x)^\alpha}=\{x\alpha\}$. Hence, by \cite[Lemma I.4.2]{ako} we have that $C_{\fs}(x)^{\alpha}\subseteq C_{\fs}(x\alpha)$. A similar argument yields that $C_{\fs}(x\alpha)^{\alpha^{-1}}\subseteq C_{\fs}(x)$ from which we deduce that $C_{\fs}(x\alpha)\subseteq C_{\fs}(x)^{\alpha}$ so that $C_{\fs}(x\alpha)= C_{\fs}(x)^{\alpha}$ and the claim holds. Similar arguments, along with an application of \cite[Theorem I.7.7]{ako}, imply that

\textbf{(2)} $O^{p'}(C_{\fs}(x))^\alpha=O^{p'}(C_{\fs}(x\alpha))$ for all $\alpha\in \Aut_{\fs}(S)$.\hfill 

We now aim to demonstrate that $O^{p'}(C_{\fs}(x\alpha))\subseteq \fs_0$ for all $\alpha\in\Aut_{\fs}(S)$. We first claim that $\fs_{x\alpha}=C_{\fs}(x\alpha)$ for all $\alpha\in\Aut_{\fs}(S)$. Since $x\alpha\in\mathfrak{X}(\fs)$, by \cref{CentEq} we need only show that $x\alpha$ can be chosen appropriately in $Q_{x\alpha}$. We observe that $Q_x\alpha=Q_{x\alpha}$. In particular, $Q_x\cong Q_{x\alpha}$ and $x\alpha\in Q_{x\alpha}'$. If $Q_x$ is isomorphic to a Sylow $2$-subgroup of $\PSp_4(2^n)$ then for $\mathcal{A}(Q_x)=\{A, B\}$, we have that $\{A\alpha, B\alpha\}=\mathcal{A}(Q_{x\alpha})$ and so $x\alpha$ has the same properties as $x$ required for the proof of \cref{CentEq}. Thus, $\fs_{x\alpha}=C_{\fs}(x\alpha)$ so that 

\textbf{(3)} $O^{p'}(C_{\fs}(x\alpha))=O^{p'}(\fs_{x\alpha})$ for all $\alpha\in \Aut_{\fs}(S)$.\hfill 

We aim to show that $E(\fs_{x\alpha})\subseteq \fs_0$. By \cref{QTI2}, we may assume that $Q_{x\alpha}\cap Q_x=Q_{x\alpha}\cap Q_y=\{1\}$. Write $\Comp(C_{\fs}(x\alpha))=\{\mathcal{K}_1,\dots, \mathcal{K}_l\}$ supported on $K_1,\dots, K_l$ respectively. If $[Q_x, K_i]=\{1\}$ for all $i\in \{1,\dots, l\}$ then since $C_S(\mathcal{K}_i)$ has index at most $2$ in $C_S(K_i)$ by \cref{CompAutos}, we have that $Q_x'\le Q_{x\alpha}$. Since $Q_x$ is non-abelian by \cref{NonAb}, we have a contradiction. Hence, there is $i$ with $z\in [Q_x, K_i]\cap \Omega_1(Z(S))\le Q_x\cap \Omega_1(Z(S))$. By \cref{QTI1}, we have that $\mathcal{K}_j\subseteq E(\fs_z)$ for all $j\ne i$. Then by \cref{QTI3}, we see that $\mathcal{K}_j\subseteq E(\fs_x)\subseteq \fs_0$. By a similar argument with $Q_y$, we see that $E(\fs_{x\alpha})\subseteq \fs_0$ unless $[Q_y, K_i]\ne \{1\}=[Q_y, K_j]$ for all $j\ne i$. Since $Q_x\cap Q_y=\{1\}$, we must have that $\mathcal{K}_i$ is isomorphic to the $2$-fusion category of $\PSp_4(q)$ for $q>2$ and $|[Q_x, K_i]|=|[Q_y, K_i]|=q$. Replaying \cref{QeqT} with $Q_x, Q_{x\alpha}$ in place of $Q_x, Q_y$ reveals that $Q_x\cong K_i$. Since $Q_x\cap C_S(\mathcal{K}_i)=Q_x\cap Q_{x\alpha}=\{1\}$, we conclude that we conclude that $[Q_x, K_i]=[Q_xC_S(\mathcal{K})]=[S, K_i]$ has order $q^2$, a contradiction. Hence, $E(\fs_{x\alpha})\subseteq \fs_0$.

Since $\fs_0$ is supported on $S$, we see that $SE(\fs_{x\alpha})\subseteq \fs_0$, and by \cref{CompDecomp2} we deduce that $O^{p'}(\fs_{x\alpha})\subseteq \fs_0$. We have shown that

\textbf{(4)} $O^{p'}(C_{\fs}(x))^{\alpha}=O^{p'}(C_{\fs}(x\alpha))=O^{p'}(\fs_{x\alpha})\subseteq \fs_0$ for all $\alpha\in\Aut_{\fs}(S)$.\hfill

A similar argument for $y$ in place of $x$ reveals that $\fs_0$ satisfies the invariance condition of normality of fusion systems. Moreover, $\fs=\langle \Aut_{\fs}(S), \fs_0\rangle_S$ and so by \cite[Proposition I.6.4]{ako}, we have that $\fs_0$ is $\fs$-invariant. Again by \cite[Proposition I.6.4]{ako}, for each $\fs$-centric subgroup of $S$, since $\Aut_S(P)\subseteq \fs_0$, we have that $O^{p'}(\Aut_{\fs}(P))\le \Aut_{\mathcal{E}}(P)$. Hence, by \cite[Lemma I.7.6(a)]{ako}, 

\textbf{(5)} $\fs_0$ has index prime to $p$ in $\fs$, and $\fs^c=\fs_0^c$.\hfill

It remains to show that $\fs_0$ is saturated. We wish to apply \cite[Lemma I.7.6(b)]{ako}, and we do so ``sequentially" through a series of subsystems. We set $\mathcal{E}_x^0:=O^{p'}(\fs_x)$ and $\mathcal{E}_y^0:=O^{p'}(\fs_y)$. Define $\mathcal{E}_x^i$ to be the saturated fusion subsystem of $\fs_x$ containing $O^{p'}(\fs_x)$ with $\Aut_{\mathcal{E}_x^i}(S)=\Aut_{\fs_{i-1}}(S)$ where $\fs_{i-1}=\langle \mathcal{E}_x^{i-1}, \mathcal{E}_y^{i-1}\rangle_S$. Define $\mathcal{E}_y^i$ similarly and observe that $\mathcal{E}_x^i$ and $\mathcal{E}_y^i$ are saturated subsystems of index prime to $p$ in $\fs_x$ and $\fs_y$ respectively, while $\fs_i$ is a (not necessarily) saturated subsystem of $\fs$ of index prime to $p$. Indeed, since $\fs=O^{p'}(\fs)$, if $\fs_i$ is saturated then $\fs_i=\fs$. Let $r\in \N_0$ be minimal such that for $Q\in \fs^c$, $\Hom_{\fs_r}(Q, S)\cap \Hom_{\fs_x}(Q, S)\subseteq \Hom_{\mathcal{E}_x^r}(Q, S)$ and $\Hom_{\fs_r}(Q, S)\cap \Hom_{\fs_y}(Q, S)\subseteq \Hom_{\mathcal{E}_y^r}(Q, S)$.

To see this is well defined, let $Q\in\fs^c$ and $\alpha\in \Hom_{\fs_i}(Q, S)\cap \Hom_{\fs_x}(Q, S)$. Since $\mathcal{E}_x^{i-1}$ has index prime to $p$ in $\fs_x$, we have by \cite[Lemma I.7.6]{ako} that $\alpha=\phi\circ \theta$ for some $\phi\in \Hom_{\mathcal{E}_x^{i-1}}(Q, S)$ and $\theta\in \Aut_{\fs_x}(S)$. Indeed, we recognize $\phi=\hat{\phi}\circ i$ where $\hat{\phi}\in\Hom_{\mathcal{E}_x^{i-1}}(Q, \phi(Q))$ and $i$ is the inclusion map from $\phi(Q)$ to $S$. Then $\hat{\phi}^{-1}\circ \alpha=\theta|_{\phi(Q)}$ and we deduce that $\theta|_{\phi(Q)}\in \Hom_{\fs_i}(\phi(Q), S)$. But $\theta\in \Aut_{\fs_x}(S)$ and as $\Aut_{\fs_i}(S)=\Aut_{\mathcal{E}_x^{i+1}}(S)$, we deduce by \cite[Lemma I.5.6]{ako} that $\theta\in \Aut_{\mathcal{E}_x^{i+1}}(S)$. But then $\phi\circ \theta\in \Hom_{\mathcal{E}_x^{i+1}}(Q, S)$ and so $\alpha\in \Hom_{\mathcal{E}_x^{i+1}}(Q, S)$. Finally, by finiteness of $\fs_x$, we must have that $\mathcal{E}_x^{i}$ stabilizes whenever $i\geq a$ for some $a$ and so there is $r$ such that $\Hom_{\fs_r}(Q, S)\cap \Hom_{\fs_x}(Q, S)\subseteq \Hom_{\mathcal{E}_x^r}(Q, S)$.

Since $\fs_0\subseteq \fs_{r-1}\subseteq \fs_r\subseteq \fs=\langle \fs_0, \Aut_{\fs}(S)\rangle_S$, we have that $\fs_r=\langle \Aut_{\fs_r}(S), \fs_{r-1}\rangle_S$. Let $\gamma\in \Aut_{\fs_r}(S)$ with $\gamma\not\in \fs_{r-1}$. Then $\gamma=\gamma_1\circ \dots \circ \gamma_l$ where $\gamma_1\in \Aut_{\mathcal{E}_x^r}(S)$. Then, we may as well assume that $\gamma_{2j}\in\Aut_{\mathcal{E}_y^r}(S)$ and $\gamma_{2j+1}\in\Aut_{\mathcal{E}_x^r}(S)$ for $j\in\{1,\dots, \lfloor l/2 \rfloor\}$. But $\Aut_{\mathcal{E}_x^r}(S)=\Aut_{\fs_{r-1}}(S)$ and we conclude that $\gamma\in \Aut_{\fs_{r-1}}(S)$. Hence, $\fs_r=\fs_{r-1}$ and applying a similar logic, we see that $\fs_0=\fs_r$.

Note that every $\fs$-centric subgroup contains $Z(S)$ and so is normal in $S$. In particular, since $O^{p'}(\fs_x)$ is saturated each $\fs$-centric subgroup is also fully $O^{p'}(\fs_x)$-centralized, and so is $O^{p'}(\fs_x)$-centric. In a similar manner, we see that every $O^{p'}(\fs_x)$-centric subgroup is also $\fs$-centric so that $O^{p'}(\fs_x)^c=\fs^c=\fs_0^c$. Since $\fs_i$ is generated by $O^{p'}(\fs_x)$ and $O^{p'}(\fs_y)$, which are $\fs^c$ generated since they are saturated, we see that $\fs_0$ is also $\fs^c$-generated.

With the intent of applying \cite[Lemma I.7.6(b)]{ako}, aiming for a contradiction, let $P\in \fs_0^c$ with $|P|$ maximal such that there is $\psi\in \Hom_{\fs}(P, S)$ and $Q\in \fs^c$ with $\psi|_Q\in \fs_0$, but $\psi\not\in \fs$. Since $\fs=\langle \fs_0, \Aut_{\fs}(S)\rangle_S$, we have that $P=S$. We then choose $Q$ with $|Q|$ maximal subject to the previous conditions.

Suppose first that $\mathfrak{L}(\fs)=1$. By \cref{QeqT}, we have that $S=Q_x\times Q_y$, unless $\mathcal{C}_1\cong \mathcal{T}_1$ is isomorphic to the $2$-fusion category of $\Alt(6)$ and $Q_x\cong Q_y\cong \Dih(8)$. However, in this case we have that $|S|=2^7$ and \cite{bobsmall} yields a contradiction since $\fs$ is simple. Hence, $S=Q_x\times Q_y$ and the Krull--Schmidt-Remak theorem for groups implies that $\alpha$ permutes $Q_xZ(Q_y)$ and $Q_yZ(Q_x)$. In particular, unless $p$ is odd and $Q_x\cong Q_y$, we have that $\alpha$ normalizes $Q_xZ(Q_y)$ and $Q_yZ(Q_x)$. If $\alpha$ normalizes $Q_xZ(Q_y)$, then $\alpha$ normalizes $(Q_xZ(Q_y))'=Q_x'$ and so $\alpha\in N_{\fs}(Q_x)=\fs_x$. 

Assume that $\mathfrak{L}(\fs)=\{1\}$ and $\alpha$ does not normalize $Z(Q_x)$. Thus, $p$ is odd and $Q_x\cong Q_y$. Observe that $C_S(Q_x)=Z(Q_x)Q_y=Z(Q_x)R$ so that $Z(R)=R'=C_S(Q_x)'=Q_y'=Z(Q_y)$. Since $\alpha|_Q\in \fs_0$, we may decompose $\alpha|_Q$ as $\beta_1\circ \dots \circ \beta_n$, where $\beta_i\in O^{p'}(\fs_x)$ or $O^{p'}(\fs_y)$. Since $O^{p'}(\fs_x)$ and $O^{p'}(\fs_y)$ are saturated we can further employ the Alperin--Goldschmidt theorem to decompose each $\beta_i$. It is clear from the description of $O^{p'}(\fs_x)$ that each element of $\Aut_{O^{p'}(\fs_x)}(S)$ normalizes each essential subgroup of $O^{p'}(\fs_x)$. Furthermore, since $p$ is odd, for each essential subgroup $E\in\mathcal{E}(\fs_x)$, we have that $\Aut_{\fs}(E)$ centralizes $Z(Q_x)$ and preserves $(E\cap R)[E, \Aut_{\fs}(E)]$. Following the decomposition of $\alpha|_Q$, and using that $Z(Q_y)=Z(R)=R\cap E \cap Z(S)$, we must have that $\alpha|_Q$ normalizes $Z(Q_x)$ and $Z(Q_y)$, a contradiction.

Suppose now that $\mathfrak{L}(\fs)>1$. We observe that $Z(S)\le Q$ so that $Z(R)\le Q$. Indeed, since $\alpha|_Q\in\fs_0$, and $R_2\times\dots \times R_l$ is strongly closed in $\fs_0$ by \cref{KrullRemak}, we see that $\alpha|_Q$ normalizes $Z:=Z(R_2)\times \dots \times Z(R_l)$. Hence, $\alpha\in N_{\fs}(Z)$ and as $\mathfrak{L}(\fs)>1$, $Z\ne\{1\}$. It follows from \cref{QTI1} that $\mathcal{C}_1, \mathcal{T}_1\in \Comp(N_{\fs}(Z))$. Moreover, by minimality of $\fs$, $E(N_{\fs}(Z))$ is determined by the \hyperlink{MainThm}{Main Theorem} and since $S/R_1T_1Z$ is abelian, we must have that $E(N_{\fs}(Z))=\mathcal{C}_1\times \mathcal{T}_1$ and $\Comp(N_{\fs}(A))=\{\mathcal{C}_1, \mathcal{T}_1\}$. Then $\alpha$ acts on $\Comp(N_{\fs}(Z))$ and and since $S$ is Sylow in $N_{\fs}(A)$, we must have that either $\alpha$ normalizes $Q_x$ and $Q_y$; or $p$ is odd, $\mathcal{C}_1\cong \mathcal{T}_1$, $\alpha^2\in N_{\fs}(Q_x)$ and $\alpha^2\in N_{\fs}(Q_y)$. In the latter case, a similar argument as in the $\mathfrak{L}(\fs)=1$ case yields that $\alpha$ normalizes $Z(Q_x)$ and $Z(Q_y)$, from which we deduce that $\alpha$ normalizes $Q_x$ and $Q_y$.

Therefore, we proceed under the assumption that $\alpha$ normalizes both $Q_x$ and $Q_y$. Now, $\alpha|_Q\in \fs_0$ and $\alpha$ normalizes $Q_x$ from which we conclude that $\alpha|_Q\in \Hom_{\mathcal{E}_x^r}(Q, S)$. Since $\mathcal{E}_x^r$ is saturated, a final application of \cite[Lemma I.7.6(b)]{ako} implies that $\alpha\in \mathcal{E}_x^r\subseteq \fs_0$, a contradiction.
\end{proof}

\begin{lemma}\label{CentGeneration1}
We have that $S=Q_x\times Q_y$, $Q_x\cong T$ and $Q_y\cong R$.
\end{lemma}
\begin{proof}
We have that $\mathcal{C}_2\times \dots\times \mathcal{C}_n$ is a normal subsystem of both $O^{p'}(C_{\fs}(x))$ and $O^{p'}(C_{\fs}(y))$ by \cref{CompDecomp2}, and as $\fs=\langle O^{p'}(C_{\fs}(x)), O^{p'}(C_{\fs}(y))\rangle_S$, we have that $\mathcal{C}_2\times \dots\times \mathcal{C}_n\normaleq \fs$, a contradiction since $\fs$ is simple. Hence, both $E(\fs_x)$ and $E(\fs_y)$ are simple, $R=R_1$ and $T=T_1$.

Aiming for a contradiction, we assume that $Q_x\not\cong T$. Hence, by \cref{QeqT} we must have that $T\cong \Dih(8)$ and $Q_x\cong \Dih(8)\times 2$. In particular, we see that $S=Q_x\times Q_y$ and $|S/Q_yT|=2$. We observe that by a transfer argument \cite[Theorem TL]{transfer}, $Q_yT$ is not strongly closed in $\fs$. If $|Q_y|\leq 2^4$ then $|S|\leq 2^8$ and we apply \cite{bobsmall} to obtain a contradiction. Hence, applying \cref{QeqT} we may assume that $Q_y\cong R$ has order strictly larger than $2^4$.

We observe that both $Q_x$ and $R$ have exactly two maximal elementary abelian subgroups. It follows from \cref{Radical} that $\fs$ has exactly four essential subgroups: two which contain $Q_x$ and intersect $R$ in a maximal elementary abelian subgroup, and two which contain $R$ and intersect $Q_x$ in an elementary abelian subgroup. Moreover, if $Q_x\not\le E\in\mathcal{E}(\fs)$ then $R\le E$, $S=Q_xE$ and $|S/E|=2$. On the other hand, if $R\not\le E\in\mathcal{E}(\fs)$, then $Q_x\le E$, $S=RE$ and $|S/E|>2$. Since the image of an essential subgroup of $\fs$ under an element of $\Aut_{\fs}(S)$ is essential, and each essential is normal in $S$, we conclude that every essential subgroup of $\fs$ is $\Aut_{\fs}(S)$-invariant. Set $E_1, E_2\in\mathcal{E}(\fs)$ such that $|S/E_1|=|S/E_2|=2$ and let $E_3, E_4$ be the remaining essential subgroups.

Now, $Q_y\le E_1\cap E_2$ so that $y\in Q_y'\le E_1'\cap E_2'$. Since $O^{2'}(\Aut_{\fs}(E))$ centralizes $E'$ for all $E\in\mathcal{E}(\fs)$, we have that $O^{2'}(\Aut_{\fs}(E_i))\subseteq C_{\fs}(y)$ so that $O^{2'}(\Aut_{\fs}(E_i))$ normalizes $Q_y$ and $T\cap E_i$ for $i\in \{1,2\}$. Hence, $\Aut_{\fs}(E_i)$ normalizes $E_i\cap Q_yT=Q_y(E_i\cap T)$ by the Dedekind modular law. We note that $[Q_x, E_i]=Q_x'$ has order $2$ so that $O^{2'}(\Out_{\fs}(E_i))\cong \Sym(3)$, $E_i=[E_i, O^{2'}(\Out_{\fs}(E_i))]\times C_{E_i}(O^{2'}(\Out_{\fs}(E_i)))$ and $[E_i, O^{2'}(\Out_{\fs}(E_i))]$ has order $4$ for $i\in \{1,2\}$. Since $[Q_x, Q_y]=\{1\}$ and $Q_y$ is normalized by $O^{2'}(\Aut_{\fs}(E_i))$ for $i\in \{1,2\}$, we must have that $Q_y\le C_{E_1}(O^{2'}(\Out_{\fs}(E_1)))\cap C_{E_2}(O^{2'}(\Out_{\fs}(E_2)))$. 

We observe that $Q_x\le E_3\cap E_4$ and that $O^{2'}(\Aut_{\fs}(E_i))$ normalizes $Q_x$ for $i\in\{3,4\}$. Since $E_3\not\ge Q_y\not\le E_4$, we conclude that $Q_x\le C_{E_3}(O^{2'}(\Out_{\fs}(E_3)))\cap C_{E_4}(O^{2'}(\Out_{\fs}(E_4)))$. Since $|R|>2^4$ it follows that $\fs_x/Q_x$ is simple. Indeed, $\fs_x/Q_x$ has exactly two essential subgroups: $E_3/Q_x$ and $E_4/Q_x$. It follows that $(C_{E_3}(O^{2'}(\Out_{\fs}(E_3)))\cap C_{E_4}(O^{2'}(\Out_{\fs}(E_4))))/Q_x$ is trivial and so $Q_x=C_{E_3}(O^{2'}(\Out_{\fs}(E_3)))\cap C_{E_4}(O^{2'}(\Out_{\fs}(E_4)))$ is $\Aut_{\fs}(S)$-invariant.

Assume that $Q_y<Y:=C_{E_1}(O^{2'}(\Out_{\fs}(E_1)))\cap C_{E_2}(O^{2'}(\Out_{\fs}(E_2)))$. Since $E_1, E_2$ are $\Aut_{\fs}(S)$-invariant, so too is $Y$. Moreover, since $S=Q_x\times Q_y$, we deduce that $Q_x\cap Y\ne \{1\}$. But then $\{1\}\ne Z(S)\cap Q_x\cap Y\le C_{E_i}(O^{2'}(\Out_{\fs}(E_i)))$ for all $i\in\{1,\dots, 4\}$ and as $Z(S)\cap Q_x\cap Y$ is $\Aut_{\fs}(S)$-invariant, by  \cite[Proposition I.4.5]{ako} we have that $O_2(\fs)\ne\{1\}$, a contradiction. Hence, $Q_y=C_{E_1}(O^{2'}(\Out_{\fs}(E_1)))\cap C_{E_2}(O^{2'}(\Out_{\fs}(E_2)))$.

Since $E_1$ and $E_2$ are $\Aut_{\fs}(S)$-invariant, $Q_y$ is $\Aut_{\fs}(S)$-invariant and $N_{\fs}(S)\subset N_{\fs}(Q_y)$. But $\mathcal{T}_1=E(N_{\fs}(Q_y))$ and it follows that $\Aut_{\fs}(S)$ normalizes $Q_y$ and $T$, and so normalizes $Q_yT$. Now, $Q_y$ centralizes a subgroup of $E_3$ of index $|S/E_3|$ and we deduce by \cref{SEFF} that $E_3/C_{E_3}(O^{2'}(\Out_{\fs}(E_3)))$ is a natural module for $O^{2'}(\Out_{\fs}(E_3))\cong \SL_2(2^n)$ for $n>1$. Since $S=Q_x\times Q_y$ and $Q_x\le C_{E_3}(O^{2'}(\Out_{\fs}(E_3)))$, we must have that $E_3=(E_3\cap Q_y)C_{E_3}(O^{2'}(\Out_{\fs}(E_3)))$.

Now, $[E_3, Q_y]=Q_y'$ which has order $2^n$ or $2^{2n}$ where $Q_y$ is isomorphic to a Sylow $2$-subgroup of $\PSL_3(2^n)$ or $\PSp_4(2^n)$ respectively. Either way, set $V:=Q_y'\cap C_{E_3}(O^{2'}(\Out_{\fs}(E_3)))$ from which it follows that $[E_3, O^{2'}(\Out_{\fs}(E_3))]V/V$ is a natural module for $O^{2'}(\Out_{\fs}(E_3))$. Let $K$ be a complement to $\Aut_S(E_3)$ in $N_{O^{2'}(\Aut_{\fs}(E_3))}(\Aut_S(E_3))$ so that $K$ is cyclic of order $2^n-1$. Then $[E_3, O^{2'}(\Out_{\fs}(E_3))]V=[E_3, K]V$ and since $Q_y$ is $\Aut_{\fs}(S)$-invariant, we conclude that $[E_3, O^{2'}(\Out_{\fs}(E_3))]\le [E_3, K]V=[Q_y\cap E_3, K]V\le Q_y\cap E_3$. Hence, $Q_y\cap E_3$ is normalized by $O^{2'}(\Aut_{\fs}(E_3))$. Moreover, $[T, K]\le [E_3, K]\le Q_y$ and as $T$ is $\Aut_{\fs}(S)$-invariant, we deduce that $[T, K]\le T\cap Q_y=\{1\}$ and $T\le C_{E_3}(O^{2'}(\Out_{\fs}(E_3)))$ and so $T$ is normalized by $O^{2'}(\Aut_{\fs}(E_3))$. Thus $E_3\cap Q_yT=T(E_3\cap Q_y)$ is normalized by $O^{2'}(\Aut_{\fs}(E_3))$. By a similar argument we conclude that $T$, $Q_y\cap E_4$ and $E_4\cap TQ_y$ are normalized by $O^{2'}(\Aut_{\fs}(E_4))$. Hence, by the Alperin--Goldschmidt theorem, $TQ_y$ is strongly closed in $\fs$, and we have a contradiction.

Hence, $Q_x\cong T$. The proof that $Q_y\cong R$ is similar. Now, if $Q_xQ_y<S$ then we must have that $S\ne Q_x\times R$ and so $Q_y\cong R\cong \Dih(8)$ and $Q_xQ_y$ has index $2$ in $S$. Similar, we conclude that $Q_x\cong T\cong \Dih(8)$ and so $|S|=2^7$. Then \cite{bobsmall} yields a contradiction since $\fs$ is simple. Hence, $S\cong Q_x\times Q_y$.
\end{proof}

\begin{lemma}\label{CentGeneration}
We have that $Q_x=T$ and $Q_y=R$.
\end{lemma}
\begin{proof}
We prove that $Q_x=T$. The proof that $Q_y=R$ is similar. We assume first that $\mathcal{T}_1$ is not isomorphic to the $2$-fusion category of $\Alt(6)$. Hence, $S=Q_y\times T=Q_y\times Q_x$, $C_S(Q_y)=Z(Q_y)\times T$ and as $Q_x\cong T$, we see that $Z(Q_x)=Q_x'=C_S(Q_y)'=T'=Z(T)$.

Suppose there is $Z\le Z(T)$ and $K^*\le C_{\Aut_{\mathcal{T}_1}(T)}(Z)$ with $T=[T, K^*]Z$. Since $\mathcal{T}_1$ centralizes $Q_y$, $K^*$ lifts to a group of automorphisms of $S$. Let $K$ be the lift of $K^*$ to $\Aut_{\fs}(S)$. Then $K$ centralizes $Z$ and by \cref{QTI3}, we have that $K$ normalizes $C_S(E(C_{\fs}(Z)))=C_S(E(\fs_x))=Q_x$. But then, \[T=[K, T]Z=[K, S]Z=[K, Q_xQ_y]Z=[K, Q_x]Z\le Q_x\] and we conclude that $T=Q_x$. We now show that such a $Z$ and $K^*$ exists.

Suppose that $\mathcal{T}_1$ is isomorphic to the $p$-fusion category of $\PSL_3(q)$ where $q=p^a>2$. Then we may take $Z=Z(T)$ and $K^*$ to be a Hall $p$'-subgroup of $C_{\Aut_{\mathcal{T}_1}(T)}(Z(T))$. Suppose that $\mathcal{T}_1$ is isomorphic to the $2$-fusion category of $\PSp_4(2^a)$ where $a>1$. Choose $A\in\mathcal{A}(T)$ so that $\Aut_{\mathcal{T}_1}(A)\cong \GL_2(2^a)$. Choose $Z=C_A(O^{2'}(\Aut_{\mathcal{T}_1}(A)))$ and $K^*$ to be the lift to $\Aut_{\mathcal{T}_1}(T)$ of a Hall $2'$-subgroup of $N_{O^{2'}(\Aut_{\mathcal{T}_1}(A))}(\Aut_T(A))$. Finally, suppose that $p$ is odd and $T\cong p^{1+2}_+$. By \cite{RV1+2}, a Hall $p'$-subgroup of $C_{\Aut_{\mathcal{T}_1}(T)}(Z(T))$ is non-trivial. If $|C_T(K^*)|=p^2$, then an easy application of the three subgroups lemma implies that $K^*$ centralizes $T/C_T(K^*)$. Then coprime action yields a contradiction. Hence, we may take $Z=Z(T)$ and $K^*$ a Hall $p'$-subgroup of $C_{\Aut_{\mathcal{T}_1}(T)}(Z(T))$ so that $C_T(K^*)=Z$ and $T=[T, K^*]Z(T)$.

It remains to analyze the case where $\mathcal{T}$ is isomorphic to the $2$-fusion category of $\Alt(6)$. Assume that $\mathcal{C}_1$ is isomorphic to the $2$-fusion category of $\Alt(6)$. Since $Q_y\cong R$, we have that $|S|\leq 2^7$ and an appeal to \cite{bobsmall} implies that $\fs$ is not a minimal counterexample. By the above arguments, we see that $Q_y=R$. Moreover, $S=Q_x\times R=Q_x\times Q_y=Q_y\times T$.

Set $J^*=\Aut_{\mathcal{C}_1}(R)$ so that $[J^*, R]=R$. Let $J$ be the lift of $J^*$ to $\Aut_{\fs}(S)$ so that $[J, S]=R=Q_y$. In particular, $Q_x=C_S(J)$. Since $J$ normalizes $Q_y$, $J\subseteq N_{\fs}(Q_y)$. Since $E(N_{\fs}(Q_y))=\mathcal{T}_1$ is supported on $T$, we deduce that $J$ normalizes $T$. Then $[J, T]\le [J, S]\cap T=Q_y\cap T=\{1\}$ from which we conclude that $T\le C_S(J)=Q_x$. Since $T\cong Q_x$, we conclude that $T=Q_x$, as desired.
\end{proof}

\begin{lemma}
$Q_x$ and $Q_y$ are strongly closed in $\fs$.
\end{lemma}
\begin{proof}
We prove this for $Q_x$ but the proof for $Q_y$ is similar. Aiming for a contradiction, let $A\le Q_x$ and $\alpha\in\Hom_{\fs}(A, S)$ with $A\alpha\not\le Q_x$. By \cref{CentGeneration}, we may write $\alpha=\alpha_1\circ\dots \circ \alpha_n$ where each $\alpha_i$ is the restriction of some morphism in $O^{p'}(C_{\fs}(x))$ or $O^{p'}(C_{\fs}(y))$. Since $Q_x$ is strongly closed in $O^{p'}(C_{\fs}(x))$ we deduce that there is $B\le Q_x$ and $\phi\in\Hom_{O^{p'}(C_{\fs}(y)}(B, S)$ with $B\phi\not\le Q_x$. But $E(C_{\fs}(y))\normaleq C_{\fs}(y)$ is supported on $Q_x$ and so $Q_x$ is also strongly closed in $C_{\fs}(y)$ and we must have that $B\phi\le Q_x$, a contradiction. Hence, $Q_x$ is a strongly closed in $\fs$.
\end{proof}

\begin{theorem}\label{ComponentTheorem}
If $\fs$ satisfies \cref{mainhyp} then $\fs$ is not of parabolic component type.
\end{theorem}
\begin{proof}
By \cite[Proposition 3.3]{AOV}, since $Q_x$ and $Q_y$ are strongly closed in $\fs$ we have that $\fs=\fs_1\times \fs_2$ for some saturated fusion systems $\fs_1, \fs_2$ supported on $Q_x$ and $Q_y$ respectively. But $\fs$ is simple, a contradiction.
\end{proof}

\section{Fusion Systems of Parabolic Characteristic $p$}

In truth, the methods in this section are almost exactly the same used as in the treatment of simple groups with class two Sylow $p$-subgroups, and so there is little novelty for those familiar with this kind of analysis. On the other hand, the exotic systems supported on $7^{1+2}_+$ found by Ruiz and Viruel are of parabolic characteristic $7$ and so appear in this analysis. However, with the goal of only identifying the isomorphism type of $S$, there appears to be little utility in using fusion systems in place of groups when compared to the drastic simplifications obtained in \cref{Component Section}.

For this section, we operate under the following hypothesis:

\begin{hypothesis}\label{CharpHyp}
\cref{mainhyp} holds and $\fs$ is of parabolic characteristic $p$.
\end{hypothesis}

Because of \cref{ComponentTheorem}, in order to prove the \hyperlink{MainThm}{Main Theorem} it suffices to show that \cref{CharpHyp} provides no examples.

\begin{proposition}\label{FF-Action}
There is $E_1, E_2\in\mathcal{E}(\fs)$ such that $Z(E_1)\not\le E_2$ and $Z(E_2)\not\le E_1$.
\end{proposition}
\begin{proof}
Suppose first that there is $E_1, E_2\in\mathcal{E}(\fs)$ with $Z(E_1)\not\le E_2$. If $Z(E_2)\le E_1$ then $[Z(E_1), Z(E_2)]=\{1\}$ and so $Z(E_1)$ centralizes the $\Aut_{\fs}(E_2)$-invariant chain $\{1\}\normaleq Z(E_2)\normaleq E_2$, a contradiction by \cref{Chain} as $E_2\in\mathcal{E}(\fs)$. Hence, to prove the result it suffices to demonstrate a contradiction in the case where $Z(E)\le F$ for all $E,F\in\mathcal{E}(\fs)$. 

To this end, set $X:=\prod\limits_{E\in\mathcal{E}(\fs)} Z(E)$ and, aiming for a contradiction, assume that $X\le\bigcap\limits_{E\in\mathcal{E}(\fs)} E$. Indeed, since $[O^{p'}(\Aut_{\fs}(E)), E]\le Z(E)$ by \cref{GeneralEssentialStructure}, we see that $O^{p'}(\Aut_{\fs}(E))$ normalizes $X$ for all $E\in\mathcal{E}(\fs)$. Furthermore, since $\mathcal{E}(\fs)$ is invariant under the action of $\Aut_{\fs}(S)$, we see that $X$ is also normalized by $\Aut_{\fs}(S)$ and, by the Frattini argument, that $X$ is normalized by $\Aut_{\fs}(E)$ for all $E\in\mathcal{E}(\fs)$. Hence, $X\normaleq \fs$ by \cref{normalinF}, a contradiction since $O_p(\fs)=\{1\}$.
\end{proof}

We choose $E_1, E_2\in\mathcal{E}(\fs)$ as described in the above proposition.

\begin{lemma}\label{SpecificEssentialStructure}
For $i\in\{1,2\}$, we have that $E_i$ is elementary abelian, $S=E_iE_{3-i}$, $C_{E_1}(O^{p'}(\Aut_{\fs}(E_1)))\cap C_{E_2}(O^{p'}(\Aut_{\fs}(E_2)))=\{1\}$ and $E_i/C_{E_i}(O^{p'}(\Aut_{\fs}(E_i)))$ is a natural module for $O^{p'}(\Aut_{\fs}(E_i))\cong \SL_2(p^{n})$ for some $n\in\N$.
\end{lemma}
\begin{proof}
By \cref{GeneralEssentialStructure}, we have that $E_i=C_S(Z(E_i))$ for $i\in\{1,2\}$, and $Z(E_i)$ admits $O^{p'}(\Out_{\fs}(E_i))$ faithfully. We have that 
\[a_i:=|Z(E_i)/C_{Z(E_i)}(Z(E_{3-i}))|=|Z(E_i)/Z(E_i)\cap E_{3-i}|=|Z(E_i)Z(E_{3-i})/Z(E_{3-i})|.\] 
Without loss of generality, assume that $a_1\geq a_2$. Then, applying \cref{SEFF}, we deduce that \linebreak $Z(E_2)/C_{Z(E_2)}(O^p(\Aut_{\fs}(E_2)))$ is a natural module for $O^{p'}(\Out_{\fs}(E_2))\cong \SL_2(p^n)$. Therefore, we have that $a_2=p^n$ and $a_1\leq |S/E_2|=p^n$ from which we deduce that $a_1=a_2$. By symmetry, we get that $Z(E_i)/C_{Z(E_i)}(O^p(\Aut_{\fs}(E_i)))$ is a natural module for $O^{p'}(\Out_{\fs}(E_i))\cong \SL_2(p^n)$ and $S=Z(E_i)E_{3-i}$ for $i\in\{1,2\}$. 

Hence, we have that $E_i=Z(E_i)(E_i\cap E_{3-i})$ for $i\in\{1,2\}$. Therefore, we have that $E_1'=(E_1\cap E_2)'=E_2'$ so that \[E_i'\le C_{E_1}(O^{p'}(\Aut_{\fs}(E_1)))\cap C_{E_2}(O^{p'}(\Aut_{\fs}(E_2)))\cap Z(S)=:Z.\] Now, $O^{p'}(\Aut_{\fs}(E_1))$ and $O^{p'}(\Aut_{\fs}(E_2))$ act trivially on $Z$. Indeed, $O^{p'}(\Aut_{\fs}(E_i))=O^{p'}(\Aut_{N_{\fs}(Z)}(E_i))$ and $E_i\in\mathcal{E}(N_{\fs}(Z))$. By \cref{normalinF} we have that $Z\le O_p(N_{\fs}(Z))\le E_1\cap E_2$ and as $Z(E_1)$ centralizes $E_1\cap E_2$, $O_p(N_{\fs}(Z))$ is not $\fs$-centric. As $\fs$ is of parabolic characteristic $p$, we see that $O_p(N_{\fs}(Z))=\{1\}$ from which it follows that \[E_1'=E_2'=C_{E_1}(O^{p'}(\Aut_{\fs}(E_1)))\cap C_{E_2}(O^{p'}(\Aut_{\fs}(E_2)))=Z=\{1\}.\] Hence, both $E_1$ and $E_2$ are abelian and $Z(S)=E_1\cap E_2$.

Now, $E_i/C_{E_i}(O^{p'}(\Aut_{\fs}(E_i)))$ is elementary abelian of order $p^{2n}$ and to show that $E_i$ is elementary abelian it suffices to show that $Z(S)$ is. But $\mho^1(Z(S))\le C_{E_1}(O^{p'}(\Aut_{\fs}(E_1)))\cap C_{E_2}(O^{p'}(\Aut_{\fs}(E_2)))=\{1\}$ and so $Z(S)$ is elementary. Hence, both $E_1$ and $E_2$ are elementary abelian, as desired.
\end{proof}

\begin{lemma}
$S=E_1E_2$, $\mathcal{E}(\fs)\subseteq \mathcal{A}(S)$, $Z(F)\not\le E$ for any $E,F\in\mathcal{E}(\fs)$ with $E\ne F$, and if $p=2$ then $\mathcal{A}(S)=\{E_1, E_2\}=\mathcal{E}(\fs)$.
\end{lemma}
\begin{proof}
Since $E_1$ and $E_2$ are elementary abelian, we see that $Z(S)=E_1\cap E_2$. In particular, $C_{E_i}(x)=Z(S)$ for any $x\in S\setminus E_i$ for $i\in\{1,2\}$. Then for $A\in\mathcal{A}(S)$ we have that $|A|\leq |S/E_i||A\cap E_i|\leq q |Z(S)|=|E_i|$ and we infer that $E_1, E_2\in\mathcal{A}(S)$. Hence, $S=E_1E_2$. Since $E_1$ is elementary abelian and $C_S(Z(F))=F$ for all $F\in\mathcal{E}(\fs)$, we have that $Z(F)\not\le E_1$ for any $F\in\mathcal{E}(\fs)\setminus \{E_1\}$. Hence, applying \cref{FF-Action}, \cref{SpecificEssentialStructure} and the earlier arguments in this proof, we have that $\mathcal{E}(\fs)\subseteq \mathcal{A}(S)$. A similar argument also yields that $Z(F)\not\le E$ for any $E,F\in\mathcal{E}(\fs)$ with $E\ne F$.

If $p=2$ then we apply \cref{elementary2} to see that $|\mathcal{A}(S)|=2$. Since $\mathcal{E}(\fs)\subseteq \mathcal{A}(S)$ and $|\mathcal{E}(\fs)|\geq 2$, we conclude that $\mathcal{A}(S)=\{E_1, E_2\}=\mathcal{E}(\fs)$.
\end{proof}

\begin{proposition}\label{TwoEssentials}
If $|E(\fs)|=2$ then $\fs$ is not a minimal counterexample to the \hyperlink{MainThm}{Main Theorem}. In particular, if $p=2$ then $\fs$ is not a minimal counterexample.
\end{proposition}
\begin{proof}
We appeal to \cite[Corollary A]{vbbook} (although this only really uses the results from \cite{Greenbook}), to see that $F^*(\fs)$ is isomorphic to the $p$-fusion system of $\PSL_3(p^n)$ or $\PSp_4(2^n)$. Since $\fs$ is simple, $\fs$ is not a minimal counterexample.
\end{proof}

\begin{theorem}
$\fs$ is not a minimal counterexample to the \hyperlink{MainThm}{Main Theorem}.
\end{theorem}
\begin{proof}
By \cref{TwoEssentials}, we may assume throughout that $p$ is odd and $|E(\fs)|>2$. We observe that for $E,F\in\mathcal{E}(\fs)$ we have that $E\cap F=Z(S)$ and $|S/Z(S)|=p^{2n}$ and $|E/Z(S)|=p^n$. By an elementary counting argument, $|E(\fs)|\leq p^n+1$. Let $\hat{K}$ be a complement to $\Aut_S(E_1)$ in $N_{O^{p'}(\Aut_{\fs}(E_1))}(\Aut_S(E_1))$ so that $K$ is cyclic of order $p^n-1$. Let $K$ be the lift of $\hat{K}$ to $\Aut_{\fs}(S)$ and let $t$ be the unique involution in $K$. Then $t|_{E_1}\in Z(O^{p'}(\Aut_{\fs}(E_1)))$. 

Assume that $1\ne k\in K$ normalizes some $F\in E(\fs)\setminus E_1$. We may as well assume that $k$ normalizes $E_2$. Set $V_i:=[E_i, O^{p'}(\Aut_{\fs}(E_i))]$ and $Z_i:=C_{E_i}(O^{p'}(\Aut_{\fs}(E_i)))$ so that $E_i=V_i\times Z_i$, $V_i$ is a natural $\SL_2(p^n)$-module for $O^{p'}(\Aut_{\fs}(E_i))$ and $[E_1, k]=V_1$. Then $V_i\not\le E_{3-i}$ and $C_{V_1}(S)=[V_1, V_2]=C_{V_2}(S)$. Since $k$ normalizes $E_2$, we see that $k$ normalizes $Z_2$. Since $Z_1=C_{E_1}(k)$ and $Z_1\cap Z_2=\{1\}$ by \cref{SpecificEssentialStructure}, we have that $Z_2=[Z_2, k]\le [E_1, k]=V_1$. But then, $Z_2\le C_{V_1}(S)=C_{V_2}(S)$ and as $V_2\cap Z_2=\{1\}$, we see that $Z_2$ is trivial. Hence, $E_2$ is a natural $\SL_2(p^n)$-module and $S$ is isomorphic to a Sylow $p$-subgroup of $\SL_3(p^n)$. Then $S\in\mathcal{S}$, a contradiction since $\fs$ is a minimal counterexample.

Thus, no other essential subgroup of $\fs$ is normalized by by $k$. In particular, the preimage of $C_{S/Z(S)}(t)\not\subseteq\mathcal{E}(\fs)$ and our counting argument now gives $|\mathcal{E}(\fs)\setminus \{E_1\}|\leq p^n-1$. Indeed, $K$ acts regularly on this set. Hence, we have that either $E_1$ is $\Aut_{\fs}(S)$-invariant, or $|\mathcal{E}(\fs)|=p^n$. In the latter case, since $S=N_S(E_1)$, we have a contradiction. In the former case, for $t_2$ the lift of $1\ne \hat{t_2}\in Z(O^{p'}(\Aut_{\fs}(E_2)))$ to $\Aut_{\fs}(S)$, we have that $t_2$ normalizes $E_1$ so normalizes $Z_1$. Arguing as above, this implies that $Z_1\le C_{V_2}(S)=C_{V_1}(S)$ and we see that $Z_1$ is trivial and $S$ is isomorphic to a Sylow $p$-subgroup of $\SL_3(p^n)$, a contradiction.
\end{proof}

\section{Finite Simple Groups With Sylow $2$-subgroups of Nilpotency Class Two}

In this final section, we use our fusion system result to reprove Gilman and Gorenstein's classification of finite simple groups with Sylow $2$-subgroups of nilpotency class at most two. We recall this theorem as stated in the introduction.

\begin{thmgg}\hypertarget{thmgg2}{}
Suppose that $G$ is a finite simple group such that a Sylow $2$-subgroup $S$ of $G$ has nilpotency class at most two. Then one of the following holds:
\begin{enumerate}
    \item $S$ is abelian and $G$ is isomorphic to $\mathrm{J}_1$, $\PSL_2(q)$ where $q\equiv 3,5 \mod 8$, $\Ree(3^{2n+1})$, or $\PSL_2(2^n)$ for $n>1$;
    \item $S\cong \Dih(8)$ and $G$ is isomorphic to $\PSL_3(2)$, $\Alt(7)$ or $\PSL_2(q)$ where $q\equiv 7,9 \mod 16$; or
    \item $G\cong \Sz(2^n)$, $\PSU_3(2^n)$, $\PSL_3(2^n)$ or $\PSp_4(2^n)$ for $n>1$.
\end{enumerate}
\end{thmgg}

For use later, we recall the following facts about the groups mentioned in \hyperlink{thmgg2}{Theorem GG}.

\begin{lemma}\label{autsSchur}
Suppose that $G$ is a finite simple group such that a Sylow $2$-subgroup $S$ of $G$ has nilpotency class at most two.
\begin{enumerate}
    \item If $S$ is abelian then a Sylow $2$-subgroup of $\Out(G)$ is cyclic; and a Sylow $2$-subgroup of the Schur multiplier of $G$ is cyclic of order at most $2$.
    \item If $S\cong \Dih(8)$ then a Sylow $2$-subgroup of $\Out(G)$ is either cyclic, or a fours group when $G\cong \PSL_2(q)$ and $q \equiv 9\mod 16$; and a Sylow $2$-subgroup of the Schur multiplier of $G$ is cyclic of order at most $2$.
    \item If $G\cong \Sz(2^n)$ for $n>1$ then a Sylow $2$-subgroup of $\Out(G)$ is cyclic; and a Sylow $2$-subgroup of the Schur multiplier of $G$ is non-trivial only if $2^n=8$. Moreover, a $2$-central extension of $\Sz(8)$ has a class Sylow $2$-subgroup if and only if it is a trivial extension.
    \item If $G\cong \PSU_3(2^n)$ for $n>1$ then a Sylow $2$-subgroup of $\Out(G)$ is cyclic; and a Sylow $2$-subgroup of the Schur multiplier of $G$ is trivial.
    \item If $G\cong \PSp_4(2^n)$ for $n>1$ then a Sylow $2$-subgroup of $X$ with $\Inn(G)\le X\le \Aut(G)$ has nilpotency class $2$ if and only if a Sylow $2$-subgroup of $X/\Inn(G)$ is cyclic; and a Sylow $2$-subgroup of the Schur multiplier of $G$ is trivial.
    \item If $G\cong \PSL_3(2^n)$ for $n>1$ then a Sylow $2$-subgroup of $X\le \Aut(G)$ has nilpotency class $2$ if and only if a Sylow $2$-subgroup of $X/\Inn(G)$ is cyclic; and a Sylow $2$-subgroup of the Schur multiplier of $G$ is is non-trivial only if $2^n=4$. Moreover, a $2$-central extension of $\PSL_3(4)$ has a class Sylow $2$-subgroup if and only if it is a quotient of $2^2.\PSL_3(4)$. 
\end{enumerate}
\end{lemma}
\begin{proof}
Most of the statements concerning Schur multipliers can be gleaned from \cite[Section 6.1]{GLS3}. For $G\cong \PSL_3(4)$ or $\Sz(8)$, one can employ MAGMA \cite{MAGMA} to verify the remaining claims.

If $G\cong \Alt(n)$ for $5\leq n\leq 6$ then \cite[Theorem 5.2.1]{GLS3} applies to determine information about $\Out(G)$. If $G\cong \mathrm{J}_1$ then the triviality of $\Out(G)$ can be found in \cite[Table 5.3f]{GLS3}. If $G\cong \Ree(3^{2n+1})$ then we appeal to \cite[Theorem 2.5.12]{GLS3} to determine that $\Out(G)$ is comprised only of field automorphisms and so is cyclic. Suppose that $G\cong \PSL_2(p^a)$ for $p$ an odd prime and $p^a\equiv 3,5 \mod 8$ or $p^a\equiv 7,9 \mod 16$. Note that if $2\divides a$, then $p^a\equiv 9 \mod 16$ and $4\not| a$. We apply \cite[Theorem 2.5.12]{GLS3}, observing that $\Out(G)$ is a split extension of the diagonal outerautomorphisms (which generate a subgroup of order $2$) by the field automorphisms. This delivers the promised result when $G\cong \PSL_2(p^a)$ for $p$ an odd prime. We also observe that $\PSL_3(2)\cong \PSL_2(7)$.

Hence, we may assume that $G$ is isomorphic to a group of Lie type over a a field of characteristic $2$ and order at least $4$. We apply \cite[Theorem 2.5.12]{GLS3} and observe that a Sylow $2$-subgroup of the diagonal outerautomorphisms is trivial. Hence, a Sylow $2$-subgroup of $\Out(G)$ is generated by field automorphisms and graph automorphisms. Since the group generated by field automorphisms is cyclic, the result holds unless $G$ has a graph automorphism. Let $\sigma$ be a graph or graph-field automorphism of $G$ so that $G\cong \PSL_3(2^n)$ or $\PSp_4(2^n)$. As noted in \cref{CompAutos}, $\sigma$ can be chosen to normalize a Sylow $2$-subgroup $S$ of $G$ and swaps the two maximal elementary abelian subgroups $A$ and $B$ of $S$. Indeed, if a Sylow $2$-subgroup of $X$ has nilpotency class at most two and contains $\sigma$, then $[\sigma, S]\le Z(S)\le A\cap B$ and we deduce that $\sigma$ normalizes $A$ and $B$, a contradiction. Hence, a Sylow $2$-subgroup of $X/\Inn(G)$ is contained in the subgroup of $\Out(G)$ generated by field automorphisms, so is cyclic.
\end{proof}

Throughout, we let $G$ be a finite simple group whose Sylow $2$-subgroups have nilpotency class at most two, and is minimal subject to not appearing in the conclusion of \hyperlink{thmgg}{Theorem GG}. Let $S$ be a Sylow $2$-subgroup of $G$ and $\fs:=\fs_S(G)$.

\begin{proposition}
$O^{2'}(\fs)$ is simple.
\end{proposition}
\begin{proof}
If $G$ is a \emph{Goldschmidt group} then $G$ is determined in \cite{Goldschmidt2-Fus}. That is, either $G$ is isomorphic to simple group of Lie type of characteristic $2$ and rank $1$, or $S$ is abelian and $G$ is determined by Walter \cite{walter} (see also \cite{BenderAbelian}). In either case, $G$ is not a minimal counterexample to \hyperlink{thmgg}{Theorem GG}. Hence, $G$ is not a Goldschmidt group and applying \cite[Theorem1]{aschbachersimple}, since $G$ is a minimal counterexample, we see that $\fs$ is almost simple. That is, $F^*(\fs)$ is a simple fusion system.

Now, by \cref{p'genfit} we see that $F^*(\fs)\subseteq O^{p'}(\fs)$, $F^*(\fs)$ is simple and $O^{p'}(\fs)$ is determined by the \hyperlink{MainThm}{Main Theorem}. Thus, $O^{2'}(\fs)$ is isomorphic to the $2$-fusion category of $\PSL_3(2^n)$ or $\PSp_4(2^n)$ for $n\geq 1$. Indeed, either $O^{2'}(\fs)$ is simple and the proposition holds, or $O^{2'}(\fs)$ is isomorphic to the $2$-fusion system of $\PSp_4(2)\cong \Sym(6)$. By \cref{Sym6}, we have that $\fs\cong O^{2'}(\fs)$ is isomorphic to the $2$-fusion system of $\Sym(6)$ and so $O^2(\fs)<\fs$. But $\hyp(\fs)=S\cap O^2(G)<S$ by \cite[pg. 33]{ako}, a contradiction since $G$ is simple.
\end{proof}

We now have that $O^{2'}(\fs)$ is isomorphic to the $2$-fusion category of $\PSL_3(2^n)$ or $\PSp_4(2^{n+1})$ where $n\in \N$.

\begin{proposition}
$S$ is not dihedral of order $8$.
\end{proposition}
\begin{proof}
If $S\cong \Dih(8)$ then $G$ is classified by Gorenstein and Walter \cite{GorenWaltDih}. We instead appeal to the much shorter proof of Bender \cite{BenderDih}.
\end{proof}

From this point on, we set $q=2^n$ such that $O^{2'}(\fs)$ is either isomorphic to the $2$-fusion category of $\PSL_3(q)$ or $\PSp_4(q)$ where $n\in \N$. Since $S$ is not dihedral of order $8$, we may assume that $q\geq 4$. We have that $S$ contains exactly two maximal elementary abelian subgroups which we denote by $A$ and $B$.

We wish to apply the \emph{Bender method} to demonstrate that for $X\in\{A, B\}$, we have that $N_G(X)$ is a maximal subgroup of $G$ with $F^*(N_G(X))=X$. At this point, the model theorem applied to $\fs_S(G)$ reveals that $N_G(A)$ and $N_G(B)$ are isomorphic (in a strong sense) to parabolic subgroups in $\PSL_3(q)$ or $\PSp_4(q)$. Then we may recognize a BN-pair of rank $2$ in $G$ and the characterization of such groups by Fong and Seitz \cite{FongSeitz}, or the more modern interpretation using Moufang polygons as classified by Tits and Weiss \cite{titsweiss}, implies that $G$ is not a counterexample to \hyperlink{thmgg}{Theorem GG}.

We begin with some general results.

\begin{lemma}
We have that $N_G(S)=N_G(A)\cap N_G(B)<N_G(A)$ and $N_G(S)<N_G(B)$.
\end{lemma}
\begin{proof}
Note that $N_G(S)$ acts on the set $\{A, B\}$ and as $S$ is Sylow in $N_G(S)$, we have that $N_G(S)\le N_G(A)\cap N_G(B)$. Since $S=AB$, we see that $N_G(A)\cap N_G(B)\le N_G(S)$. From the actions present in $\fs$, we deduce that $N_G(S)<N_G(A)$ and $N_G(S)<N_G(B)$, which completes the proof.
\end{proof}

\begin{lemma}\label{RadicalNoComps}
For $X\in \{A, B\}$ and $\bar{N_G(X)}=N_G(X)/O_{2'}(N_G(X))$, we have that $F^*(\bar{N_G(X)})=\bar{X}$. Consequently, $\bar{N_G(X)}/\bar{X}\cong \Aut_{\fs}(X)$. 
\end{lemma}
\begin{proof}
As in the setup of the lemma, we let $\bar{N_G(X)}=N_G(X)/O_{2'}(N_G(X))$ for $X\in\{A, B\}$. Since $\fs_S(N_G(X))\cong \fs_{\bar{S}}(\bar{N_G(X)})$, we see that $O_2(\bar{N_G(X)})=\bar{X}$. To prove that $F^*(\bar{N_G(X)})=\bar{X}$, it remains to show that $E(\bar{N_G(X)})=\{1\}$.

Aiming for a contradiction, suppose that $E(\bar{N_G(X)})\ne \{1\}$ and write $T:=\bar{S}\cap E(\bar{N_G(X)})$. Then $[T, \bar{X}]\le [E(\bar{N_G(X)}), \bar{X}]=\{1\}$ from which we deduce that $T\le \bar{X}$. But then $T\le Z(E(\bar{N_G(X)}))$ and $E(\bar{N_G(X)})/Z(E(\bar{N_G(X)}))$ has odd order, a contradiction. Hence, $E(\bar{N_G(X)})=\{1\}$ and $\bar{X}=F^*(\bar{N_G(X)})$. Since \[\Aut_{\fs}(X)=\Aut_{N_G(X)}(X)\cong \Aut_{\bar{N_G(X)}}(\bar{X})\cong \bar{N_G(X)}/C_{\bar{N_G(X)}}(\bar{X}),\] the lemma holds.
\end{proof}

\begin{proposition}\label{Centcent}
Suppose that $H$ is a finite group with $S\in\syl_2(H)$. Assume that $S$ is isomorphic to a Sylow $2$-subgroup of $\PSL_3(2^a)$ or $\PSp_4(2^a)$ for $a>1$. Then $[O(H), Z(S)]=\{1\}$.
\end{proposition}
\begin{proof}
Assume that $H$ is a minimal counterexample with regards to the proposition. By minimality, $H=O(H)S$. Moreover, for $\{1\}\ne W\normaleq H$ with $W\le O(H)$, we have that $[Z(S)W/W, O(H)/W]=\{1\}$ so that $[Z(S), O(H)]\le W$. Applying the minimal hypothesis to $WS$, we see that $W=O(H)$ for otherwise by $[Z(S), W]=\{1\}$ and by coprime action, $[Z(S), O(H)]=\{1\}$. Hence, $O(H)$ is a minimal normal subgroup of $H$ from which we deduce that $O(H)$ is an elementary abelian $r$-group for some odd prime $r$. Moreover, since $S$ centralizes $Z(S)$, we have that $O(H)$ is irreducible under the action of $Z(S)$.

Now, by \cite[Proposition 11.23]{GLS2}, we have that $O(H)=\langle C_{O(H)}(A_0) \mid A_0\le A, [A: A_0]=2\rangle$. For each such $A_0$, $Z(S)$ centralizes $A_0$ and so normalizes $C_{O(H)}(A_0)$ from which we deduce that $C_{O(H)}(A_0)\in\{\{1\}, O(H)\}$. Choose $A^0\le A$, with $[A: A^0]=2$ and $[A^0, O(H)]=\{1\}$. Then $A=Z(S)A^0$. But now, $C_S(O(H))\normaleq H$ and so $Z(S)\le [S, A_0]\le [S, C_S(O(H))]\le C_S(O(H))$, and $H$ is not a minimal counterexample. This completes the proof.
\end{proof}

Recall that a finite $\mathcal{K}$-group is a finite group all of whose composition factors are known finite simple groups.

\begin{proposition}\label{BigO2New}
Suppose that $H$ is a finite $\mathcal{K}$-group with $S\in\syl_2(H)$. Assume that $S$ is isomorphic to a Sylow $2$-subgroup of $\PSL_3(q)$ or $\PSp_4(q)$ for $q=2^a$ and $a>1$; and write $\wt{H}:=H/O(H)$. Then either $\wt S\le F^*(\wt{H})\cong \PSL_3(q)$, $2^2.\PSL_3(4)$ or $\PSp_4(q)$; or $F^*(\wt{H})=O_2(\wt{H})$.
\end{proposition}
\begin{proof}
Let $H$ be a minimal counterexample to the proposition. Throughout set $\bar{H}:=H/O_2(H)$. 

By minimality we have that $O(H)=\{1\}$. Moreover, again by minimality, it follows that $H=F^*(H)S=E(H)S$. We assume throughout that $F^*(H)\ne O_2(H)$ so that $E(H)\ne \{1\}$. Let $T:=S\cap E(H)$ and write $E(H)=K_1\dots K_l$. Considering $\bar{E(H)}$ and applying \cref{NoWreath}, it follows that $S$ normalizes $K_iO_2(H)$ for all $i\in \{1,\dots, l\}$. In particular, we have that $X=S(K_2\dots K_l)$ is a proper subgroup of $H$ containing $S$. Then $E(X)=K_2\dots K_l$ and $O(X)\le E(X)\normaleq SE(H)=H$ so that $O(X)=\{1\}$. Hence, $X$ is a proper subgroup of $H$ containing $S$ with $O(X)=\{1\}$. By minimality, we conclude that either $X=E(X)\cong \PSL_3(q)$ or $\PSp_4(q)$ or $X=S$. In the former case, we see that $K_1$ centralizes $X$ so that $K_1$ centralizes $S$ from which we deduce that $\bar{K_1}$ is a non-abelian finite simple group of odd order, a contradiction. Hence, we have that $X=S$ and so $E(H)=K_1$ is quasisimple and 

\textbf{(1)} $\bar{E(H)}$ is a non-abelian simple group.\hfill

Since $H$ is a finite $\mathcal{K}$-group, and $T$ has nilpotency class at most two, we have that $\bar{E(H)}$ is determined by \hyperlink{thmgg}{Theorem GG}. By \cref{autsSchur}, the only non-abelian finite simple groups with Sylow $2$-subgroups of nilpotency class at most two and Schur multiplier with $2$-part of order at least $4$ are $\PSL_3(4)$ and $\Sz(8)$. Indeed, if $K$ is a quasisimple group with $Z(K)$ a non-trivial elementary abelian $2$-subgroup, $K/Z(K)\cong \PSL_3(4)$ or $\Sz(8)$ and a Sylow $2$-subgroup of $K$ of nilpotency class at most two then $K/Z(K)\cong \PSL_3(4)$ and $|Z(K)|\leq 4$. Hence, we have that $|O_2(H)\cap E(H)|\leq 4$ and $|O_2(H)\cap E(H)|\leq 2$ unless $E(H)\cong \PSL_3(4)$. Additionally, we note that $C_S(E(H))$ is normalized by $H=SE(H)$ and so $C_S(E(H))\le O_2(H)$. Note that if $[S, E(H)]\le Z(E(H))$ then $[S, E(H), E(H)]=\{1\}$ and the three subgroups lemma yields that $[S, E(H)]=\{1\}$. Therefore, 

\textbf{(2)} $\bar{S}$ acts faithfully on $\bar{E(H)}$ and $\bar{S}/\bar{T}$ is isomorphic to a subgroup of $\Out(E(\bar{H}))$. \hfill

Suppose that $T\le Z(S)$ so that $T$ is abelian. Applying \cref{autsSchur}, we deduce that a Sylow $2$-subgroup of $\Out(\bar{E(H)})$ has order at most $2$. It follows that $\bar{T}$ has index at most $2$ in $\bar{S}$. Hence, $\bar{Z(S)}$ has index at most $2$ in $\bar{S}$ so that $Z(S)O_2(H)$ has index at most $2$ in $S$. Then $T\le Z(S)\le [S, O_2(H)]\le O_2(H)$, a contradiction. Hence, we must have that 

\textbf{(3)} $T\not\le Z(S)$. \hfill

Then $\{1\}\ne [T, S]\le Z(S)$ and so either $Z(S)< T$; or $S$ is isomorphic to a Sylow $2$-subgroup of $\PSp_4(q)$, $|Z(S)|=q^2$, $|TA|\leq 2$, $|(T\cap A)Z(S)/Z(S)|\leq 2$ and $[S, T]$ is an order $q$ subgroup of $Z(S)$ where $A\in\mathcal{A}(S)$.

Assume that the latter holds. Then $Z(S)\cap T$ has index at most $4$ in $T$ and has order at least $4$. Since $Z(S)\cap T\le Z(T)$, comparing with the groups listed in \hyperlink{thmgg}{Theorem GG}, we conclude that $T$ is abelian. Suppose that $|(T\cap A)Z(S)/Z(S)|=2$ so that $T\cap A\not\le Z(S)$. Then $TO_2(H)\le C_S(T\cap A)\le A$ and so $TO_2(H)$ has index at least $q$ in $S$. Since $T$ is abelian, once again applying \cref{autsSchur} we deduce that a Sylow $2$-subgroup of $\Out(\bar{E(H)})$ has order at most $2$, and since $q>2$, we have a contradiction. Hence, $T\cap A\le Z(S)$ so that $|TZ(S)/Z(S)|=2$ and $T$ is abelian. Moreover, $C_A(T)=Z(S)$ so that $C_S(T)$ has index at least $q$ in $S$ and so $TO_2(H)$ has index at least $q$ in $S$. Then the same argument as above yields another contradiction. Thus, 

\textbf{(4)} $Z(S)<T$. \hfill

Then we have that $|O_2(H)\cap Z(S)|\leq 4$. Indeed, if $O_2(H)\not\le Z(S)$ then $[S, O_2(H)]\le O_2(H)\cap Z(S)$ and $[S, O_2(H)]$ has order at least $q$. It follows that $q=4$ and $E(H)/Z(E(H))\cong \PSL_3(4)$. Then, we have that $|S|\leq 2^8\leq |T/Z(E(H))||Z(E(H))|=|T|$ and $S=T$. But then $O_2(H)\le Z(S)$. Hence, we have shown that $O_2(H)\le Z(S)$ and so $|O_2(H)|\leq 4$. If $\bar{T}$ is abelian, then $|O_2(H)|\leq 2$. Indeed, if $O_2(H)=\{1\}$ then $T$ is abelian while if $|O_2(H)|=2$ then $|\bar{T}|=4$ and $Z(S)$ has index at most $2$ in $T$ so that $T$ is abelian. Either way, we conclude that $T$ is also abelian. Applying \cref{autsSchur} we deduce that a Sylow $2$-subgroup of $\Out(\bar{E(H)})$ has order at most $2$ so that $T$ has index at most $2$ in $S$, and since $q\geq 4$, we have a contradiction. Therefore, 

\textbf{(5)} $\bar{T}$ is non-abelian. \hfill

Moreover, $\bar{Z(S)}\le Z(\bar{T})=\Phi(\bar{T})\le \bar{Z(S)}$. If $|O_2(H)|=4$ then $\bar{E(H)}\cong \PSL_3(4)$ and $|Z(S)|=2^4$. Hence, $S$ is either isomorphic to a Sylow $2$-subgroup of $\PSL_3(16)$; or a Sylow $2$-subgroup of $\PSp_4(4)$. In the former case, we have that $|\bar{S}/\bar{T}|=2^6$ and since $|\Out(\PSL_3(4))|_2=4$, we obtain a contradiction. In the latter case, we have that $S=T$, $S\le E(H)=F^*(H)\cong \PSL_3(4)$, a contradiction since $H$ is a counterexample. Thus, $|O_2(H)|\leq 2$.

If $|O_2(H)|=2$ then applying \cref{autsSchur}, we deduce that $|\bar{Z(T)}|^3=|\bar{T}|$. If, in addition, $S$ is isomorphic to a Sylow $2$-subgroup of $\PSp_4(q)$, then $|\bar{T}|\leq 2|Z(\bar{T})|^2$ and so $|\bar{Z(T)}|^3=|\bar{T}|\leq 2|Z(\bar{T})|^2$ and $|\bar{Z(T)}|=2$, a contradiction since $q>2$ and $|\bar{Z(T)}|=\frac{q^2}{2}$. Thus, $S$ is isomorphic to a Sylow $2$-subgroup of $\PSL_3(q)$ and $\bar{T}$ is isomorphic to a Sylow $2$-subgroup of $\PSL_3(q/2)$. Note that the $2$-part of the Schur multiplier of $\PSL_3(q/2)$ is trivial unless $q/2\leq 4$ by \cref{autsSchur}. 

If $q=8$ then $\bar{E(H)}\cong \PSL_3(4)$ and $|\bar{S}/\bar{T}|=4$. Since $|\Out(\PSL_3(4))|=12$ we must have that $\bar{S}$ is isomorphic to a Sylow $2$-subgroup of $\Aut(\PSL_3(4))$, a contradiction since such a group has nilpotency class $3$. Thus, $q=4$ and $\bar{T}\cong \Dih(8)$. Again, $|\bar{S}/\bar{T}|=4$ and we deduce that $|\Out(\bar{E(H)})|_2\geq 4$. By \cref{autsSchur}, we see that $\bar{E(H)}\cong \PSL_2(r)$ where $r \equiv  9 \mod 16$, and by \cref{autsSchur}, we deduce that $T$ is isomorphic to a Sylow $2$-subgroup of $\SL_2(r)$. But then $T$ has nilpotency class $3$, and we again have a contradiction.

Finally, if $O_2(H)=\{1\}$ then comparing $|Z(S)|\leq |Z(T)|$ and $|T|$ we either have that $S=T$; or $S$ is isomorphic to a Sylow $2$-subgroup of $\PSL_3(q)$, $Z(S)=Z(T)$, $|S/T|=|T/Z(S)|=q$ and $T$ is isomorphic to a Sylow $2$-subgroup of $\PSp_4(q^{\frac{1}{2}})$. In the latter case, we observe that $|\Out(\PSp_4(q^{\frac{1}{2}}))|_2=4$ and so we must have that $q=4$. But then $E(H)\cong \PSp_4(2)\cong \Sym(6)$ is not quasisimple, a contradiction. Hence, $S=T$ and as $H$ is a $\mathcal{K}$-group, using the Sylow $2$-subgroup structure of $S$ and comparing with the groups listed in \hyperlink{thmgg}{Theorem GG} implies that $H$ is not a counterexample to the proposition.
\end{proof}

We now begin the determination of $G$. As a first case, we assume that $O^{2'}(\fs)$ is isomorphic to the $2$-fusion category of $\PSL_3(q)$.

\begin{lemma}\label{L3invNew}
Suppose that $z\in Z(S)^\#$. Then $C_G(z)\le N_G(Z(S))$. 
\end{lemma}
\begin{proof}
Set $C:=C_G(z)$, $Q:=O_2(C)$ and $\bar{C}:=C/O(C)$. Then $O_2(\bar{C})\ne \{1\}$ and applying \cref{BigO2New} we see that $F^*(\bar{C})=O_2(\bar{C})$ and $F^*(C)=F(C)$. Hence, $\bar{C}$ is a model for $C_{\fs}(z)$. Since $S\normaleq C_{\fs}(z)$, we have that $SO(C)\normaleq C$. Then by the Frattini argument, $C=O(C)N_C(S)$ and since $[Z(S), O(C)]=\{1\}$ by \cref{Centcent}, we conclude that $Z(S)\normaleq C$. 
\end{proof}

Let $H$ be a maximal subgroup of $G$ which contains $N_G(Z(S))$.

\begin{proposition}\label{L32HNew}
Either $H=N_G(Z(S))$, or $H=N_G(X)$ for some $X\in\{A, B\}$.
\end{proposition}
\begin{proof}
Consider $\bar{H}=H/O(H)$. Assume first that $O_2(\bar{H})\ne \{1\}$. Then by \cref{BigO2New} we have that $F^*(\bar{H})=O_2(\bar{H})$ so that $F^*(H)=F(H)$. By \cref{Centcent}, we have that $Z(S)\le C_H(O(H)S)\le C_H(F(H))$ and so $Z(S)\le O_2(H)$. The action of $\Aut_{\fs}(S)$ yields that $O_2(H)\in\{Z(S), A, B, S\}$; or $|O_2(H)|=q^2$, $\Omega_1(O_2(H))=Z(S)$ and $S=O_2(H)A=O_2(H)B$. Since $H$ is a maximal subgroup of $G$, and $O_2(G)=\{1\}$, the result follows in this case.

Thus, we continue under the assumption that $O_2(\bar{H})=\{1\}$ so that $F^*(\bar{H})=O^{2'}(\bar{H})\cong \PSL_3(q)$ by \cref{BigO2New}. Note that $O(N_G(A))\le N_G(Z(S))\le H$ so applying \cref{RadicalNoComps}, it follows that $N_G(A)\le H$. Similarly, $N_G(B)\le H$. Let $T\in \syl_2(N_G(A))\setminus \{S\}$. Then $A=Z(S)Z(T)$ and for $z\in Z(T)$, we have that $C_G(z)\le N_G(Z(T))$. There is $x\in N_G(A)$ with $S^x=T$ from which it follows that $N_G(Z(T))=N_G(Z(S^x))=N_G(Z(S)^x)=N_G(Z(S))^x\le H^x\le H$. Note that for each element $a$ of $A$, there is some $P\in\syl_2(N_G(A))$ with $a\in Z(P)$. Applying a similar methodology with $B$ yields that $C_G(s)\le H$ for all involutions $s\in S$. Finally, since $N_G(S)\le H$ we have by \cite[Proposition 17.11]{GLS2} that $H$ is strongly embedded in $G$ and then \cite{Bender} implies that no such configuration exists.
\end{proof}

\begin{proposition}
Both $N_G(A)$ and $N_G(B)$ are maximal in $G$, $F^*(N_G(A))=A$, $F^*(N_G(B))=B$ and $G=\langle N_G(A), N_G(B)\rangle$.
\end{proposition}
\begin{proof}
By \cref{L32HNew}, either $H=N_G(Z(S))$ or $H=N_G(X)$ for some $X\in\{A, B\}$. Suppose first that $N_G(B)\ne H=N_G(Z(S))\ne N_G(A)$. Let $T\in\syl_2(N_G(A))\setminus \{S\}$. Then $A=Z(S)Z(T)$ so that $Z(T)\le S$. Since $q>2$, we have that $O(H)=\langle C_{O(H)}(t) \mid t\in Z(T)^\#\rangle$. By \cref{L3invNew}, we have that $O(H)\le N_G(Z(T))$. But then $O(H)\le N_G(Z(S))\cap N_G(Z(T))\le N_G(A)$ from which we conclude that $[O(H), A]=\{1\}$. Hence, $A\le C_H(O(H))$. By a similar line of reasoning, we see that $B\le C_S(O(H))$ and so $[S, O(H)]=\{1\}$. By \cref{BigO2New}, we have that $F^*(H/O(H))=O_2(H/O(H))$ and so $H/O(H)$ is a model for $N_{\fs}(O_2(H))$. Since $O_2(N_{\fs}(Z(S)))=S$, we deduce that $S=O_2(H)$. But $N_G(S)\le N_G(A)\cap N_G(A)$, and since $H$ is maximal, we have a contradiction.

Hence, without loss of generality, we may assume that $H=N_G(A)$. In particular, $N_G(B)\not\le H$ and so there is $M$ a maximal subgroup of $G$ containing $N_G(B)$. Moreover, by \cref{RadicalNoComps} we see that $N_G(Z(S))=N_{N_G(A)}(Z(S))=O(N_G(Z(S)))N_G(S)$. Let $T\in\syl_2(N_G(B))\setminus S$ so that $B=Z(S)Z(T)$ and $Z(T)\le S$. Then $O(N_G(Z(S)))=\langle C_{O(N_G(Z(S)))}(t) \mid t \in Z(T)^\#\rangle$ and by \cref{L3invNew} we have that $O(N_G(Z(S)))$ normalizes $Z(T)$. Then $O(N_G(Z(S)))$ normalizes $B=Z(S)Z(T)$ and so $N_G(Z(S))< N_G(B)\le M$. Applying \cref{L32HNew} yields $M=N_G(B)$. Thus, we have shown that both $N_G(A)$ and $N_G(B)$ are maximal in $G$ and $G=\langle N_G(A), N_G(B)\rangle$.

Finally, for $T\in\syl_2(N_G(A))\setminus \{S\}$, we have that $A=Z(S)Z(T)$ so that $S=Z(T)B$. In particular, $O(M)=\langle C_{O(M)}(t) \mid t\in Z(T)^\#\rangle$ and by \cref{L3invNew} we have that $O(M)\le N_G(Z(T))\cap N_G(B)=N_G(S)\le N_G(A)$. Similarly, $O(H)\le N_G(S)\le N_G(B)$ and by \cref{RadicalNoComps} we have that $F^*(H)\le M$ and $F^*(M)\le H$. Then \cite[Theorem 19.1]{GLS2} yields that $F^*(N_G(A))=A$, $F^*(N_G(B))=B$, as desired. 
\end{proof}

We now assume that $O^{2'}(\fs)$ is isomorphic to the $2$-fusion category of $\PSp_4(q)$. We provide a brief description of the elements of $Z(S)$. 

Let $A_0\le A$ be such that $A_0\cap Z(S)=\{1\}$ and $A=A_0Z(S)$. Then $[a, S]=[a, B]$ has order $q$ for each $a\in A_0^\#$. Let $A^0\langle a_1, a_2\rangle$ be a fours subgroup of $A_0$. Then $[a_1, S][a_2, S]=[A^0, S]$ has order $q^2$ by \cref{PSp4lem} from which we deduce that $[a_1, S]\cap [a_2, S]=\{1\}$. We note that $C_A(O^{2'}(\Aut_{\fs}(A)))\cap [a, S]=C_B(O^{2'}(\Aut_{\fs}(B)))\cap [a, S]=\{1\}$ for $a\in A_0^\#$. Indeed, the set $C_A(O^{2'}(\Aut_{\fs}(A)))\cup C_B(O^{2'}(\Aut_{\fs}(B)))$ is determined entirely by $S$. Counting elements in $Z(S)$, we deduce that \[Z(S)=C_A(O^{2'}(\Aut_{\fs}(A)))\cup C_B(O^{2'}(\Aut_{\fs}(B))) \sqcup \bigsqcup_{a \in A_0^\#} [a, S]^\#.\]

\begin{lemma}\label{SP4inv1}
Assume that $z\in Z(S)$ but $z\not\in C_A(O^{2'}(\Aut_{\fs}(A)))\cup C_B(O^{2'}(\Aut_{\fs}(B)))$. Then $C_G(z)\le N_G(Z(S))$.
\end{lemma}
\begin{proof}
Set $C:=C_G(z)$, $Q:=O_2(C)$ and $\bar{C}:=C/O(C)$. Then $O_2(\bar{C})\ne \{1\}$ and applying \cref{BigO2New} we either have that $\bar{S}\le F^*(\bar{C})\cong 2^2.\PSL_3(4)$ and $S$ is isomorphic to a Sylow $2$-subgroup of $\PSp_4(4)$; or $F^*(\bar{C})=O_2(\bar{C})$ and $F^*(C)=F(C)$. By the description of the elements of $Z(S)$, there is $x\in A\setminus Z(S)$ such that $z\in [S, x]$ and $|[S, x]|=4$. In the former case, we would have that $\bar{z}\in [\bar{S}, \bar{x}]$ and $[\bar{S}, \bar{x}]\cap O_2(\bar{C})=\{1\}$, a contradiction.

Hence, $F^*(\bar{C})=O_2(\bar{C})$, $F^*(C)=F(C)$ and $\bar{C}$ is a model for $C_{\fs}(z)$. Since $S\normaleq C_{\fs}(z)$, we have that $SO(C)\normaleq C$. Then by the Frattini argument, $C=O(C)N_C(S)$ and since $[Z(S), O(C)]=\{1\}$ by \cref{Centcent}, we conclude that $Z(S)\normaleq C$.
\end{proof}

\begin{lemma}\label{SP4inv2}
Let $X\in\{A, B\}$. Assume that $z\in C_X(O^{2'}(\Aut_{\fs}(X)))$. Then $C_G(z)\le N_G(X)$.
\end{lemma}
\begin{proof}
Set $C:=C_G(z)$, $Q:=O_2(C)$ and $\bar{C}:=C/O(C)$. Then $O_2(\bar{C})\ne \{1\}$ and applying \cref{BigO2New} we either have that $\bar{S}\le F^*(\bar{C})\cong 2^2.\PSL_3(4)$ and $S$ is isomorphic to a Sylow $2$-subgroup of $\PSp_4(4)$; or $F^*(\bar{C})=O_2(\bar{C})$ and $F^*(C)=F(C)$. In the former case, we have that $\PSL_2(4)\cong O^{2'}(\Aut_{\fs_S(C)}(X))\le O^{2'}(\Aut_{\fs}(X))\cong \PSL_2(4)$ and we deduce that $O^{2'}(\Aut_{\fs_S(C)}(X))=O^{2'}(\Aut_{\fs}(X))$ for $X\in \{A, B\}$. But $C_{Z(S)}(O^{2'}(\Aut_{\fs_S(C)}(A)))=C_{Z(S)}(O^{2'}(\Aut_{\fs_S(C)}(B)))$, and this is a contradiction since $O_2(\fs)=\{1\}$.

Hence, $F^*(\bar{C})=O_2(\bar{C})$, $F^*(C)=F(C)$ and $\bar{C}$ is a model for $C_{\fs}(z)$. Since $X=O_2(C_{\fs}(z))$, we have that $XO(C)\normaleq C$. Then by the Frattini argument, $C=O(C)N_C(X)$. Since $[Z(S), O(C)]=\{1\}$ by \cref{Centcent}, we conclude that \[[X, O(C)]=[\langle Z(S)^{\Aut_{C_{\fs}(z)}(X)}\rangle, O(C)]= [\langle Z(S)^{N_C(X)}, O(C)]=[Z(S), O(C)]=\{1\}\] and $X\normaleq C$.
\end{proof}

For $X\in\{A, B\}$, we let $H_X$ be a maximal subgroup of $G$ which contains $N_G(X)$. 

\begin{proposition}\label{SP4H}
Both $N_G(A)$ and $N_G(B)$ are maximal in $G$, $F^*(N_G(A))=A$, $F^*(N_G(B))=B$ and $G=\langle N_G(A), N_G(B)\rangle$.
\end{proposition}
\begin{proof}
We will demonstrate that $H_X=N_G(X)$. Set $\bar{H_X}:=H_X/O(H_X)$. 

Note that $O(N_G(Z(S)))$ centralizes $z\in C_X(O^{2'}(\Aut_{\fs}(X)))$ and by \cref{SP4inv2}, we conclude that $O(N_G(Z(S)))\le N_G(X)$. Hence, $O(N_G(Z(S)))$ normalizes $S=AB$ and so $[S, O(N_G(Z(S)))]=\{1\}$. Applying \cref{BigO2New}, we must have that $F^*(N_G(Z(S))/O(N_G(Z(S))))$ is a $2$-group. In other words, $N_G(Z(S))/O(N_G(Z(S)))$ is a model for $N_{\fs}(Z(S))$. Since $S\normaleq N_{\fs}(Z(S))$, we deduce that $SO(N_G(Z(S)))\normaleq N_G(Z(S))$ and the Frattini argument implies that $S\normaleq N_G(Z(S))$. Hence, we have that $N_G(Z(S))=N_G(S)\le N_G(X)$. Furthermore, $O(N_G(X))$ centralizes $Z(S)$ and so normalizes $S$. Hence, $[S, O(N_G(X))]=\{1\}$ and $O(N_G(X))\le N_G(S)$.

We apply \cref{BigO2New} to $H_X$. Assume first that $\bar{S}\le F^*(\bar{H_X})\cong 2^2.\PSL_3(4)$ and $S$ is isomorphic to a Sylow $2$-subgroup of $\PSp_4(4)$. Since $N_G(X)\le H_X$, it follows that $O_2(\bar{H_X})=\bar{C_X(O^{2'}(\Aut_{\fs}(X)))}$. Since $Z(S)$ centralizes $O(H_X)$ by \cref{Centcent}, we conclude that $C_X(O^{2'}(\Aut_{\fs}(X)))=O_2(H_X)$. Then $O^{2'}(H_X)\le C_G(C_X(O^{2'}(\Aut_{\fs}(X))))\le N_G(X)$ by \cref{SP4inv2}. By the Frattini argument, and as $N_G(S)\le N_G(X)$, we have that $H_X=N_G(X)$, impossible since $O_2(H_X)<X$. 

Suppose now that $F^*(\bar{H_X})=O^{2'}(\bar{H})\cong \PSp_4(q)$. Let $\{A, B\}=\{X, Y\}$ and consider $N_G(Y)$. By \cref{BigO2New}, we have that $F^*(N_G(Y)/O(N_G(Y)))=YO(N_G(Y))/O(N_G(Y))$ and so $N_G(Y)/O(N_G(Y))$ is a model for $N_{\fs}(Y)$. Since $F^*(\bar{H_X})=O^{2'}(\bar{H})\cong \PSp_4(q)$, we have that $\PSL_2(q)\cong O^{2'}(\Aut_{\fs_S(\bar{H_X})}(Y))\le O^{2'}(\Aut_{\fs}(Y))\cong \PSL_2(q)$. Moreover, by the Frattini argument, we have that $N_G(Y)=O^{2'}(N_G(Y))N_G(S)$ and as $O(N_G(Y))\le N_G(S)\le N_G(X)\le H_X$, we deduce that $N_G(Y)\le H_X$. 

Let $T\in \syl_2(N_G(X))\setminus \{S\}$. Then $X=Z(S)Z(T)$ and for $z\in Z(T)\setminus (Z(S)\cap Z(T))$, we have that $C_G(z)\le N_G(Z(T))$. There is $g\in N_G(X)$ with $S^g=T$ from which it follows that $N_G(Z(T))=N_G(Z(S^g))=N_G(Z(S)^g)=N_G(Z(S))^g\le N_G(X)^g=N_G(X)$. Note that for each element $x$ of $X$, there is some $P\in\syl_2(N_G(X))$ with $x\in Z(P)$. Applying a similar methodology with $Y$ with $\{X, Y\}=\{A, B\}$ yields $C_G(s)\le H_X$ for all involutions $s\in S$. Finally, since $N_G(S)\le H$ we have by \cite[Proposition 17.11]{GLS2} that $H$ is strongly embedded in $G$ and then \cite{Bender} implies that no such configuration exists.

Thus, we have by \cref{BigO2New} that $\bar{S}\le F^*(\bar{H_X})=O_2(\bar{H_X})$. Note that $O(H_X)\le N_G(Z(S))\le N_G(X)$ and so $[X, O(H_X)]=\{1\}$. Hence, $O_2(X)$ is non-trivial and as $N_G(X)\le H_X$ and $F^*(H_X)$ is self-centralizing, the only possibility is that $O_2(H_X)=X$ and $H_X=N_G(X)$ is a maximal subgroup of $G$. Thus, $F^*(H_X)=F(H_X)\le O(X)O_2(X)\le N_G(S)\le N_G(A)\cap N_G(B)$. Then \cite[Theorem 19.1]{GLS2} yields that $F^*(N_G(A))=A$, $F^*(N_G(B))=B$, as desired. 
\end{proof}

We are now in a position to recognize $G$ given the structure of the maximal subgroups $N_G(A)$ and $N_G(B)$. We revert to the more general hypothesis that $O^{2'}(\fs)$ is isomorphic to the $2$-fusion system of $\PSL_3(q)$ or $\PSp_4(q)$. 

\begin{proposition}
$\fs$ is simple. 
\end{proposition}
\begin{proof}
By the Frattini argument, we have that $G=\langle N_G(S), O^{2'}(N_G(A)), O^{2'}(N_G(B))\rangle$. Since $G$ is simple and $N_G(S)$ normalizes $O^{2'}(N_G(A))$ and $O^{2'}(N_G(B))$, we deduce that $G=\langle O^{2'}(N_G(A)), O^{2'}(N_G(B))\rangle$. It quickly follows from \cite[Lemma I.13.1]{ako} that $\fs$ is simple.
\end{proof}

\begin{proposition}
$G$ has BN-pair of rank $2$.
\end{proposition}
\begin{proof}
Let $H$ be isomorphic to $\PSL_3(q)$ or $\PSp_4(2^a)$ chosen such that $\fs_S(G)\cong \fs_T(H)$ for $T\in\syl_2(H)$. Since $A=F^*(N_G(A))$ and $B=F^*(N_G(B))$, we have that $N_G(A)$ is a model for $N_{\fs}(A)$ and $N_G(B)$ is a model for $N_{\fs}(B)$. An application of the model theorem implies that for $\mathcal{A}(T)=\{C, D\}$, we have that $N_G(A)\cong N_H(C)$, $N_G(B)\cong N_H(D)$ and $N_G(A)\cap N_G(B)=N_G(S)\cong N_H(T)=N_H(C)\cap N_H(D)$. The BN-pair for $H$ is completely determined by the groups $N_H(C)$, $N_H(D)$ and $N_H(T)$ and mapping this information back across the relevant isomorphisms delivers a BN-pair for $G$, as desired.
\end{proof}

\begin{proof}[Proof of Theorem GG]
The BN-pair for $G$ yields a finite Moufang polygon on which $G$ acts faithfully by \cite[{}(33.6)]{titsweiss}. By results of Tits and Weiss \cite[Chapter 17]{titsweiss}, the finite Moufang polygons and their automorphism groups are known. Indeed, we see that $G$ is isomorphic to a rank $2$ finite simple group of Lie type in characteristic $2$. That $G$ is not a counterexample follows quickly upon comparing the structure of $S$ with that of a Sylow $2$-subgroup of a rank $2$ finite simple group of Lie type in characteristic $2$ (e.g. for $a>1$, $G$ either $\PSL_3(2^a)$ or $\PSp_4(2^a)$ and $T\in\syl_2(G)$, $G$ is uniquely determined among rank $2$ finite simple groups of Lie type by satisfying $|T|=|Z(T)|^3$ or $|T|=|Z(T)|^2$ respectively).
\end{proof}

\printbibliography

\end{document}